\documentclass[11pt, leqno]{amsart}
\usepackage[utf8]{inputenc}
\usepackage{geometry}
\usepackage{etoolbox}
\usepackage{lipsum}
\usepackage[english]{babel}
\usepackage{csquotes}
\usepackage{mathrsfs}
\usepackage{mathtools}
\usepackage{xcolor}
\usepackage{yfonts}
\usepackage{thmtools}
\usepackage[toc,page]{appendix}

\usepackage[breaklinks=true]{hyperref}
\usepackage{cleveref}

\newcommand{\supp}[0]{\mathrm{supp}}
\newcommand{\conv}[0]{\mathrm{conv}}
\newcommand{\height}[0]{\mathrm{ht}}
\newcommand{\wt}[0]{\mathrm{wt}}
\newcommand{\ad}[0]{\mathrm{ad}}
\newcommand{\gr}[0]{\mathrm{gr}}
\usepackage{amsmath, amsthm, amssymb}

\theoremstyle{plain}
\newtheorem{theorem}{Theorem}[section]
\newtheorem{lemma}[theorem]{Lemma}
\newtheorem{prop}[theorem]{Proposition}
\newtheorem{cor}[theorem]{Corollary}
\newtheorem{thmx}{Theorem}[section]

\theoremstyle{definition}
\newtheorem{definition}[theorem]{Definition}

\newtheorem{question}{Question}
\newtheorem*{p-psp}{parabolic-PSP}
\newtheorem{mxlp}[theorem]{Maximal property}
\newtheorem{remark}[theorem]{Remark}
\newtheorem{observation}[theorem]{Observation}

\numberwithin{equation}{section}

\newtheoremstyle{problem}{5pt}{5pt}{}{}{\normalfont}{\textbf{:}}{.5em}{}
\theoremstyle{problem}
\newtheorem*{problem}{\textbf{Problem}}
\newtheorem{note}{\textbf{Note}}
\newtheorem*{fact}{\textbf{Fact}}

\begin{document}
\title{Moving between weights of weight modules}
\author{G Krishna Teja}
\date{\today}
\begin{abstract} In Lie theory the partial sum property says that for a root system in any Kac--Moody algebra, every positive root is an ordered sum of simple roots whose partial sums are all roots. In this paper, we present two generalizations of this property:\\
1) ``\textit{Parabolic}'' generalization: If $I$ is any nonempty subset of simple roots, then every root with positive $I$-height is an ordered sum of roots of $I$-height 1, whose partial sums are all roots. In fact, we show this on the Lie algebra level, by showing that every root space is spanned by the Lie words formed from root vectors of $I$-height 1. As an application, we provide a ``minimal'' description for the set of weights of every (non-integrable) simple highest weight module over any Kac--Moody algebra. This seems to be novel even in finite type.\\
2) Generalization to the set of weights of weight modules: The partial sum property gives a chain of roots between 0 (fixed) and any positive root. We generalize this phenomenon to the set of weights of weight modules to get a chain of weights between any two comparable weights. This was shown by S. Kumar, for any finite-dimensional simple module over a semisimple Lie algebra. In this paper, we extend this result to (i) a large class of highest weight modules over a general Kac--Moody algebra $\mathfrak{g}$, which includes all the simple highest weight modules over $\mathfrak{g}$; (ii) more generally, for non-highest weight modules such as $\mathfrak{g}$ itself (adjoint representation) and arbitrary submodules of parabolic Verma modules over $\mathfrak{g}$; (iii) arbitrary integrable modules (not necessarily highest weight) over semisimple $\mathfrak{g}$.\\
Khare and Dhillon studied some cases in which the sets of weights of parabolic Verma modules give the sets of weights of all the highest weight modules (over Kac--Moody $\mathfrak{g}$) with specified highest weight and integrability. Motivated by this, in this paper we find all the highest weight modules which have their sets of weights same as those of parabolic Verma modules. We also provide a Minkowski difference formula for weights of arbitrary highest weight modules over $\mathfrak{g}$, extending the known results for simple highest weight modules. 
\end{abstract}
\maketitle
\tableofcontents
\allowdisplaybreaks
\section{Introduction}
In the root system $\Delta$ of a Kac--Moody or Borcherds Kac--Moody algebra, one of the standard results one studies is the \textit{partial sum property} (PSP): every positive root is an ordered sum of simple roots such that each partial sum is also a root.\\ In this paper, we present two generalizations of this property:
\begin{itemize}
\item[A)] Structure theory: We call this property as \textit{parabolic partial sum  property} (\textit{parabolic-PSP}). This strengthens the usual PSP in two ways:
\begin{itemize} 
	\item[1)] We show that the parabolic-PSP holds over any Kac--Moody/ Borcherds Kac--Moody/ more general Lie algebras $\mathcal{G}$ graded over free abelian semigroups.
	\item[2)] Moreover, the parabolic-PSP holds at the level of Lie words in $\mathcal{G}$, not just on the level of roots/grades.
\end{itemize}	
\item[B)] Representation theory: We show that an analogous version of the PSP, see Question \ref{Q2} of Khare below, holds in the set of weights of (i) the adjoint representation, (ii) all the simple highest weight modules, and more generally (iii) all the submodules of parabolic Verma modules over any Kac--Moody algebra $\mathfrak{g}$.
\end{itemize}   
 We prove the first generalization (parabolic-PSP) to a ``best possible extent'' and level of generality. In this we were motivated by the following problem.
\begin{problem}
Let $\mathfrak{g}$ be a Kac--Moody algebra, with Cartan subalgebra $\mathfrak{h}$. Suppose $L(\lambda)$ denotes the simple highest weight $\mathfrak{g}$-module, with highest weight $\lambda \in \mathfrak{h}^*$. Find a ``minimal'' description of the set of weights of $L(\lambda)$. (Recall that the weights of non-integrable simple highest weight module $L(\lambda)$, were computed in \cite{Dhillon_arXiv}.)
\end{problem}
In this paper, we use the parabolic-PSP to obtain such a minimal description. To our knowledge, this result is novel even in finite type.

For the rest of this section, we assume that $\mathfrak{g}$ is a Kac--Moody algebra over $\mathbb{C}$. Observe that the partial sum property gives a chain of roots between 0 (fixed) and a positive root, or equivalently the weight 0 and any weight in $\wt \mathfrak{g}=\Delta\sqcup\{0\}$ for the adjoint representation. Our second generalization of the PSP involves ``moving between weights'' of arbitrary representations to get a chain of weights between two comparable weights, and is natural in view of the previous line. This property was proved for all finite-dimensional simple highest weight modules and parabolic Verma modules  over semisimple $\mathfrak{g}$ by S. Kumar and A. Khare respectively, and to our knowledge there has not been further progress towards this problem beyond the semisimple case in the literature. In this paper, we solve this problem for: (i) a large class of highest weight modules over Kac--Moody $\mathfrak{g}$, including all simple highest weight modules over $\mathfrak{g}$; (ii) more generally, for non-highest weight modules such as $\mathfrak{g}$ itself (adjoint representation) and arbitrary submodules of parabolic Verma modules over $\mathfrak{g}$; (iii) arbitrary integrable modules (not necessarily highest weight) over semisimple $\mathfrak{g}$. Additionally, we also prove/discuss in various remarks the further generalizations and limitations of the above two generalizations of the PSP.

 As one application of the above two generalizations of the PSP, we provide a different proof, than the ones in \cite{Dhillon_arXiv, Khare_Ad}, of the extremal rays to the convex hull of every highest weight module. See Proposition \ref{P2.11} below. These two generalizations of the PSP also have additional applications to representation theory and combinatorics. In a forthcoming paper, we make use of these two generalizations and obtain interesting results about certain combinatorial subsets of the sets of weights of arbitrary highest weights modules over $\mathfrak{g}$ studied in the literature. Namely, weak faces and weak-$\mathbb{A}$-faces (for arbitrary additive subgroups $\mathbb{A}$ of $\mathbb{R}$ under addition). These subsets were introduced and studied by Chari and her co-authors in \cite{Chari_contm, Chari_Adv, Chari_JGeom, Chari_JPAA, Khare_JA, Khare_AR}.
 
 In this paper, we also find all the highest weight modules having their sets of weights equal to those of parabolic Verma modules. In this, we are motivated by the works of Khare and Dhillon in \cite{Dhillon_arXiv, Khare_Ad, Khare_JA, Khare_Trans}, in which some cases where the sets of weights of parabolic Verma modules give the sets of weights of almost all (and in few cases, all) highest weight modules were studied. For instance, see Theorem \ref{T2.3} (a) proved by Dhillon and Khare below. To our knowledge, this problem has not been fully addressed, except for these cases. 
 
 Interestingly, our analysis in this paper has been fruitful in additionally enabling us to obtain a Minkowski difference formula for the sets of weights of arbitrary highest weight modules over Kac--Moody $\mathfrak{g}$. Our inspiration in this are the Minkowski difference formulas for the weights of simple highest weight modules over Kac--Moody algebras in \cite{Dhillon_arXiv, Khare_Ad, Khare_JA, Khare_Trans, Khare_AR}. 
\section{Preliminaries and main results}
In order to state and prove the results of this paper, we need the following notation.
\subsection{Notation}\label{S2.1}
 In this paper, we denote the set of non-negative, positive, non-positive and negative real numbers by $\mathbb{R}_{\geq0}, \mathbb{R}_{>0}$, $\mathbb{R}_{\leq 0}$ and $\mathbb{R}_{<0}$ respectively. Similarly, for any $S \subseteq \mathbb{R}$ we define $S_*:=S\cap \mathbb{R}_*$ for $*$ any of $\geq 0,>0,\leq0$ and $<0$. For $n\in\mathbb{N}$, we denote the set $\{1,\ldots,n\}$ by $[n]$. For $B$ a subset of an $\mathbb{R}$-vector space, we define
 \[
 \conv_{\mathbb{\mathbb{R}}}B:=\Big\{\sum\limits_{j=1}^{n}c_j b_j \text{ }\big|\text{ } n\in\mathbb{N}, b_j \in B, c_j \in \mathbb{R}_{\geq 0}\text{ } \forall\text{ } 1\leq j\leq n \text{ and } \sum\limits_{j=1}^{n}c_j=1\Big\}
 \]
 to be the convex hull of $B$ over $\mathbb{R}$. Similarly, we define $\conv_{F}B$ over any subfield $F$ of $\mathbb{R}$ (with $F_{\geq 0}$ in the place of $\mathbb{R}_{\geq 0}$ in the above equation). For any two subsets $C$ and $D$ of a real or complex vector space, $C\pm D:=\{c\pm d$ $|$ $c \in C ,d\in D\}$ denotes the Minkowski sum of $C$ and $\pm D$ respectively. When $C=\{x\}$ is singleton, we denote $C\pm D$ by $x\pm D$ for simplicity. \\
 Throughout the paper, we denote by $\mathcal{I}$ an indexing set and by $\Pi=\{\alpha_i$ $|$ $i\in\mathcal{I}\}$ the free generating set for the semigroup $
 \mathbb{Z}_{\geq0}\Pi$. For $I\subset\mathcal{I}$, we denote $\mathcal{I}\setminus  I$ by $I^c$.
 
 For every Lie algebra $\mathcal{G}$ in this paper, we denote by [.,.] its Lie bracket and by $U(\mathcal{G})$ its universal enveloping algebra. Let $\hat{\xi}:=(\xi_{i})_{i=1}^{n}=(\xi_1,\ldots,\xi_n)$, $n\in \mathbb{N}$, be an ordered (formal) sequence. Let $x_{\xi_1},\ldots,x_{\xi_{n}}\in \mathcal{G}$ be a sequence of vectors indexed by $\xi_i$, $1\leq i\leq n$. We define
 \begin{equation}\label{E2.1} [[x_{\xi_i}]]_{i=1}^{n}=[[x_{\xi}]]_{\xi\in\hat{\xi}}:=\Big[x_{\xi_1},\big[\cdots ,[x_{\xi_{n-1}},x_{\xi_n}]\cdots\big]\Big]
 \end{equation}
 to be the right normed Lie word on $x_{\xi_1},\ldots,x_{\xi_n}\in\mathcal{G}$--i.e. the iterated Lie bracket of $x_{\xi_1},\ldots, x_{\xi_n}\in \mathcal{G}$ in the same order. In this notation, when $n=1$, we define $[[x_{\xi_i}]]_{i=1}^{1}=[[x_{\xi}]]_{\xi\in\{\xi_1\}}:= x_{\xi_1}$.\newline \newline
\textbf{Notation for the Kac--Moody setting.} Let $\mathfrak{g}=\mathfrak{g}(A)$ denote the Kac--Moody algebra over $\mathbb{C}$ corresponding to a generalized Cartan matrix $A$ with the realisation $(\mathfrak{h},\Pi ,\Pi^{\vee})$, triangular decomposition $\mathfrak{n}^+\oplus\mathfrak{h}\oplus\mathfrak{n}^-$, and the root system $\Delta$. Whenever we make additional assumptions, such as $\mathfrak{g}$ being symmetrizable or semisimple or of finite/affine type, we will clearly mention it. Let $\Pi =\{\alpha_i $ $|$ $i\in \mathcal{I} \}$ be the simple system and $\Pi^{\vee}=\{\alpha_i^{\vee}$ $|$ $i\in\mathcal{I}\}$ be the simple co-root system, where $\mathcal{I}$ is a  fixed indexing set for the simple roots. $\mathcal{I}$ also stands for the set of vertices/nodes in the Dynkin diagram for $A$ or $\mathfrak{g}$. Throughout the paper, unless otherwise stated, assume that $\mathcal{I}$ is finite. Let $e_i ,f_i,\alpha_i^{\vee}$, $\forall$ $ i\in\mathcal{I}$ be the Chevalley generators for $\mathfrak{g}$, and $W$ denote the Weyl group of $\mathfrak{g}$ generated by simple reflections $\{s_i $ $|$ $i\in \mathcal{I}\}$. Let $\mathfrak{g}':=[\mathfrak{g},\mathfrak{g}]$ be the derived subalgebra of $\mathfrak{g}$, which is generated by $e_i,f_i,\alpha_i^{\vee}$, $\forall$ $i\in\mathcal{I}$. When $\mathfrak{g}$ is symmetrizable, we fix a standard non-degenerate symmetric invariant bilinear form on $\mathfrak{h}^*$ and denote it by (.,.).

For $\emptyset \neq I\subseteq \mathcal{I}$, we define $\Pi_{I}:=\{\alpha_i$ $|$ $i\in I\}$ and $\Pi_I^{\vee}:=\{\alpha_i^{\vee}$ $|$ $i\in I\}$. We define $\mathfrak{g}_I:=\mathfrak{g}(A_{I\times I})$ to be the Kac--Moody algebra corresponding to the submatrix $A_{I\times I}$ of $A$ with realisation $(\mathfrak{h}_I,\Pi_I,\Pi_I^{\vee})$, where $\mathfrak{h}_I \subset \mathfrak{h}$, and the Chevalley generators $e_i ,f_i ,\alpha^{\vee}_i$, $\forall$ $i\in I$. By \cite[Exercise 1.2]{Kac}, $\mathfrak{g}_I$ can be thought of as a subalgebra of $\mathfrak{g}$, and the subroot system $\Delta_{I}:=\Delta\cap\mathbb{Z}\Pi_I\subset\Delta$ coincides with the root system of $\mathfrak{g}_I$. Let $\mathfrak{l}_I= \mathfrak{g}_{I}+\mathfrak{h}$ and $\mathfrak{p}_I=\mathfrak{g}_I +\mathfrak{h}+\mathfrak{n}^+$ be the standard Levi and the parabolic Lie subalgebras of $\mathfrak{g}$ corresponding to $I$, respectively. Let $W_I$ denote the parabolic subgroup of $W$ generated by the simple reflections $\{s_i $ $|$ $i \in I\}$. When $I=\emptyset$, for completeness we define (i) $\Pi_I,\Pi_I^{\vee}$ and $\Delta_{I}$ to be $\emptyset$, (ii) $\mathfrak{g}_I$ and $\mathfrak{h}_I$ to be $\{0\}$, and (iii) $W_I$ to be the trivial subgroup $\{e\}$ of $W$. For $\alpha\in\Pi$, we define $\mathfrak{sl}_{\alpha}:=\mathfrak{g}_{-\alpha}\oplus \mathbb{C}\alpha^{\vee}\oplus\mathfrak{g}_{\alpha}\simeq \mathfrak{sl}_2(\mathbb{C})$. We occasionally use $s_{\alpha}$ for the simple reflection about the hyperplane perpendicular to $\alpha$.

Let $\prec$ be the usual partial order on $\mathfrak{h}^*$, under which $x\prec y \in \mathfrak{h}^* \iff y -x \in \mathbb{Z}_{\geq 0}\Pi$. Fix $\emptyset\neq I\subset\mathcal{I}$, $J\subset\mathcal{I}$, $\alpha\in\Delta^+$, and a vector $x=\sum\limits_{i\in \mathcal{I}}c_i \alpha_i \in \mathbb{C}\Pi$ for some $c_i \in \mathbb{C}$. We define 
\begin{equation}\label{E2.2}
\begin{split}
&\supp(x):=\{i \in \mathcal{I} \text{ }|\text{ } c_i \neq 0\},\\
&\supp_I(x):=\{i\in I \text{ }|\text{ } c_i \neq 0\},\\
&\Delta_{\alpha ,J}:=\{\beta \in \Delta^+ \text{ }|\text{ } \supp(\beta -\alpha)\subset J\},\\
&\height(x):=\sum_{i\in \mathcal{I}}c_i,\quad \height_{I}(x):=\sum_{i\in I}c_i,
\end{split}
\qquad\qquad
\begin{split}
&\text{for }n\in\mathbb{Z}\text{ }\text{ }\Delta_{I,n}:=\{\beta \in \Delta \text{ }|\text{ } \height_I(\beta)=n\},\\
&\mathfrak{g}_{I,n}:=\bigoplus_{\beta\in\Delta_{I,n}}\mathfrak{g}_{\beta}, \text{ and }\text{ for convenience}\\
&\supp(0):=\emptyset,\text{ }\Delta_{\alpha,\emptyset}:=\{\alpha\},\Delta_{\emptyset,n }:=\emptyset,\\
&\mathbb{Z}\Delta_{\emptyset,n}=\mathbb{R}\Delta_{\emptyset,n}:=\{0\}.\\
\end{split}
\end{equation}
In the notation as in equation \eqref{E2.2}, note that
\begin{equation}\label{E2.3}
\begin{split} \Delta_{I,0}=\Delta_{I^c}=\Delta\cap\mathbb{Z}\Pi_{I^c}
\end{split},\quad
\begin{split}  \alpha\in\Delta_{\alpha,J}\text{ } \forall \text{ }J\subset\mathcal{I},
\end{split}
\quad
\begin{split}
\Delta_{\alpha_i ,\{i\}^c}= \Delta_{\{i\},1}\text{ }\forall\text{ } i\in\mathcal{I},
\end{split}\quad
\begin{split}
\mathfrak{g}_{I,0}=\mathfrak{g}_{I^c}.    
\end{split}
\end{equation}

Given an $\mathfrak{h}$-module $M$ and
$\mu\in\mathfrak{h}^*$, denote the $\mu$-weight space and the set of
weights of $M$ by
\[ 
M_{\mu}=\{v\in M\text{ }|\text{ }h\cdot v= \mu(h)v\text{ }\forall\text{ }h\in \mathfrak{h}\}\quad\text{and}\quad \wt M=\{\mu\in\mathfrak{h}^*\text{ }|\text{ }M_{\mu}\neq\{0\}\}.
\]
We say that $M$ is a weight module if $M=\bigoplus\limits_{\mu\in\wt M}M_{\mu}$. When each weight space of a weight module $M$ is finite-dimensional, we define $char M:=\sum_{\mu\in\wt M}\dim(M_{\mu})e^{\mu}$ to be the formal character of $M$.\\
For $\lambda \in \mathfrak{h}^*$, let $M(\lambda)$ and $L(\lambda)$ denote the Verma module over $\mathfrak{g}$ with highest weight $\lambda$ and its unique simple quotient respectively. By $M(\lambda)\twoheadrightarrow V$ ($M(\lambda)$ surjecting onto $V$), we denote a non-trivial highest weight $\mathfrak{g}$-module $V$ with highest weight $\lambda$.

 For $h\in\mathfrak{h}$ and $\mu\in\mathfrak{h}^*$, we define $\langle \mu ,h \rangle$ to be the evaluation of $\mu$ at $h$, which we also denote by $\mu(h)$ occasionally. We define $P^+:=\{\mu\in\mathfrak{h}^*$ $|$ $\langle\mu,\alpha_i^{\vee}\rangle\in\mathbb{Z}_{\geq0}$ $\forall$ $i\in\mathcal{I}\}$ to be the set of dominant integral weights. For $\lambda\in\mathfrak{h}^*$, $M(\lambda)\twoheadrightarrow V$ and $I\subset\mathcal{I}$, we define
\begin{equation}\label{E2.4}
 J_{\lambda}:=\{i \in \mathcal{I} \text{ }|\text{ } \langle \lambda ,\alpha^{\vee}_i\rangle \in \mathbb{Z}_{\geq0}\}\quad
 J_{\lambda}':=\{i \in \mathcal{I} \text{ }|\text{ } \langle \lambda ,\alpha_{i}^{\vee}\rangle\in \mathbb{R}_{\geq 0}\}\quad
 \wt_I V:=\wt V\cap (\lambda-\mathbb{Z}_{\geq0}\Pi_I).
 \end{equation}
  Fix $I\subset\mathcal{I}$, and suppose $V$ is a highest weight $\mathfrak{p}_I$ or $\mathfrak{l}_I$-module with highest weight $\lambda\in\mathfrak{h}^*$. Then $V$ becomes a highest weight $\mathfrak{g}_I$-module with highest weight $\lambda\big|_{\mathfrak{h}_I}$ the restriction of $\lambda$ to $\mathfrak{h}_I$. But for simplicity, throughout the paper we also denote the highest weight of $V$ for the $\mathfrak{g}_I$-action by $\lambda$.
  
For $\mu\in\mathfrak{h}^*$, $\alpha\in\Pi$ and $k\in\mathbb{Z}_{\geq0}$, we define   \begin{equation}\label{E2.5}
  \big[\mu-k\alpha,\text{ }\mu\big]:=\{\mu-j\alpha\text{ }\big|\text{ }0\leq j\leq k\}.\quad \text{ (Not to be confused with the Lie bracket.)}
  \end{equation}\newline
   \textbf{Notation for the graded setting.}
 In section \ref{S3}, we fix an arbitrary field $\mathbb{F}$, an indexing set $\mathcal{I}$, and the abelian semigroup $\mathbb{Z}_{\geq0}\Pi$ freely generated by $\Pi=\{\alpha_i$ $|$ $i\in \mathcal{I}\}$. We work with a more general $\mathbb{Z}_{\geq0}\Pi$-graded $\mathbb{F}$-Lie algebra $\mathcal{G}$ (that is, $\mathcal{G}=\bigoplus\limits_{\gamma\in\mathbb{Z}_{\geq0}\Pi}\mathcal{G}_{\gamma}$, and $[\mathcal{G}_{\alpha},\mathcal{G}_{\beta}]\subset\mathcal{G}_{\alpha+\beta}$ $\forall$ $\alpha,\beta\in\mathbb{Z}_{\geq0}\Pi$), which is generated by its subspaces $\mathcal{G}_{\alpha_i}$, $\forall$ $i\in \mathcal{I}$. We assume that each subspace $\mathcal{G}_{\alpha_i}\subset\mathcal{G}$ is non-zero, $\forall$ $i\in\mathcal{I}$. We define $\mathcal{A}:=\{\alpha \in \mathbb{Z}_{\geq0}\Pi$ $|$ $\mathcal{G}_{\alpha}\neq\{0\}\}$ (the analogous candidate for the set of positive roots $\Delta^+$ of the root system $\Delta$). We do not assume the subspaces $\mathcal{G}_{\alpha}$, $\alpha\in\mathcal{A}$, (in particular $\mathcal{G}_{\alpha_i}$, $i\in\mathcal{I}$) to be finite-dimensional. We define the partial order $\prec$, and functions $\supp(.),\height(.)$, $\supp_I(.)$ and $\height_{I}(.)$, for $\emptyset\neq I\subset\mathcal{I}$, analogously on $\mathbb{Z}\Pi$. For $\emptyset\neq I\subset\mathcal{I}$, and $n\in \mathbb{Z}$, we analogously define 
 \begin{equation}\label{E2.6}
 \mathcal{A}_{I,n}:=\{\beta \in \mathcal{A}\text{ } |\text{ }\height_{I}(\beta)=n\},\qquad\mathcal{G}_{I,n}:=\bigoplus_{\beta\in\mathcal{A}_{I,n}}\mathcal{G}_{\beta}.
 \end{equation}
 \begin{note}\label{N1}
Throughout the paper, for $\lambda_1,\ldots,\lambda_n\in\mathbb{R}_{\geq0}\Pi$, $n\geq 2$, and $U\subset\mathbb{R}_{\geq0}\Pi$ (or respectively $\mathfrak{h}^*$ in the place of $\mathbb{R}_{\geq0}\Pi$ for the Kac--Moody setting), the notation
\[
\lambda_1\prec\cdots\prec\lambda_i\prec\cdots\prec\lambda_n\in\text{ }U\text{ }\forall i
\]
denotes: (1) $\lambda_{i-1}\prec\lambda_{i}$ $\forall$ $i>1$, and (2) $\lambda_i\in U$ $\forall$ $i\in [n]$.
\end{note}
\subsection{Preliminaries and useful results}
We begin by noting the following basic facts about Lie algebras and highest weight modules over Kac--Moody algebras, which we use without mention. Throughout this subsection, $\mathfrak{g}$ stands for a Kac--Moody algebra.
\begin{itemize}
	\item[(F1)]\label{(F1)} Let $L$ be a Lie algebra generated by $X\subset L$. Then $L$ is spanned by the right normed Lie words on the elements in $X$. More precisely, there exists a basis of $L$ consisting of elements of the form $\Big[x_1,\big[\cdots, [x_{n-1},x_n]\cdots\big]\Big]$ ($=[[x_i]]_{i=1}^{n}$ in the notation as in equation \eqref{E2.1}), $n\in\mathbb{N}$, such that $x_1,\ldots,x_n\in X$. 
	\item[(F2)]\label{(F2)} Let $\mathfrak{g}$ be a Kac--Moody algebra, $\lambda\in\mathfrak{h}^*$, $M(\lambda)\twoheadrightarrow V$ and $0\neq v\in V_{\lambda}$ be a highest weight vector. Then for $\mu\in\wt V$ the weight space $V_{\mu}$ is spanned by the weight vectors of the form $f_{i_1}\cdots f_{i_n}v$ such that $f_{i_j}\in\mathfrak{g}_{-\alpha_{i_j}}$, $i_j\in\mathcal{I}$, and $\alpha_{i_j}\in\Pi$ for each $j\in [n]$, and $\sum\limits_{j=1}^{n}\alpha_{i_j}=\lambda-\mu$.
	\item[(F3)] With the notation as in (F2), suppose $\mu\in\wt V$ and $\alpha_i\in\Pi$ such that $\langle\mu,\alpha_i^{\vee}\rangle>0$. Then $\big[\mu-\lceil\langle\mu,\alpha_i^{\vee}\rangle\rceil\alpha_i,\text{ }\mu\big]\subset\wt V$ by the $\mathfrak{sl}_{\alpha_i}$-action on $V_{\mu}$ and $\mathfrak{sl}_2$-theory, where $\lceil \cdot\rceil$ denotes the ceiling function. Recall, $s_i(\mu)=\mu-\langle\mu,\alpha_i^{\vee}\rangle\alpha_i$ when $\langle\mu,\alpha_i^{\vee}\rangle\in\mathbb{Z}_{\geq 0}$.
\end{itemize}  

 Most of the results of this paper hold true for a large class of highest weight modules -- namely, the class of highest weight modules whose set of weights coincides with that of a parabolic Verma module. This class contains all Verma, parabolic Verma, simple highest weight modules and many more; by Theorem \ref{T2.3} below, proved by Dhillon and Khare \cite{Dhillon_arXiv}. Parabolic Verma modules, also known as generalized Verma modules, were introduced and studied by Lepowsky in a series of papers; see \cite{L_JA} and the references therein. For a detailed list of the properties of parabolic Verma modules, we refer the reader to \cite[Chapter 9]{Hump_BGG}. 
\begin{definition}\label{D2.1}
\begin{itemize}
    \item[(1)] Given a generalized Cartan matrix $A$, let $\bar{\mathfrak{g}}(A)$ denote the Lie algebra generated by $e_i, f_i, \mathfrak{h}$ modulo only the Serre relations, and let the Kac--Moody Lie algebra $\mathfrak{g}=\mathfrak{g}(A)$ be the further quotient of $\overline{\mathfrak{g}}(A)$ by the largest ideal intersecting $\mathfrak{h}$ trivially.
	\item[(2)] Let $\lambda\in\mathfrak{h}^*$, $\emptyset\neq J\subset \mathcal{I}$ and $M(\lambda)$ denote the Verma module over $\mathfrak{g}$ with highest weight $\lambda$. Throughout the paper, let $0\neq m_{\lambda}\in M(\lambda)_{\lambda}$ denote a highest weight vector of $M(\lambda)$. Recall that the parabolic Lie subalgebra of $\mathfrak{g}$ corresponding to $J$ is denoted by $\mathfrak{p}_J$. Assume $J\subset J_{\lambda}$ for the rest of this point. Let $L^{\max}_J(\lambda)$ denote the largest integrable highest weight module over $\mathfrak{g}_J$ (or equivalently $\mathfrak{p}_J$, via the natural action of $\mathfrak{h}$ and the trivial action of $\bigoplus\limits_{\alpha\in\Delta^+\setminus\Delta^+_J} \mathfrak{n}^+_{\alpha}$) with highest weight $\lambda$. $L^{\max}_J(\lambda)$ has the property that $L^{\max}_J(\lambda)\twoheadrightarrow L'_J(\lambda)$ for any integrable highest weight $\mathfrak{g}_J$-module $L'_J(\lambda)$ with highest weight $\lambda$. Let $L_J(\lambda)$ denote the simple highest weight module over $\mathfrak{g}_J$ (or equivalently $\mathfrak{p}_J$) with highest weight $\lambda$. Note that $L_J(\lambda)$ is integrable over $\mathfrak{g}_J$. Recall that when $\mathfrak{g}_J$ is symmetrizable, by \cite[Corollary 10.4]{Kac} $L^{\max}_J(\lambda)$ is simple. 

	\item[(3)] For $J\subset J_{\lambda}$, denote the
	parabolic Verma module corresponding to $J$ (and $\lambda$) by
		\[
		M(\lambda,J):=U(\mathfrak{g})\otimes_{U(\mathfrak{p}_J)}L^{\max}_J(\lambda)
		\simeq M(\lambda)\big/\big(\sum_{j\in
		J}U(\mathfrak{n}^-)f_{j}^{\langle\lambda,\alpha_j^{\vee}\rangle+1}m_{\lambda}\big),\]
		where $L^{\max}_J(\lambda)$ is the
		parabolic Verma module over $\mathfrak{g}_J$
		corresponding to $J$ (and $\lambda$).
		\end{itemize} 
\end{definition}
 Given $J\subset \mathcal{I}$, recall that a weight module $M$ of $\mathfrak{g}$ is said to be $\mathfrak{g}_J$-integrable if $e_j$ and $f_j$ act locally nilpotently on $M$ $\forall$ $j\in J$. If $M$ is $\mathfrak{g}_J$-integrable, then $\wt M$ is $W_J$-invariant and every submodule of $M$ is also $\mathfrak{g}_J$-integrable.\\
Let $M(\lambda)\twoheadrightarrow V$ and $J\subset J_{\lambda}$. Then $V$ is $\mathfrak{g}_J$-integrable if $f_j$ acts locally nilpotently on $V$ $\forall$ $j\in J$. 
\begin{lemma}\label{L2.2}
\begin{itemize}
\item[(1)] Let $j\in \mathcal{I}$. Then $f_j$ acts locally nilpotently on $V$ if and only if $f_j$ acts nilpotently on the highest weight space $V_{\lambda}$.
\item[(2)] $V$ is $\mathfrak{g}_J$-integrable if and only if $char V$ is $W_J$-invariant.
\end{itemize}
\end{lemma}
We define $I_V:=\big\{i\in \mathcal{I}$ $|$ $\langle\lambda,\alpha_i^{\vee}\rangle\in\mathbb{Z}_{\geq0}$ and $f_i^{\langle\lambda,\alpha_i^{\vee}\rangle+1}V_{\lambda}=\{0\}\big\}$ to be the integrability of $V$. Note that $V$ is $\mathfrak{g}_{I_V}$-integrable.

In our results on the parabolic Verma module $M(\lambda,J)$ or precisely its weights, we make use of the $\mathfrak{g}_{J}$-integrability of $M(\lambda,J)$ and the following formulas for its set of weights. See e.g. \cite[Proposition 3.7 and Section 4]{Dhillon_arXiv} for the proofs and consequences of these formulas.
\begin{eqnarray}
\text{Minkowski decomposition:}&\quad&\wt M(\lambda,J)=\wt L_J(\lambda)-\mathbb{Z}_{\geq0}(\Delta^+\setminus\Delta^+_J).\label{E2.7}\\
\qquad\text{Integrable slice decomposition:}&\quad&\wt M(\lambda,J)= \bigsqcup\limits_{\xi\in\mathbb{Z}_{\geq0}\Pi_{J^c}}\wt L_J(\lambda-\xi).\label{E2.8}
\end{eqnarray}
The results of this paper on $\wt M(\lambda,J)$, respectively on $\conv_{\mathbb{R}}\wt M(\lambda,J)$ hold true for almost all, respectively all the highest weight modules by parts (a) and (b) of the following Theorem. See \cite[Theorem 2.9]{Dhillon_arXiv} and \cite[Theorem 3.13]{Khare_Ad} for the proofs.
\begin{theorem}[Dhillon--Khare]\label{T2.3}
Fix $\lambda\in\mathfrak{h}^*$ and $J \subset I_{L(\lambda)}$ the integrability of $L(\lambda)$. (In the notation of this paper, note that $I_{L(\lambda)}=J_{\lambda}$.)
\begin{itemize}
\item[(a)] Every highest weight module V of highest weight $\lambda$ and integrability $J$ has the same weights if and only if the Dynkin subdiagram on $I_{L(\lambda)}\setminus J$ is complete, i.e. $\langle\alpha_i,\alpha_j^{\vee}\rangle \neq0,$ $\forall$ $i, j \in I_{L(\lambda)}\setminus J$.\\
In particular,
\begin{equation}\label{E2.9}
    \wt L(\lambda)=\wt M(\lambda,I_{L(\lambda)}).
\end{equation}
\item[(b)] For $M(\lambda)\twoheadrightarrow V$, the following data are equivalent:
\begin{itemize}
\item[(1)] $I_V$, the integrability of $V$.
\item[(2)] $\conv_{\mathbb{R}}\wt V$, the convex hull of the set of weights of V.
\item[(3)] The stabilizer of $\conv_{\mathbb{R}}\wt V$ in $W$.
\end{itemize}
In particular, the convex hull in (2) is always that of the parabolic Verma module
$M(\lambda, I_V )$, and the stabilizer in (3) is always the parabolic subgroup $W_{I_V}$.
\end{itemize}
\end{theorem}
\begin{definition}\label{D2.4}
	Let $C$ be a convex subset of a real vector space, and fix a point $x\in C$ and a vector $y$. Then $x+\mathbb{R}_{\geq0}y$ denotes the ray at/through the point $x$ in the direction of $y$. We say $x+\mathbb{R}_{\geq0}y$ is an extremal ray of $C$ if whenever $x+ry=\sum\limits_{i=1}^{n}d_ix_i$ for some $r,d_i\in\mathbb{R}_{\geq0}$ such that $\sum\limits_{i=1}^{n}d_i=1$, and $x_i\in C$, $1\leq i\leq n$, we have $x_i\in x+\mathbb{R}_{\geq0}y$ whenever $d_i>0$.  
\end{definition}

\subsection{Main results}
 We now begin stating the main results of this paper.
Throughout this subsection, unless specified, assume $\mathfrak{g}$ to be
a Kac--Moody algebra over $\mathbb{C}$. Let $\mathcal{I}$
denote a fixed indexing set for the set of simple roots. We first state
the \textit{parabolic partial sum property} (\textit{parabolic-PSP)}.
\begin{definition}[\textbf{Parabolic-PSP}]
	Let $\Delta$ be the root system of a Kac--Moody algebra $\mathfrak{g}$. We say $\Delta$ has the \textit{parabolic partial sum property} if given $\emptyset\neq I\subset\mathcal{I}$ and a root $\beta\in\Delta^+$ with $\height_I(\beta)>1$, there exists a root $\gamma\in \Delta_{I,1}:=\{\alpha \in \Delta \text{ }|\text{ } \height_I(\alpha)=1\}$ such that $\beta -\gamma \in \Delta$; equivalently, there exists a sequence of roots $\gamma_i\in \Delta_{I,1}$, $1\leq i\leq n=\height_{I}(\beta)$, such that
	\begin{equation*}
		\gamma_1 \prec \cdots \prec \sum\limits_{j=1}^{i}\gamma_j\prec \cdots \prec \sum\limits_{j=1}^n \gamma_j=\beta\in\Delta\text{ }\forall i.\quad\text{ (See Note \ref{N1} for the notation.)}
	\end{equation*}
\end{definition}
\begin{question}[Khare]
	Does the parabolic-PSP hold in the root system of a Kac--Moody algebra?
\end{question}
The first main result of this paper positively answers this question, together with applications to the sets of weights of non-integrable highest weight modules over $\mathfrak{g}$. For $\lambda\in\mathfrak{h}^*$, recall from equation \eqref{E2.4} that $J_{\lambda}:=\{j\in \mathcal{I}$ $|$ $\langle \lambda ,\alpha_j^{\vee}\rangle \in \mathbb{Z}_{\geq 0}\}$. Also, recall from equation \eqref{E2.2} that when $I=\emptyset$, $\Delta_{I,1}:=\emptyset$ and $\mathbb{Z}_{\geq 0}\Delta_{I,1}=\mathbb{Z}_{\geq0}\Pi_I:=\{0\}$.

\begin{thmx}\label{thmA}
	Let $\mathfrak{g}$ be a Kac--Moody algebra with root system $\Delta$. Then:
	\begin{itemize}
	\item[(A1)] The parabolic-PSP holds true in $\Delta$ for any $\emptyset \neq I\subset \mathcal{I}$.
\item[(A2)] For $\lambda \in \mathfrak{h}^*$, $J \subset J_{\lambda}$ and for any $M(\lambda)\twoheadrightarrow V$ such that $\wt V=\wt M(\lambda,J)$, we have the following description:
\end{itemize}
	\begin{equation}\label{E2.10}
	\begin{aligned}[t] \mathbb{Z}_{\geq 0}(\Delta^+ \setminus \Delta_{J}^+)=\mathbb{Z}_{\geq 0}\Delta_{J^c,1},
	\end{aligned}
	 \qquad
	 \begin{aligned}[t]
	 &\wt V=\wt L_{J}(\lambda)-\mathbb{Z}_{\geq0}\Delta_{J^c,1} = \wt L_J(\lambda)+\mathbb{Z}_{\geq0}\Delta_{J^c,-1}.
	 \end{aligned}
	\end{equation}
	\[\text{In particular,}\quad\wt L(\lambda)=\wt L_{J_{\lambda}}(\lambda)-\mathbb{Z}_{\geq0}\Delta_{J_{\lambda}^c,1}= \wt L_{J_{\lambda}}(\lambda)+\mathbb{Z}_{\geq0}\Delta_{J_{\lambda}^c,-1}.
	\]
	The description of $\wt V$ in equation \eqref{E2.10} is minimal, in that $\Delta_{J^c,1}$ cannot be reduced.
   \end{thmx}     
\begin{remark}
To our knowledge, neither the parabolic-PSP, nor its applications in representation theory to $\wt V$---for example any simple non-integrable highest weight module $V=L(\lambda)$---was known even for $\mathfrak{g}$ of finite type. 
\end{remark}
In fact, we prove a more general version of the parabolic-PSP for any $\mathbb{Z}_{\geq0}\Pi$-graded $\mathbb{F}$-Lie algebra $\mathcal{G}$ generated by $\mathcal{G}_{\Pi}$. We also prove a ``going up'' version of the parabolic-PSP for Kac--Moody $\mathfrak{g}$, see Theorem \ref{T3.6} below. This also generalizes the well-known basic property \cite[Proposition 4.9]{Kac}.

Given the minimal description for $\wt M(\lambda,J)$ or in particular for $\wt L(\lambda)$ as above, and also in view of Theorem \ref{T2.3} by Dhillon and Khare, it is natural to seek to find all the highest weight modules $V$ such that $\wt V=\wt M(\lambda,J)$. The second main result of this paper, Theorem \ref{thmB}, completely solves this problem. For stating it, we first define the following submodules of $M(\lambda)$. 

For $\lambda\in\mathfrak{h}^*$ and $\emptyset\neq J\subset J_{\lambda}$, we define $N(\lambda,J)$ to be the largest proper submodule of $M(\lambda)$ with respect to the property: 
\begin{align*} 
&\text{ If }\mu\in\wt N(\lambda,J) \text{ is such that the Dynkin subdiagram on }\supp(\lambda-\mu) \text{ has no edge,}\tag{\textbf{P1}}\\&\text{ then }\supp(\lambda-\mu)\cap J\neq \emptyset.
\end{align*}
Similarly, we define $N(\lambda)$ to be the largest proper submodule of $M(\lambda)$ with respect to the property:
\begin{align*}\tag{\textbf{P2}} \text{If }\mu\in\wt N(\lambda),\text{ then the Dynkin subdiagram on }\supp(\lambda-\mu)\text{ has at least one edge.}\end{align*}
The existence of $N(\lambda, J)$ and $N(\lambda)$, and some useful observations about them, will be discussed in Section \ref{S4}. For a quick and brief understanding of $N(\lambda,J)$, check that the submodule $\sum\limits_{j\in J}U(\mathfrak{g})f_j^{\langle\lambda,\alpha_j^{\vee}\rangle+1}$ $m_{\lambda}\subset M(\lambda)$ satisfies the property (P1), and hence is contained in $N(\lambda,J)$. Similarly, check that the integrability of $\frac{M(\lambda)}{N(\lambda)}$ is empty. With these definitions, we now state our next main result.
\begin{thmx}\label{thmB}
Let $\lambda\in\mathfrak{h}^*$, $\emptyset \neq J\subset J_{\lambda}$,
and $0\rightarrow N_V\rightarrow M(\lambda)\rightarrow
V\rightarrow 0$ be an exact sequence. Then
\[\wt V=\wt M(\lambda,J)\quad\text{ if and only if }\quad \sum\limits_{j\in J}U(\mathfrak{g})f_j^{\langle\lambda,\alpha_j^{\vee}\rangle+1}m_{\lambda}\text{ }\subseteq\text{ }N_V\text{ }\subseteq \text{ }N(\lambda,J).
\]\end{thmx}
Recall that $M(\lambda,J)\simeq M(\lambda)\big/ \big(\sum_{j\in J}
U(\mathfrak{g})f_j^{\langle\lambda,\alpha_j^{\vee}\rangle+1}m_{\lambda}\big)$.
The key result used in the proof of Theorem \ref{thmB} is the following
proposition, which gives all the highest weight modules $V$
such that $\wt V=\wt M(\lambda)$. Theorem \ref{thmB} and Proposition
\ref{P2.7} seem to be novel even in finite type.
\begin{prop}\label{P2.7}
Let $\lambda\in\mathfrak{h}^*$, and let $0\rightarrow
N_V\rightarrow M(\lambda)\rightarrow V\rightarrow 0$ be an exact
sequence. Then \[\wt V=\wt M(\lambda)\quad\text{ if and only if }\quad
N_V\text{ }\subseteq \text{ }N(\lambda).\] \end{prop}
 
 Our third main result, Theorem \ref{thmC}, extends the Minkowski difference formulas for the weights of simple highest weight modules in \cite{Dhillon_arXiv, Khare_Ad, Khare_JA, Khare_Trans, Khare_AR} to the weights of any highest weight module over a general Kac--Moody algebra $\mathfrak{g}$. For $I\subset\mathcal{I}$, recall from equation \eqref{E2.4} that $\wt_{I}V:=\wt V\cap (\lambda-\mathbb{Z}_{\geq 0}\Pi_{I})$, and note that this set is precisely the set of weights of the $\mathfrak{g}_I$-module $U(\mathfrak{g}_I)V_{\lambda}$.
\begin{thmx}\label{thmC}
Let $\lambda\in\mathfrak{h}^*$, and $M(\lambda)\twoheadrightarrow{ }V$. Then
\begin{equation}\label{E2.11}
\wt V= \wt _{J_{\lambda}}V-\mathbb{Z}_{\geq 0}\Delta_{J_{\lambda}^c,1}=\wt _{J_{\lambda}}V-\mathbb{Z}_{\geq 0}(\Delta^+\setminus \Delta_{J_{\lambda}}^+).
\end{equation}
More strongly, we have for arbitrary $J\subset\mathcal{I}$
  \begin{equation}\label{E2.12}
  \wt_JV-\mathbb{Z}_{\geq 0}\Pi_{J^c}\subset \wt V\quad \iff \quad \wt V=\wt_JV-\mathbb{Z}_{\geq 0}\Delta_{J^c,1}.
  \end{equation}
\end{thmx}
\begin{remark}
Notice that the second equation in Theorem \ref{thmC} improves on the first when more information is available on the structure of $V$. For instance, it applies when $V$ is a parabolic Verma module $M(\lambda,J')$ or a simple module $L(\lambda)$, with $J = J', J_\lambda$ respectively.
\end{remark}
\begin{remark}
In view of Theorem \ref{thmC}, to understand the sets of weights of highest weight module $M(\lambda)\twoheadrightarrow{ }V$ over Kac--Moody $\mathfrak{g}$, it suffices to work with dominant integral highest weights $\lambda$. 
\end{remark}
We now state the fourth main result of this paper, Theorem \ref{thmD}, which is the second generalization of the PSP (in representation theory) as mentioned in the introduction. Theorem \ref{thmD} positively answers the following question in several prominent cases. 

\begin{question}[Khare]\label{Q2}
   Let $\mathfrak{g}$ be a Kac--Moody algebra, $\lambda\in\mathfrak{h}^*$, and $M(\lambda)\twoheadrightarrow V$. Suppose $\mu_0 \precneqq \mu \in \wt V$. Then does there exist a sequence of weights $\mu_i \in \wt V$, $1\leq i\leq n=\height(\mu -\mu_0)$, such that
  \[
   	\mu_0 \prec \cdots \prec \mu_i \prec \cdots \prec \mu_n=\mu \in \wt V\quad \text{ and }\quad
   	\mu_i-\mu_{i-1}\in \Pi\text{ } \forall i ?
   \]
    See Note \ref{N1} for the notation.
\end{question}
This question was answered positively in the cases when $V$ is a finite-dimensional simple highest module or a parabolic Verma module over semisimple $\mathfrak{g}$, by S. Kumar and Khare respectively. See Appendix in \cite{Khare_JA}.
\begin{thmx}\label{thmD}
Let $\mathfrak{g}$ be a Kac--Moody algebra and $V$ be a $\mathfrak{g}$-module. Suppose 
\begin{itemize}
\item[(D1)] $V=\mathfrak{g}$ (adjoint representation) or
\item[(D2)] $\lambda\in\mathfrak{h}^*$, $J\subset J_{\lambda}$ and $V$ is a submodule of $M(\lambda,J)$,
\end{itemize}
and $\mu_0\precneqq \mu \in \wt V$. There exists a sequence of weights $\mu_i \in \wt V, 1\leq i\leq n=\height(\mu-\mu_0)$, such that
 \begin{equation*}
   	\mu_0 \prec \cdots \prec \mu_i \prec \cdots \prec \mu_n=\mu \in \wt V\quad and \quad 
   	\mu_i-\mu_{i-1}\in \Pi\text{ } \forall i.
   \end{equation*} 
\end{thmx}
Theorem \ref{thmD} applies to the above and other class of modules:
\begin{itemize}
\item On one hand, note that the result holds for the set of weights of parabolic Verma modules. Hence, Theorem \ref{thmD} holds true for a bigger class of modules, by Theorem \ref{T2.3} (a).
\item In particular, this proves the theorem for all simple highest weight modules. Also, more generally, this proves the theorem for all the highest weight modules $V$ of the form on the right hand side of the (if and only if) implication in Theorem \ref{thmB}.  
\item On the other hand, observe that many of the submodules of parabolic Verma modules need not be highest weight modules. Thus, we are able to prove Theorem \ref{thmD} not only for the above highest weight modules but also for certain ``non-highest weight'' modules.
\item Motivated by the previous point, we also extend Theorem \ref{thmD}
for any integrable module (not necessarily highest weight or
with finite-dimensional weight spaces) over semisimple $\mathfrak{g}$, see Lemma \ref{L5.4} below.
\end{itemize}
\begin{remark}\label{R2.10}
A perhaps surprising observation is that Theorem \ref{thmD} holds true for the sets $\mathbb{Z}_{\geq 0}\Delta_{I,1}$ for $\emptyset \neq I \subset\mathcal{I}$ (with $\mathbb{Z}_{\geq0}\Delta_{I,1}$ in the place of $\wt V$ in Theorem \ref{thmD}). The previous line holds true by Theorem \ref{thmD} applied to $\wt M(\hat{\lambda},I^c)$ where $\langle\hat{\lambda},\alpha_{i}^{\vee}\rangle=0$ $\forall$ $i\in I^c$.
\end{remark}
In this paper, we also prove additional (parabolic) generalizations of Theorem \ref{thmD} to the ``best possible extent'', see subsection \ref{S5.2}.

For $\lambda \in \mathfrak{h}^*$, $J\subseteq J_{\lambda}':=\{j \in \mathcal{I}$ $|$ $\langle 
\lambda ,\alpha_j^{\vee}\rangle \in \mathbb{R}_{\geq 0}\}$ we define
\begin{equation}\label{E2.13}
   \mathcal{P}(\lambda ,J):= \conv_{\mathbb{R}}W_J\lambda-\mathbb{R}_{\geq 0}(\Delta^+ \setminus \Delta_{J}^+),
\end{equation}
where $W_J$ is the parabolic subgroup of $W$ generated by $\{ s_j$ $|$ $j \in J\}$. The sets $\mathcal{P}(\lambda,J)$ were introduced, and the faces and inclusion relations among its faces studied, in \cite{Khare_Ad, Khare_Trans}. As one application of Theorems \ref{thmA} and \ref{thmD} in this paper, we identify the extremal rays to the shape $\mathcal{P}(\lambda,J)$ defined over Kac--Moody $\mathfrak{g}$, in the Appendix of this paper; see Proposition \ref{P2.11} below. This result was proved in \cite{Dhillon_arXiv, Khare_Ad}, and our proof is different to the proofs therein. The shape $\mathcal{P}(\lambda,J)$ generalizes the convex hull of the weights of parabolic Verma module $M(\lambda,J)$ when $J\subset J_{\lambda}$. In view of Theorem \ref{T2.3} (b), this gives the extremal rays of $\conv_{\mathbb{R}}\wt V$ for any $M(\lambda)\twoheadrightarrow V$ such that $I_V=J$. Closedness, polyhedrality, and the faces/weak faces of the convex hull of the set of weights of parabolic Verma modules, and hence of the simple highest weight modules, are well studied in \cite{Khare_Ad,Khare_JA,Khare_AR}. This builds on previous work for finite-dimensional modules by Borel--Tits \cite{Borel}, Casselman \cite{Casselman}, Cellini--Marietti \cite{Cellini}, Satake \cite{Satake}, Vinberg \cite{Vinberg}. Additionally, we also discuss a maximal property of $\mathcal{P}(\lambda,J)$, see Maximal property \ref{MA.2} in the Appendix. This was studied and proved in \cite{Dhillon_arXiv, Khare_Ad, Khare_Trans}, and again we provide a different proof. 
\begin{prop}\label{P2.11}
 	Let $\mathfrak{g}$ be a Kac--Moody algebra, $\lambda \in \mathfrak{h}^*$ and $J\subseteq J'_{\lambda}$. Define $J_0:= \{j \in J$ $|$ $s_{j}(\lambda)=\lambda \}$. Then $\mathcal{P}(\lambda ,J)$ is a $W_J$-invariant convex subset of $\lambda -\mathbb{R}_{\geq0}\Pi$, with the set of extremal rays $\bigsqcup\limits_{i\in J^c}W_J(\lambda -\mathbb{R}_{\geq 0}\alpha_i)$. In particular, the extremal rays at $\lambda$ are $\bigsqcup\limits_{i\in J^c}W_{J_0}(\lambda-\mathbb{R}_{\geq 0}\alpha_i)$.
 \end{prop}
The following remark is important to be observed, as it will be invoked in most of the results of this paper without  mention.
\begin{remark}[Results over related Kac--Moody algebras]
  While the above results---and results in the later sections---are stated and proved over the Kac--Moody algebra $\mathfrak{g}(A)$, they hold equally well over the Kac--Moody algebra $\overline{\mathfrak{g}}(A)$ (see Definition \ref{D2.1} (1)) and hence uniformly over any `intermediate' algebra $\overline{\mathfrak{g}}(A)\twoheadrightarrow\tilde{\mathfrak{g}}\twoheadrightarrow\mathfrak{g}=\mathfrak{g}(A)$. This is because as was clarified in \cite{Dhillon_arXiv}, the results there hold over all such Lie algebras $\tilde{\mathfrak{g}}$; similarly, the root system and the weights of highest weight modules (of a fixed highest weight $\lambda$) remain unchanged over all $\tilde{\mathfrak{g}}$. \end{remark}
\section{Parabolic-PSP and analogues}\label{S3}
\subsection{Proof of Theorem \ref{thmA}}
We begin this subsection with the proof of parabolic-PSP in the more general setting of a $\mathbb{Z}_{\geq0}\Pi$-graded $\mathbb{F}$-Lie algebra $\mathcal{G}$ generated by its non-zero subspaces $\mathcal{G}_{\alpha_i}$, $\forall$ $i\in \mathcal{I}$. Here, $\mathbb{F}$ is an arbitrary field, $\Pi$ freely generates the abelian semigroup $\mathbb{Z}_{\geq 0}\Pi$ and $\mathcal{I}$ is a fixed indexing set for $\Pi$. Recall, $\mathcal{A}:=\{\alpha\in\mathbb{Z}_{\geq0}\Pi$ $|$ $\mathcal{G}_{\alpha}\neq \{0\}\}$. Throughout this subsection, we neither assume $\Pi$ to be finite, nor $\mathcal{G}_{\beta}$ for $\beta \in \mathcal{A}$ (in particular, $\mathcal{G}_{\alpha_i}$ $\forall$ $i\in\mathcal{I}$) to be finite-dimensional.\\ 
For $\eta\in\mathbb{Z}_{\geq0}\Pi$ and $0\neq x\in\mathcal{G}_{\eta}$, we define $\gr(x):=\eta$, the grade of $x$. Recall from equation \eqref{E2.6} that $\mathcal{A}_{I,1}:=\{\gamma \in \mathcal{A}$ $|$ $\height_{I}(\gamma)=1\}$. 
	\begin{theorem}\label{T3.1}
	Let $\mathcal{G}$ be a $\mathbb{Z}_{\geq0}\Pi$-graded $\mathbb{F}$-Lie algebra generated by non-zero subspaces $\mathcal{G}_{\alpha_i}$, $\forall$ $i \in \mathcal{I}$. Suppose $\emptyset \neq I\subset \mathcal{I}$, and $\beta \in \mathcal{A}$ such that $\height_{I}(\beta)>0$. Then $\mathcal{G}_{\beta}$ is spanned by the Lie words of the form $\Big[x_{\gamma_1},\big[\cdots ,[x_{\gamma_{n-1 }},x_{\gamma_n}]\cdots \big]\Big]$ such that $\gamma_j \in \mathcal{A}_{I ,1}$ and $0\neq x_{\gamma_j} \in \mathcal{G}_{\gamma_j}$ for each $1\leq j\leq n= \height_{I}(\beta)$, and $\sum\limits_{j=1}^{n}\gamma_j =\beta$. 
	\end{theorem}
	\begin{proof}
		We use the following notation for convenience (only) in this proof: let $\mathfrak{I}$ be an indexing set and $\Theta=\{\theta_t $ $|$ $ t\in \mathfrak{I}\}$ be a fixed basis of $\mathcal{G}_{I,1}:=\bigoplus\limits_{\eta \in \mathcal{A}_{I,1}}\mathcal{G}_{\eta}$ consisting of homogeneous elements---i.e.  $\gr(\theta_t)\in \mathcal{A}_{I,1}$ for each $t \in \mathfrak{I}$. For a finite ordered sequence $\hat{a}$ with terms in $\mathfrak{I}$, we define $\theta_{\hat{a}} := [[\theta_{a}]]_{a\in\hat{a}}$, see equation \eqref{E2.1} where this notation was introduced.
		
		We prove this theorem by induction on $\height(\beta)\geq 1$. In the base step $\height(\beta)=1$, $\beta$ must belong to $\Pi_I$, and so the theorem is immediate.\\
		Induction step: Assume $\height(\beta)>1$, observe that the result is trivial if $\height_{I}(\beta)=1$. So, we assume for the rest of the proof that $\height_{I}(\beta)>1$. Pick 
		\begin{align*} 0\neq X =\Big[e_{i_1}, \big[\cdots [e_{i_{k-1}}&, e_{i_k}] \cdots\big]\Big]=[[e_{i_j}]]_{j=1}^{k}\in \mathcal{G}_{\beta}\\
		\text{ such that }i_j \in \mathcal{I}\text{ and }0\neq e_{i_j}\in \mathcal{G}_{\alpha_{i_j}}
		\text{ for }&\text{each }1\leq j\leq k:=\height(\beta)\text{, and }\sum\limits_{j=1}^{k}\alpha_{i_j}=\beta.
		\end{align*}
		Note that $\mathcal{G}_{\beta}$ is spanned by the Lie words of the form $X$, as $\mathcal{G}$ is $\mathbb{Z}_{\geq0}\Pi$-graded and generated by $\mathcal{G}_{\Pi}$. We will show that $X$ is a linear combination of the Lie words as in the statement, which implies the proof of the theorem. Note that $\height_I(\beta-\alpha_i)\geq 1$, so by the induction hypothesis we have
		\begin{equation}\label{E3.1} [[e_{i_j}]]_{j=2}^{k}= \sum\limits_{\hat{u}}d_{\hat{u}}\theta_{\hat{u}} \quad\text{and}\quad X = \sum\limits_{\hat{u}}d_{\hat{u}}[e_{i_1} , \theta_{\hat{u}}],
		\end{equation}
		where both the sums are over some finitely many finite
		sequences $\hat{u}$, each with terms in $\mathfrak{I}$,
		and $d_{\hat{u}}\in \mathbb{F}$. If $i_1\in I$, then we
		are done, as each term in the summation for $X$ in
		equation \eqref{E3.1} is in the form of the Lie words as
		in the statement. Else if $i_1\notin I$, consider a
		non-zero summand $[e_{i_1} ,\theta_{\hat{a}}]$ of $X$ in
		equation \eqref{E3.1} for some ordered sequence
		$\hat{a}=(a_{l})_{l=1}^{\small{\height_{I}(\beta)}}$.
		Let $\hat{b}=(a_{l})_{l=2}^{\small{\height_{I}(\beta)}}$,
		and recall that we assumed $\height_I(\beta)\geq 2$. By
		the Jacobi identity,
		\begin{align*}
		\big[e_{i_1}, [\theta_{a_{1}}, \theta_{\hat{b}}]\big]&=\big[[e_{i_1} ,\theta_{a_1}], \theta_{\hat{b}}\big]+\big[\theta_{a_1},[e_{i_1}, \theta_{\hat{b}}]\big].
		\end{align*}
		Note that $[e_{i_1},\theta_{a_1}]\in\mathcal{G}_{I,1}$. So, we can express $[e_{i_1},\theta_{a_1}]$ as $\sum\limits_{t} p_t \theta_t$, where the sum is over some finitely many $\theta_t \in \Theta$, $t\in \mathfrak{I}$, such that $\gr(\theta_t)= \alpha_{i_1} + \gr( \theta_{a_1})$ and $p_{t} \in \mathbb{F}$. By the induction hypothesis, we can express $[e_{i_1}, \theta_{\hat{b}}]$ as $\sum\limits_{\hat{c}} q_{\hat{c}}\theta_{\hat{c}}$, where the sum is over some finitely many finite sequences $\hat{c}$ each with terms in $\mathfrak{I}$ such that $\gr( \theta_{\hat{c}}) =\alpha_{i_1} + \gr(\theta_{\hat{b}})$ and $q_{\hat{c}}\in \mathbb{F}$. By the previous two lines we get
		\begin{equation*}
		[e_{i_1},\theta_{\hat{a}}]=\big[e_{i_1}, [\theta_{a_1}, \theta_{\hat{b}}]\big]= \sum\limits_{t} p_t [\theta_t ,\theta_{\hat{b}}] + \sum\limits_{\hat{c}} q_{\hat{c}}[\theta_{a_1}, \theta_{\hat{c}}].
		\end{equation*}
		Note that each summand on the right hand side of the equation just above is a Lie word of the form similar to that of the ones in the statement of the theorem. 
		So, every non-zero summand of $X$, and hence $X$ can be expressed as a linear combination of the Lie words as in the statement. Hence, the proof is complete.
		\end{proof}
	Observe that if $\Pi$ is not free, then we cannot define the functions $\height(.)$ and $\height_I(.)$. In view of this, Theorem \ref{T3.1} cannot be further extended to the Lie algebras graded over arbitrary semigroups. We now prove a corollary of Theorem \ref{T3.1}, part (b) of this corollary proves Theorem \ref{thmC} in one direction.\begin{cor}\label{C3.2}
		\begin{itemize}
		\item[(a)] Let $\mathcal{G}$ be as in Theorem \ref{T3.1} and $\emptyset\neq I\subsetneq\mathcal{I}$. Then $U(\mathcal{G})$ is spanned by the monomials all of the form either $\prod\limits_{j=1}^{n}x_{\beta_j}^{p_j}\cdot\prod\limits_{i=1}^{m}x_{\gamma_i}^{q_i}$ or $\prod\limits_{i=1}^{m}x_{\gamma_i}^{q_i}\cdot\prod\limits_{j=1}^{n}x_{\beta_j}^{p_j}$, where $\beta_j\in\mathcal{A}_{I^c}=\mathcal{A}_{I,0}$, $\gamma_i\in\mathcal{A}_{I,1}$, $x_{\beta_j}\in\mathcal{G}_{\beta_j}$, $x_{\gamma_i}\in\mathcal{G}_{\gamma_i}$, $p_j,q_i\in\mathbb{Z}_{\geq0}$ $\forall$ $1\leq j\leq n,1\leq i\leq m$.
		\item[(b)] Let $\mathfrak{g}$ be a Kac--Moody algebra, $\lambda\in\mathfrak{h}^*$, $J\subset \mathcal{I}$, and $M(\lambda)\twoheadrightarrow{ }V$ be a highest weight $\mathfrak{g}$-module. (Recall from equation \eqref{E2.4} that $\wt_JV:=\wt V\cap (\lambda-\mathbb{Z}_{\geq 0}\Pi_{J})$.) Then 
		\[\wt V\subset \wt_J V-\mathbb{Z}_{\geq 0}\Delta_{J^c,1}.\] 
	\end{itemize}
	\end{cor}
	\begin{proof}
	Observe that in the extreme cases where $J=\emptyset$ or $J=\mathcal{I}$ (b) trivially holds true. When $\emptyset\neq J\subsetneq \mathcal{I}$ observe that (b) just follows from (a), as the weight spaces of $V$ are spanned by the vectors obtained when the monomials of the second form in (a) act on $V_{\lambda}$ (with $\mathcal{G}$, $\mathcal{A}$, $I$, $\mathcal{A}_{I,1}$ in (a) replaced by $\mathfrak{n}^-$, $-\Delta$, $J^c$, $\Delta_{J^c,-1}$ respectively). So, we only prove (a). Fix an ordered basis $\mathcal{B}$ for $\mathcal{G}$ consisting of homogeneous elements such that the elements corresponding to $\mathcal{A}_{I^c}$ always occur either before or after those corresponding to $\mathcal{A}\setminus\mathcal{A}_{I^c}$ in $\mathcal{B}$. By the PBW theorem, (ordered) monomials on the elements of $\mathcal{B}$ span $U(\mathcal{G})$. Apply Theorem \ref{T3.1} to the elements of $\mathcal{B}$ corresponding to $\{\alpha\in\mathcal{A}$ $|$ $\height_{I}(\alpha)>1\}$. Upon re-writing the Lie brackets in terms of commutators in $U(\mathcal{G})$, observe that each element in $\mathcal{B}$ corresponding to $\{\alpha\in\mathcal{A}$ $|$ $\height_{I}(\alpha)>1\}$ is further a ``polynomial'' on the elements of $\mathcal{B}$ corresponding to $\mathcal{A}_{I,1}$. Now, observe that the previous line finishes the proof.
	\end{proof}
	\begin{proof}[\textnormal{\textbf{Proof of Theorem \ref{thmA}:}}]
		The proof of (\ref{thmA}1) follows by Theorem \ref{T3.1}
		applied to $\mathcal{G}=\mathfrak{n}^+$ with $\Pi$ as the
		base of $\Delta$. By the Minkowski difference formula for
		$\wt M(\lambda,J)$ in equation \eqref{E2.7}, and by
		applying (\ref{thmA}1) for $I=J^c$, the proof of
		(\ref{thmA}2) follows. Observe that $\Delta_{I ,1}$ is
		the minimal generating set for the semigroup
		$\mathbb{Z}_{\geq 0}(\Delta^+ \setminus \Delta_{I^c}^+)$,
		as a root in $\Delta_{I,1}$ cannot be further written as
		a sum of roots in $\Delta_{I,1}$. This justifies the term
		``minimal'' description, completing the
		proof of Theorem \ref{thmA}.
	\end{proof}
	The following remark addresses some questions related to the ``free-ness'' of the subset $\Delta_{I,1}$ in generating the cones/semigroups $\mathbb{Z}_{\geq0}(\Delta^+\setminus\Delta^+_{I^c})$  and $\mathbb{R}_{\geq0}(\Delta^+\setminus\Delta^+_{I^c})$.
	\begin{remark}\label{R3.3}
		\begin{itemize}
		\item[(1)] Given $\emptyset \neq I \subset \mathcal{I}$, $\Delta_{I,1}$ need not ``freely'' generate the semigroup $\mathbb{Z}_{\geq0}(\Delta^+\setminus\Delta^+_{I^c})$, as the following example shows. Let $\mathfrak{g} =\mathfrak{sl}_4(\mathbb{C})$  and $\mathcal{I}=\{1, 2,3\}$ where the node $2$ is not a leaf in the Dynkin diagram, and suppose $I=\{2\}$. Then we have $(\alpha_1 +\alpha_2 ) + (\alpha_2 +\alpha_3 ) = (\alpha_1 +\alpha_2 +\alpha_3 )+ (\alpha_2)$.
		\item[(2)] Note that $\Delta_{I,1}$ is a generating set for the cone $\mathbb{R}_{\geq 0}(\Delta^+ \setminus \Delta_{I^c}^+)$, but it need not be minimal, as the following example shows. Let $\mathfrak{g}$ be of type $B_2$ and $\mathcal{I}=\{1,2\}$ where the node 2 corresponds to the long simple root, and suppose $I=\{2\}$. Then we have $\frac{1}{2}(\alpha_2)+\frac{1}{2}(\alpha_2+2\alpha_1)=\alpha_2+\alpha_1$.
		\item[(3)] In the above spirit, we find the minimal generating (over $\mathbb{R}_{\geq0}$) set for the cone $\mathbb{R}_{\geq 0}(\Delta^+ \setminus \Delta_{I^c}^+)$ in Lemma \ref{L6.1} in subsection \ref{S6.1}.
		\end{itemize} 
	\end{remark}
	Let $\mathcal{G}$ be as in Theorem \ref{T3.1} and $I\subset\mathcal{I}$. We end this subsection by exhibiting (1) a spanning set for $\mathcal{G}_{I,1}$, and thereby (2) a lower bound on the size of $\mathcal{A}_{I,1}$, see Lemma \ref{L3.4}. Lemma \ref{L3.4} was proved when $\mathcal{G}$ is a Borcherds Kac--Moody algebra by Arunkumar et al in \cite[Lemma 4.6]{Venkatesh}. We now prove it in the more general graded setting of this paper.
	\begin{lemma}\label{L3.4}
		Let $\mathcal{G}$ be as in Theorem \ref{T3.1}. Fix an $i \in \mathcal{I}$, and suppose $\beta \in \mathcal{A}$ such that $\height_{\{i\}}(\beta)>0$. Then $\mathcal{G}_{\beta}$ is spanned by the Lie words of the form $\Big[e_{i_n}, \big[\cdots ,[e_{i_1},e_{i}]\cdots\big] \Big]$ such that $i_j, i\in \mathcal{I}$, $0\neq e_{i_j} \in \mathcal{G}_{\alpha_{i_j}}$, $0\neq e_i \in\mathcal{G}_{\alpha_i}$ for each $1\leq j\leq n$, and $\alpha_i+\sum_{j=1}^{n}\alpha_{i_j}=\beta$.
	\end{lemma}
	\begin{proof}
		We proceed by induction on $\height(\beta)\geq 1$. In the base step $\height(\beta)=1$, the lemma is trivial.\\
		Induction step: Let $\beta$ be as in the statement with $k:=\height(\beta )>1$. Pick a non-zero Lie word \begin{align*} &X=\Big[e_{\nu_{1}},\big[\cdots, [e_{\nu_{k-1}}, e_{\nu_k}]\cdots\big]\Big] =[[e_{\nu_j}]]_{j=1}^k\in \mathcal{G}_{\beta}\\
		\text{ such that }\nu_j \in& \Pi_{\supp(\beta)}\text{ and }0\neq e_{\nu_j}\in \mathcal{G}_{\nu_j}\text{ for each }1\leq j\leq  k,\text{ and }\sum_{j=1}^k\nu_j=\beta.\end{align*}
		Note that Lie words of the form $X$ span $\mathcal{G}_{\beta}$ as $\mathcal{G}$ is generated by $\mathcal{G}_{\Pi}$ and graded over $\mathbb{Z}_{\geq0}\Pi$. Therefore, it suffices to show that $X$ is a linear combination of Lie words of the form as in the statement. Fix $s\in [k]$ largest such that $\nu_s=\alpha_{i}$; such an $s$ exists as $\height_{\{i\}}(\beta)>0$. Note that $k=ht(\beta)\geq 2$. When $k=2$, check that the lemma trivially holds. So, we assume that $k\geq 3$. If $s>1$, then $[[e_{\nu_{j}}]]_{j=2}^{k}\in \mathcal{G}_{\beta-\nu_1}$ can be expressed as
		 a linear combination of the Lie words of the desired form by the induction hypothesis applied to $\beta -\nu_1$, and hence so can be $X$. Else if $s=1$, then observe by the Jacobi identity that
		\begin{equation}\label{E3.2}
			\begin{split}
			\big[e_{\nu_{1}},[\cdots [e_{\nu_{k-1}}, e_{\nu_k}]\cdots ]\big] =&\sum\limits_{j=2}^{k-1}(-1)^{j}\Big[e_{\nu_j},\big[\big([[e_{\nu_{j-p}}]]_{p=1}^{j-1}\big),\big([[e_{\nu_q}]]_{q=j+1}^{k}\big)\big]\Big]+\\&(-1)^{k}\Big[\big([[e_{\nu_{k-p}}]]_{p=1}^{k-1}\big),e_{\nu_k}\Big].
			\end{split}
		\end{equation}
		Each non-zero $\Big[\big([[e_{\nu_{j-p}}]]_{p=1}^{j-1}\big),\big([[e_{\nu_q}]]_{q=j+1}^{k}\big)\Big]\in \mathcal{G}_{\beta -\nu_j}$ can be expressed as the desired linear combination by the induction hypothesis applied to $\beta -\nu_j$ for each $2\leq j\leq k$. Hence, every term in the summation on the right hand side of equation \eqref{E3.2} above can be expressed as a linear combination of the Lie words of the desired form. Notice that the last term outside the summation on the right hand side of equation \eqref{E3.2} is already in the desired form (once we reverse the order of the outermost Lie bracket). Hence, the proof is complete.
	\end{proof}
	Lemma \ref{L3.4} immediately proves the following Corollary. Notice that Corollary \ref{C3.5} (1) is a special case of Theorem \ref{thmD} part (\ref{thmD}1), more generally, in the graded setting not just general Kac--Moody.
		\begin{cor}\label{C3.5}
			Let $\mathcal{G}$ and $\mathcal{A}$ be as in Lemma \ref{L3.4}.
			\begin{itemize}
			    \item[(1)] Fix an $i\in \mathcal{I}$. Suppose $\beta \in \mathcal{A}$ such that $\beta \succneqq \alpha_i$. Then there exists a sequence of grades $\beta_j \in \mathcal{A}$, $1\leq j\leq n=\height(\beta)$, such that
			\[
				\alpha_i=\beta_1\prec\cdots\prec\beta_j\prec\cdots\prec \beta_n=\beta
			\in\mathcal{A}\quad\text{ and }\quad
				\beta_{j+1}-\beta_{j}\in \Pi \text{ }\forall j. 	
			\]
			(See Note \ref{N1} for the notation.) Notice that $\height_{\{i\}}(\beta)>0 \implies \height_{\{i\}}(\beta_j)>0$ $\forall j$. In particular, $\height_{\{i\}}(\beta)=1\implies \height_{\{i\}}(\beta_j)=1$ $\forall j$, due to which we get a lower bound on the size of $\mathcal{A}_{I,1}$ as in the next part. 
			\item[(2)] Let $\emptyset\neq I\subset \mathcal{I}$. For each $i\in I$, fix a grade $\beta_i \in \mathcal{A}_{I,1}$ such that $\supp_I(\beta_i)=\{i\}$. Then
			\begin{equation*}
			\#\mathcal{A}_{I,1}\geq \sum_{i\in I}\height(\beta_i).
			\end{equation*}
            \end{itemize}		
		\end{cor}
		\subsection{Parabolic-PSP-going up version} 
		In the rest of the paper, we work only over Kac--Moody algebras, which we denote by $\mathfrak{g}$, and also unless otherwise stated we assume that $\mathcal{I}$ is finite.
		
		In this subsection, we prove a ``going up'' version of the parabolic-PSP for an indecomposable Kac--Moody $\mathfrak{g}$, using Theorem \ref{T3.1}. This generalizes a basic result in the theory, see e.g. Proposition 4.9 in Kac's book \cite{Kac}. 
	\begin{theorem}\label{T3.6}
		Let $\mathfrak{g}$ be an indecomposable Kac--Moody algebra, and $\emptyset \neq I\subset \mathcal{I}$. Suppose $\beta \in \Delta^+$ is such that $\height_{I}(\beta)< \sup\{ \height_I(\alpha)$ $|$ $\alpha\in\Delta\}$ and $0\neq x\in\mathfrak{g}_{\beta}$. Then there exists a root $\gamma \in \Delta_{I,1}$ and $0\neq x_{\gamma} \in \mathfrak{g}_{\gamma}$ such that $[x_{\gamma}, x]\neq 0$, which implies $\beta +\gamma \in \Delta$. 
	\end{theorem}
	\begin{proof}
		Fix $\beta$ and $x$ as in the statement of the theorem, and also a root $\beta'$ such that $\height_{I}(\beta')>\height_{I}(\beta)$. Consider the ideal $L:=[\mathfrak{g}, x]\oplus \mathbb{C}x$. By \cite[Proposition 1.7]{Kac}, we must have either
		a) $L\subset Z(\mathfrak{g})$ the center of $\mathfrak{g}$, or b) $L\supset \mathfrak{g}'\supset\bigoplus\limits_{\alpha \in \Delta}\mathfrak{g}_{\alpha}$, where $\mathfrak{g}'$ is the subalgebra generated by the Chevalley generators of $\mathfrak{g}$. Note that a) is not possible, as $Z(\mathfrak{g})\subset \mathfrak{h}$ by \cite[Proposition 1.6]{Kac}. So, b) holds, and we must have a non-zero Lie word of the form 
		\begin{equation}\label{E3.3}
		\begin{split}
		\Big[x_{\gamma_k},\big [ \cdots ,[x_{\gamma_1}, x]\cdots \big] \Big]& \in \mathfrak{g}_{\beta'} \subset L \\ \text{such that} \text{ } \gamma_j \in \Delta \sqcup \{0\},\text{ } 0\neq x_{\gamma_j}\in \mathfrak{g}_{\gamma_j}\text{ }&\forall \text{ }j\in [k]\text{, } \text{and} \text{ } \sum\limits_{j=1}^{k}\gamma_j =\beta' -\beta.
		\end{split}
		\end{equation}
		Assume without loss of generality that the above Lie word has least ``length'' $k$---length here denotes the number of elements occurring in the iterated Lie bracket---among all the Lie words satisfying all the conditions and of the form in equation \eqref{E3.3}. Now we proceed in two cases below.\newline\newline
		(1) $k=1$: Note in this case that $0\neq[x_{\gamma_1},x]\in \mathfrak{g}_{\beta'}$ and $\height_I(\beta'-\beta)>0$ together imply that $\height_{I}(\gamma_1)>0$. So, $x_{\gamma_1}\in\mathfrak{g}_{\gamma_1}$ is some linear combination of the Lie words as in the statement of Theorem \ref{T3.1} (for Kac--Moody $\mathfrak{g}$). Observe then that  in such a linear combination for $x_{\gamma_1}$, there must exist a non-zero Lie word
		\[ Y= \Big[x_{\eta_t},\big[\cdots [x_{\eta_2},x_{\eta_1}]\cdots\big]\Big]\text{ such that }[Y,x]\neq 0,\]
		\[\text{ where }\eta_j \in \Delta_{I ,1}\text{ and }0\neq x_{\eta_{j}}\in\mathfrak{g}_{\eta_{j}}\text{ for each }1\leq j\leq t= \height_{I}(\gamma_1),\text{ and }\sum\limits_{j=1}^{t}\eta_j =\gamma_1.\]
		As `$\ad\text{ }x$' acts on $Y$ by the derivation rule, we must have $[x_{\eta_p},x]\neq 0$ for some $p\in [t]$. Hence, the proof is complete in this case.\newline\newline		
		(2) $k>1$: By the derivation rule, for each $2\leq i\leq k$ we have 
		\begin{align*}\label{E3.4}
		\tag{3.4}
		&\bigg[x_{\gamma_k},\Big[\cdots ,\big[x_{\gamma_i },\cdots ,[x_{\gamma_1},x]\cdots \big] \cdots\Big] \bigg]=\\ \Big(\prod\limits_{\ell=k}^{i+1}\ad x_{\gamma_{\ell}}\Big)\Big(\prod\limits_{\ell=i-1}^{1}\ad x_{\gamma_{\ell}}\Big)&([x_{\gamma_i},x]) +\sum\limits_{j=1}^{i-1}\Big[\big(\prod\limits_{\ell =k}^{i+1}\ad x_{\gamma_{\ell}}\big) \big(\prod\limits_{\ell =i-1}^{j+1}\ad  x_{\gamma_{\ell}}\big) \ad ([x_{\gamma_i},x_{\gamma_j}])\big(\prod\limits_{\ell =j-1}^{1}\ad x_{\gamma_{\ell}}\big)\Big](x).
		\end{align*}
		(Note that the products on the right hand side of equation \eqref{E3.4} are written in such a way that they run over non-increasing indices, for convenience. Treat a product that might run over strictly increasing indices for some values of $i\text{ and }j$---for instance $j=i-1$---as the identity map on $\mathfrak{g}$.) Recall, $k$ is assumed to be the least length of all the Lie words of the form in equation \eqref{E3.4}. So, each term appearing in the summation on the right hand side of equation \eqref{E3.4} is zero. By the previous line and the assumption that the Lie word on the left hand side of equation \eqref{E3.4} is non-zero, we therefore have $[x_{\gamma_i},x]\neq 0$ $\forall$ $i\in [k]$. The previous line and the assumption that $\height_I(\beta'-\beta)>0$ together imply that $[x_{\gamma_q}, x]\neq 0$ for some $q\in [k]$ such that $\height_{I}(\gamma_q)>0$. This brings us to case (1) (with $\gamma_q+\beta$ in the place of $\beta'$ in case (1)). Hence, the proof is complete.
	\end{proof}
	We end this subsection with the following Lemma which generalizes another fundamental result on Kac--Moody algebras---see e.g. Lemma 1.5 in Kac's book \cite{Kac}. Note that this also gives an alternate proof of (a relatively stronger version of) the parabolic-PSP for Kac--Moody algebra $\mathfrak{g}$.
		\begin{lemma}\label{L3.7}
		Let $\mathfrak{g}$ be a Kac--Moody algebra and $\emptyset\neq I\subset \mathcal{I}$. Suppose $x\in\bigoplus\limits_{n\in\mathbb{Z}_{>0}} \mathfrak{g}_{I,n}$ is such that \[ [x,f_{\eta}]=0\text{ }\forall\text{ }\eta\in\Delta_{I,-1},\text{ }f_{\eta} \in \mathfrak{g}_{\eta}.\]
		Then $x=0$.
		\end{lemma}
		\begin{proof}
		Given $\beta \in \Delta^+$ with $\height_I(\beta)>0$ and $0\neq x\in \mathfrak{g}_{\beta}$, we prove by induction on $\height(\beta)\geq 1$ that there exists a root $\gamma \in \Delta_{I,-1}$ and a root vector $f_{\gamma} \in \mathfrak{g}_{\gamma}$ such that $\beta+\gamma \in\Delta^+\sqcup \{0\}$ and $[x,f_{\gamma}]\neq 0$. Observe that this proves the lemma. Fix an $\alpha \in \Pi$ and $f_{\alpha}\in \mathfrak{g}_{-\alpha}$ such that $\beta -\alpha \in \Delta^+\sqcup \{0\}$ and $[x, f_{\alpha}]\neq 0$, which exist by \cite[Lemma 1.5]{Kac}. Observe that in the base step $\height(\beta)=1$, $\beta\in\Pi_I$, and so the lemma is trivial by the previous sentence (with $\alpha=\beta$).\\
		 Induction step: Assume that $\height(\beta)>1$. Now, if $\alpha \in \Pi_{I}$, we are done. Otherwise, the induction hypothesis applied to $\beta -\alpha$ yields $\eta \in \Delta_{I,-1}$ and $f_{\eta}\in \mathfrak{g}_{\eta}$ such that $\big[[x,f_{\alpha}],f_{\eta}\big]\neq 0$ and $\beta -\alpha +\eta\in\Delta^+\sqcup \{0\}$. 
		 Now, by the Jacobi identity we have
		 \[ 		 	\big[[x,f_{\alpha}],f_{\eta}\big]= \big[x,[f_{\alpha} ,f_{\eta}]\big]-\big[f_{\alpha},[x,f_{\eta}]\big].
		 \]	  
		 Now, the result follows as one of the two terms on the
		 right hand side of the above equation must be non-zero.
		\end{proof}
		\subsection{Parabolic-PSP in the set of short roots in finite type}
		In this subsection, we assume that $\mathfrak{g}$ is of finite type---i.e. $\mathfrak{g}$ is a finite-dimensional simple Lie algebra---with root system $\Delta$. We prove the analogous parabolic-PSP and its ``going up'' version in the set of short roots of $\Delta$, see Proposition \ref{P3.9} below. We came across these phenomena while proving the parabolic-PSP for affine Kac--Moody algebras in a case by case manner. In view of Remarks \ref{R3.10} and \ref{R3.11} below, these phenomena cannot be further extended.
		
		In this subsection, we denote (1) the usual Killing form on $\mathfrak{h}^*$ by (.,.); (2) the length of a root $\alpha$ by $(\alpha,\alpha)$; (3) the set of short, short positive, long and long positive roots in $\Delta$ by $\Delta_{s}, \Delta^{+}_{s}, \Delta_{\ell}$ and $\Delta^+_{\ell}$ respectively; (4) the highest short root (which is also dominant) by $\theta_s$. When the Dynkin diagram of $\mathfrak{g}$ is simply laced---i.e $\mathfrak{g}$ is of type $A_n,D_n$ ($n\geq 2$), $E_6,E_7,E_8$---we assume $\Delta_s=\Delta_{\ell}=\Delta$ for convenience. Note that the simplicity of $\mathfrak{g}$ implies that $\Delta$ is irreducible. In finite type, recall:
		\begin{itemize} 
		\item[a)] The Killing form is positive definite, i.e. $(x,x)>0\text{ } \forall0\neq x\in\mathfrak{h}^*$.
		\item[b)] $\langle\beta,\alpha^{\vee}\rangle=2\frac{(\beta,\alpha)}{(\alpha,\alpha)}\text{ }\forall\alpha,\beta\in\Delta$.
		\item[c)] There can be at most two lengths in $\Delta$, by \cite[\S 10.4 Lemma C]{Hump}.
		\item[d)] Short roots in $\Delta$ are the roots of shortest length. If $\beta$ is short, then $\beta\prec\theta_s$.
		\end{itemize}
		In view of the fact that there can be at most two lengths in $\Delta$ and the results on the subroot system generated by any two roots in Subsection 9.4 of \cite{Hump}, one easily verifies the following observation.		\begin{observation}\label{O3.8}
		Let $\Delta$ be of finite type, and $\gamma\in\Delta$, $\alpha\in\Delta_s$ and $\beta\in\Delta_{\ell}$. Then both $\langle\alpha,\gamma^{\vee}\rangle$ and $\langle\gamma,\beta^{\vee}\rangle$ belong to $\{-1,0,1\}$. (As otherwise there would be more than two lengths in $\Delta$.)
		\end{observation}
		We now proceed to prove the main result of this subsection.
		\begin{prop}\label{P3.9}
			Let $\Delta$ be a finite type root system, $\emptyset \neq I\subset \mathcal{I}$, and fix $\beta \in \Delta^{+}_{s}$.
			\begin{itemize}
			\item[(a)] Suppose $\height_{I}(\beta)>1$. Then there exists a sequence of roots $\gamma_i \in \Delta_{I ,1}$, $1\leq i \leq n= \height_{I}(\beta)$, such that\[
			\gamma_1 \prec \cdots \prec \sum\limits_{j=1}^{i}\gamma_j \prec \cdots \prec \sum\limits_{j=1}^{n}\gamma_j =\beta \in \Delta_{s}\text{ }\forall i.\quad \text{(See Note \ref{N1} for the notation.)}
			\]
			\item[(b)] Suppose $\height_{I}(\beta) < \height_{I}(\theta_s)$. Then there exists a sequence of roots $\gamma'_i \in \Delta_{I ,1}$, $1\leq i \leq m=\height_{I}(\theta_s - \beta)$, such that	\[
			\beta \prec \cdots \prec \beta+ \sum\limits_{j=1}^{i}\gamma'_j \prec \cdots \prec \beta +\sum\limits_{j=1}^{m}\gamma'_j \in \Delta_s\text{ }\forall i.
			\]
			\end{itemize}
			\end{prop}
		\begin{proof}
		 We assume throughout the proof that the Dynkin diagram of $\mathfrak{g}$ is not simply laced, as otherwise the proposition obviously holds true by Theorems \ref{thmA} and \ref{T3.6}. We repeatedly use Observation \ref{O3.8} in the proof without mention.\\ We first prove (a) by showing that $\beta'\in\Delta_s$ with $\height_I(\beta')>1$ implies that $\beta'-\gamma'\in\Delta_s$ for some $\gamma'\in\Delta_{I,1}$. Assume $\beta'\in\Delta_s$ and $m:=ht_I(\beta')>1$. Then by the parabolic-PSP we can write $\beta'=\sum_{t=1}^{m}\gamma_t$ for some $\gamma_t\in\Delta_{I,1}$, $\forall$ $t\in[m]$. Now, observe the following implications. \[\langle\beta',\beta'^{\vee}\rangle=2\implies\langle\gamma_i,\beta'^{\vee}\rangle>0\text{ for some }i\in [m]\implies\langle\beta',\gamma_i^{\vee}\rangle>0\implies \langle\beta',\gamma_i^{\vee}\rangle=1.\]
		 The last implication in the above equation follows by Observation \ref{O3.8} as $\beta'$ is short.
		 The above equation yields $s_{\gamma_i}\beta'=\beta'-\gamma_i$. So, $\beta'-\gamma_i\in\Delta_s$, and therefore we are done.\\
		 We now prove (b) via showing by induction on $\height(\theta_s-\beta)\geq 1$ that $\beta\in\Delta_s$ with $ht_I(\theta_s-\beta)>0$ implies that $\beta+\gamma\in\Delta_s$ for some $\gamma\in\Delta_{I,1}$. In the base step $\height(\theta_s-\beta)=1$, the assumption $\height_I(\theta_s-\beta)>0$ forces $\theta_s-\beta\in\Pi_I$, and therefore we are done.\\
		 Induction step: Assume $\height(\theta_s-\beta)>1$. Recall that $\supp(\theta_s)=\mathcal{I}$. As $\beta\precneqq\theta_s$ and $\theta_s$ is the only dominant short root in $\Delta$, there must exist $\alpha\in\Pi$ such that $\langle\beta,\alpha^{\vee}\rangle<0$. As $\beta$ is short, we must have $\langle\beta,\alpha^{\vee}\rangle=-1$. So, $\beta+\alpha=s_{\alpha}\beta\in\Delta_s$, and also $\beta+\alpha\precneqq\theta_s$ as $\height(\theta_s-\beta)>1$. Now, if $\alpha\in\Pi_I$ we are done. Otherwise, by the induction hypothesis there exists $\eta\in\Delta_{I,1}$ such that $\beta+\alpha+\eta\in\Delta_s$. Now, observe that we must have $\langle\beta+\alpha+\eta,\eta^{\vee}\rangle>0$, as $\langle\beta+\alpha+\eta,\eta^{\vee}\rangle\leq0$ implies that $\langle\beta+\alpha,\eta^{\vee}\rangle\leq -2$ which cannot happen as $\beta+\alpha$ is short. Also, as $\beta+\alpha+\eta$ is short, the previous line implies $\langle\beta+\alpha+\eta,\eta^{\vee}\rangle$ $=1$, which further implies $\langle\beta+\alpha,\eta^{\vee}\rangle=-1$. This leads to two cases:  (i) $\langle\beta,\eta^{\vee}\rangle<0$ which implies $\langle\beta,\eta^{\vee}\rangle=-1$ (as $\beta$ is short), or (ii) $\langle\alpha,\eta^{\vee}\rangle<0$. If $\langle\beta,\eta^{\vee}\rangle=-1$, then $\beta+\eta=s_{\eta}\beta\in\Delta_s$, and therefore we are done. Else if $\langle\alpha,\eta^{\vee}\rangle<0$, then $\eta+\alpha\in\Delta_{I,1}$, and we are once again done as $\beta+(\eta+\alpha)\in\Delta_s$. Hence the proof of the proposition is complete. 
		 \end{proof}
		The following remarks discuss the limitations to further extend the above proposition.
		\begin{remark}\label{R3.10}
			\begin{itemize}
			\item[(1)] The analogous statements to those in Proposition \ref{P3.9} with $\beta \in \Delta^{+}_{\ell}$ are not true, as the following example shows. Let $\Delta$ be the root system of type $B_2$, $\mathcal{I}=\{1,2\}$ where the node 2 corresponds to the long simple root. Let $\beta' = \alpha_2 + 2\alpha_1$, $\beta = \alpha_2$ and $I = \mathcal{I}$. Then, we can neither come down from $\beta'$, nor go up from $\beta$, to a root in $\Delta_{\ell}^+$.
			\item[(2)] Note that when $\mathfrak{g}$ is semisimple, the above proposition works for each indecomposable component of $\Delta$.
		\end{itemize} 
		\end{remark}
		\begin{remark}\label{R3.11}
			Let $\Delta$ be an affine root system of type $C_{2}^{(1)}$, $\mathcal{I}=\{0,1,2\}$. Assume that the node $1$ is not a leaf in the Dynkin diagram and it corresponds to a short simple root (see \cite[Table Aff 1]{Kac}). Let $\beta=\alpha_0 +3\alpha_1 +\alpha_2 \in\Delta_{s}$, then observe that the only simple root that can be subtracted from $\beta$ to get again a root is $\alpha_1$. But then $\beta -\alpha_1$ is imaginary. In view of this, Proposition \ref{P3.9} cannot be extended beyond finite type. 
		\end{remark}		\section{Proofs of Theorems \ref{thmB} and \ref{thmC}}\label{S4}
	In this section, we prove Theorems \ref{thmB} and \ref{thmC}. We first recall the relevant notation. Throughout this section, $\mathfrak{g}$ stands for a general Kac--Moody algebra. For $\lambda\in\mathfrak{h}^*$, we denote by $m_{\lambda}$ a non-zero highest weight vector of $M(\lambda)_{\lambda}$, and $J_{\lambda}:=\{j\in \mathcal{I}$ $|$ $\langle\lambda,\alpha_j^{\vee}\rangle\in\mathbb{Z}_{\geq0}\}$. For $j\in J_{\lambda}$ we define $m_j:=\langle\lambda,\alpha_j^{\vee}\rangle+1\in\mathbb{Z}_{>0}$. Recall from equation \eqref{E2.2} that when $I=\emptyset$, we define $\Delta_{I,1}:=\emptyset$, and further we define both $\mathbb{Z}_{\geq 0}\Delta_{I,1}$ and $\mathbb{Z}_{\geq 0}\Pi_I$ in this case to be $\{0\}$ for convenience. For $\mu\in\mathfrak{h}^*$, $k\in\mathbb{Z}_{\geq 0}$ and $\alpha\in\Pi$, recall from equation \eqref{E2.5} that $\big[\mu-k\alpha,\text{ }\mu\big]:=\{\mu-j\alpha\text{ }\big|\text{ }j\in\mathbb{Z}\text{ and }0\leq j\leq k\}$. (This notation should not be confused with the Lie bracket.) We now begin by recalling the following fact about weight modules, and the definitions of $N(\lambda,J)$ and $N(\lambda)$. For $\lambda\in\mathfrak{h}^*$, until the end of subsection \ref{S4.1}, unless otherwise stated we will assume that $\emptyset\neq J\subset J_{\lambda}$. 
\begin{fact} 
If $S,T$ are two weight modules over $\mathfrak{g}$, then $\wt (S+T)=\wt S\cup \wt T$.
\end{fact}
For $\lambda\in\mathfrak{h}^*$ and $\emptyset\neq J\subset J_{\lambda}$, we define $N(\lambda,J)$ to be the largest proper submodule of $M(\lambda)$ with respect to the property: 
\begin{align*} 
&\text{ If }\mu\in\wt N(\lambda,J) \text{ is such that the Dynkin subdiagram on }\supp(\lambda-\mu) \text{ has no edge,}\tag{\textbf{P1}}\\&\text{ then }\supp(\lambda-\mu)\cap J\neq \emptyset.
\end{align*}
 In view of the above fact, observe that $N(\lambda,J)$ exists and is equal to the sum of all the submodules of $M(\lambda)$ each of whose set of weights (in the place of $\wt N(\lambda,J)$ in (P1)) have the property (P1). Also, check that the submodule $\sum\limits_{j\in J}U(\mathfrak{g})f_j^{\langle\lambda,\alpha_j^{\vee}\rangle+1}m_{\lambda}\subset M(\lambda)$ satisfies the property (P1). So, $\sum\limits_{j\in J}U(\mathfrak{g})f_j^{\langle\lambda,\alpha_j^{\vee}\rangle+1}m_{\lambda}\subset N(\lambda,J)$, and therefore $N(\lambda,J)$ is non-trivial.

Similarly, define $N(\lambda)$ to be the largest proper submodule of $M(\lambda)$ with respect to the property:
\begin{align*}\tag{\textbf{P2}} \text{If }\mu\in\wt N(\lambda),\text{ then the Dynkin subdiagram on }\supp(\lambda-\mu)\text{ has at least one edge.}\end{align*}
Note that $N(\lambda)$ also exists by the above fact (similar to $N(\lambda,J)$). Now, observe the following points about the submodules $N(\lambda)$ and $N(\lambda,J)$ of $M(\lambda)$. 
\begin{observation}\label{O4.1}
(1) If the Dynkin diagram of $\mathfrak{g}$ has no edges (isolated graph), then $N(\lambda)=\{0\}$.\newline
(2) If two nodes $i,j\in J_{\lambda}$ are adjacent in the Dynkin diagram, then $f_i^{\langle s_j\bullet\lambda,\alpha_i^{\vee}\rangle+1}f_j^{\langle\lambda,\alpha_j^{\vee}\rangle+1}m_{\lambda}\in M(\lambda)$ is a maximal vector, where $\bullet$ denotes the dot action of $W$ on $\mathfrak{h}^*$. Observe that $\{i,j\}\subset\supp(\lambda-\mu)$ $\forall$ $\mu\in\wt M((s_is_j)\bullet\lambda)$. So, for any $\mu\in\wt M((s_is_j)\bullet\lambda)$ the Dynkin subdiagram on $\supp(\lambda-\mu)$ is a non-isolated graph, as is the Dynkin subdiagram on $\{i,j\}$. Thus, $M((s_is_j)\bullet\lambda)\xhookrightarrow{ } N(\lambda)$, and hence $N(\lambda)$ is non-trivial in this case.\newline
(3) If $K$ is any submodule of $M(\lambda)$ such that $N(\lambda)\subsetneq K$, then observe by the definition of $N(\lambda)$ that there must exist a weight $\mu\in\wt K$ such that the Dynkin subdiagram on $\supp(\lambda-\mu)$ contains no edges. Note that $M(\lambda)_{\mu}$ is one-dimensional, so $K_{\mu}=M(\lambda)_{\mu}$ and $\mu\notin\wt\frac{M(\lambda)}{K}$. More generally, one can easily check by $\mathfrak{sl}_2$-theory that 
\[ \Big(\lambda-\sum_{i\in\supp(\lambda-\mu)}\big(\langle\lambda,\alpha_i^{\vee}\rangle+c_i\big)\alpha_i\Big)\notin \wt \frac{M(\lambda)}{K}\text{ }\forall \text{ }c_i\in\mathbb{Z}_{>0},\text{ }i\in \supp(\lambda-\mu),\]
as these weights belong to the $\mathfrak{g}_{\supp(\lambda-\mu)}$-module $U(\mathfrak{g}_{\supp(\lambda-\mu)})M(\lambda)_{\mu}\subset K$, and also their corresponding weight spaces in $M(\lambda)$ are all one-dimensional.\newline
(4) The integrability of $\frac{M(\lambda)}{N(\lambda,J)}$ is $J$, as for $j\in \mathcal{I}$, $f_j^{\langle\lambda,\alpha_j^{\vee}\rangle+1}m_{\lambda}\in N(\lambda,J)$ if and only if $j\in J$.\\
(5) If $\mu\in \wt N(\lambda,J)$ is such that the Dynkin subdiagram on $\supp(\lambda-\mu)$ has no edge, then, necessarily, there must exist $j\in J$ such that $\height_{\{j\}}(\lambda-\mu)\geq m_j$. This can be verified by checking via $\mathfrak{sl}_2$-theory that for $j'\in J_{\lambda}$ and $\mu$ as in the previous line, if $\height_{\{j'\}}(\lambda-\mu)<m_{j'}$, then $\mu+\big(\height_{\{j'\}}(\lambda-\mu)\big)\alpha_{j'}\in\wt N(\lambda,J)$.\\
(6) Suppose $x\in M(\lambda)_{\mu}$ is a maximal vector for some weight $\mu$ such that the Dynkin subdiagram on $\supp(\lambda-\mu)$ has an edge. Then $x\in N(\lambda)$. Similarly, for $J\subset J_{\lambda}$, suppose $y\in M(\lambda)_{\mu'}$ is a maximal vector for some weight $\mu'$ such that $\supp(\lambda-\mu')\subset J^c$ and the Dynkin subdiagram on $\supp(\lambda-\mu')$ has an edge. Then $y\in N(\lambda,J)$.   
\end{observation}

We now proceed to prove Theorems \ref{thmB} and \ref{thmC}. The key results used in the proof of Theorem \ref{thmB} are the following Lemma \ref{L4.2} and Proposition \ref{P2.7}. Lemma \ref{L4.2} is also the key result in the proof of Theorem \ref{thmC}. So, we first prove Lemma \ref{L4.2} below. For $I\subset \mathcal{I}$, recall from equation \eqref{E2.4} that $\wt_I V:=\wt V\cap (\lambda-\mathbb{Z}_{\geq 0}\Pi_I)$.
\begin{lemma}\label{L4.2}
Let $\lambda\in\mathfrak{h}^*$, and $M(\lambda)\twoheadrightarrow{ }V$. Then 
\[\mu-\sum\limits_{i\in J_{\lambda}^c}c_i\alpha_i\in\wt V\quad \text{ for any } \mu\in \wt_{J_{\lambda}}V \text{ and for any sequence }(c_i)_{i\in J_{\lambda}^c}\in(\mathbb{Z}_{\geq 0})^{|J_{\lambda}^c|}.  \]
\end{lemma}
\begin{proof}
Our approach in the proof of this lemma is very similar to that in the proof of Theorem 5.1 in \cite{Dhillon_arXiv}. Let $v_{\lambda}$ be the highest weight vector in $V$. Let $(c_i)_{i\in J_{\lambda}^c}$ be a sequence of non-negative integers, and $\mu\in \wt_{J_{\lambda}}V$. Fix an element $F$ in the graded piece $U(\mathfrak{n}^-)_{\mu-\lambda}$ such that $0\neq Fv_{\lambda}\in V_{\mu}$. We prove that $(F\prod_{i\in J_{\lambda}^c}f_i^{c_i}\big)v_{\lambda}\neq 0$, which proves the lemma as $\mu$ and $(c_i)_{i\in J_{\lambda}^c}$ are arbitrary. Treat $f_i^{c_i}$, similarly $e_i^{c_i}$, as $1\in U(\mathfrak{g})$ whenever $c_i=0$. Consider $(\prod_{i\in J_{\lambda}^c}e_i^{c_i})(F\prod_{i\in J_{\lambda}^c}f_i^{c_i})v_{\lambda}$. As $e_i$ and $f_j$ commute $\forall$ $i\neq j\in \mathcal{I}$, re-write $(\prod_{i\in J_{\lambda}^c}e_i^{c_i})(F\prod_{i\in J_{\lambda}^c}f_i^{c_i})v_{\lambda}$ as $(F\prod_{i\in J_{\lambda}^c}e_i^{c_i}f_i^{c_i})v_{\lambda}$. Now, as $e_i^{c_i}f_i^{c_i}$ acts by a non-zero scalar on $v_{\lambda}$ (by $\mathfrak{sl}_2$-theory) and $Fv_{\lambda}\neq 0$, observe that $(F\prod_{i\in J_{\lambda}}f_i^{c_i}v_{\lambda})$ must be non-zero in $V$. This finishes the proof. 
\end{proof}
\subsection{Proof of Theorem \ref{thmB}}\label{S4.1}
We begin this subsection with the proof of Proposition \ref{P2.7}. As stated earlier, Proposition \ref{P2.7} gives all the highest weight modules $V$ such that $\wt V=\wt M(\lambda)$, and also proves that $N(\lambda)$ is the largest submodule of $M(\lambda)$ such that the  weights of the corresponding quotient module are the same as those of $M(\lambda)$. Similarly, Theorem \ref{thmB} proves that $N(\lambda,J)$ is the largest submodule of $M(\lambda)$ such that the weights of the corresponding quotient module are the same as those of the parabolic Verma module $M(\lambda,J)$.  
\begin{proof}[\textnormal{\textbf{Proof of Proposition \ref{P2.7}:}}] Let $\lambda\in\mathfrak{h}^*$, $M(\lambda)\twoheadrightarrow{ } V$ and $V=\frac{M(\lambda)}{N_V}$. Suppose $\wt V=\wt M(\lambda)$. Then observe that $\mu'\in\wt M(\lambda)\setminus \wt N_V$ whenever the Dynkin subdiagram on $\supp(\lambda-\mu')$ has no edge, as $\dim(M(\lambda))_{\mu'}=1$ for such $\mu'$. By the definition of $N(\lambda)$, this proves that $N_V\subset N(\lambda)$ whenever $\wt V=\wt M(\lambda)$, which proves the forward implication in the statement of the proposition.\\ 
Now, we prove by induction on $\height(\lambda-\mu)\geq 1$ that $\mu\in\wt M(\lambda)$ implies that $\mu\in \wt \frac{M(\lambda)}{N(\lambda)}$. This proves that $\wt \frac{M(\lambda)}{N(\lambda)}=\wt M(\lambda)$, which immediately proves the reverse implication.\\
Base step: Let $\mu=\lambda-\alpha_t$ for some $t\in\mathcal{I}$. Check by the definition of $N(\lambda)$ that $\mu\notin \wt N(\lambda)$ as the Dynkin subdiagram on $\{t\}$ has no edge. So, $\lambda-\alpha_t\in\wt \frac{M(\lambda)}{N(\lambda)}$.\\
Induction step: Let $\mu\in \wt M(\lambda)$ be such that $\height(\lambda-\mu)>1$. Write $\mu=\lambda-\sum\limits_{l\in \supp(\lambda-\mu)}c_l\alpha_l$ for some $c_l\in\mathbb{Z}_{> 0}$. If the Dynkin subdiagram on $\supp(\lambda-\mu)\cap J_{\lambda}$ has no edge (note that this includes the case where $\supp(\lambda-\mu)\cap J_{\lambda}=\emptyset$), then clearly\[\lambda-\sum\limits_{l\in \supp(\lambda-\mu)\cap J_{\lambda}}c_l\alpha_l \notin \wt N(\lambda)\text{ by the definition of } N(\lambda),\text{ and so } \lambda-\sum\limits_{l\in \supp(\lambda-\mu)\cap J_{\lambda}}c_l\alpha_l\in \wt \frac{M(\lambda)}{N(\lambda)}.\]
Now by Lemma \ref{L4.2}, $\mu\in \wt \frac{M(\lambda)}{N(\lambda)}$. So, we assume for the rest of the proof that the Dynkin subdiagram on $\supp(\lambda-\mu)\cap J_{\lambda}$ has at least one edge. Let $i,j\in \supp(\lambda-\mu)\cap J_{\lambda}$ be two nodes such that there is at least one edge between them in the Dynkin diagram. Consider the numbers $c_i$ and $c_j$. Assume without loss of generality that $c_i\geq c_j$. As $c_j>0$, $\height(\lambda-\mu-c_j\alpha_j)<\height(\lambda-\mu)$, so by the induction hypothesis $\mu+c_j\alpha_j\in \wt \frac{M(\lambda)}{N(\lambda)}$. Now one verifies the following implication.
\[
 \langle\lambda,\alpha_j^{\vee}\rangle\geq 0,\text{ }\langle\alpha_l,\alpha_j^{\vee}\rangle\leq 0\text{ }\forall\text{ }l\in\supp(\lambda-\mu)\setminus\{i,j\}\text{ and }\langle\alpha_i,\alpha_j^{\vee}\rangle\leq -1\implies \langle\mu+c_j\alpha_j,\alpha_j^{\vee}\rangle\geq c_i\geq c_j.\] 
By the $\mathfrak{sl}_{\alpha_j}$-action on $M(\lambda)_{\mu+c_j\alpha_j}$ and by $\mathfrak{sl}_2$-theory, it can be easily seen that $[s_j(\mu+c_j\alpha_j), \mu+c_j\alpha_j]\subset \wt\frac{M(\lambda)}{N(\lambda)}$. Now, as $\langle\mu+c_j\alpha_j,\alpha_j^{\vee}\rangle\geq c_j$, $\mu\in [s_j(\mu+c_j\alpha_j), \mu+c_j\alpha_j]$, and therefore $\mu\in\wt\frac{M(\lambda)}{N(\lambda)}$. Hence, the proof of the proposition is complete.
\end{proof}
We are now able to show Theorem \ref{thmB}.
\begin{proof}[\textnormal{\textbf{Proof of Theorem \ref{thmB}}}] Let $\lambda\in\mathfrak{h}^*$, $\emptyset\neq J\subset J_{\lambda}$, $M(\lambda)\twoheadrightarrow{ } V$ and $V=\frac{M(\lambda)}{N_V}$. For $M(\lambda)\twoheadrightarrow{ }V'$, it can be easily seen by the definition of $M(\lambda,J)$ and by looking at the integrability of $V'$ that $M(\lambda,J)\twoheadrightarrow{ }V'$ if and only if $\wt V'\subset \wt M(\lambda,J)$. Recall, $\wt M(\lambda,J_{\lambda})=\wt L(\lambda)$ by equation \eqref{E2.9}, and so $\wt M(\lambda,J_{\lambda})=\wt V'$ for any $M(\lambda,J_{\lambda})\twoheadrightarrow{ }V'$. Observe that the previous two lines prove the theorem when $J=J_{\lambda}$. So, for the rest of the proof we assume that $J\subsetneq J_{\lambda}$, which implies $J^c\neq \emptyset$. \\
Now, suppose $\wt V=\wt M(\lambda,J)$. Then observe that $\mu'\in\wt V$ and $\mu'\notin \wt N_V$ whenever the Dynkin subdiagram on $\supp(\lambda-\mu')$ has no edge and $\supp(\lambda-\mu')\subset J^c$. The previous holds true as $\dim(M(\lambda))_{\mu'}=1$ for such $\mu'$, and as $\lambda-\mathbb{Z}_{\geq 0}\Pi_{J^c}\subset\wt M(\lambda,J)$ by equation \eqref{E2.7}. By the definition of $N(\lambda,J)$, this proves that $N_V\subset N(\lambda,J)$ whenever $\wt V=\wt M(\lambda,J)$, which proves the forward implication in the statement of the theorem.\\
Conversely, we begin by recalling the integrable slice decomposition for $\wt M(\lambda,J)$.
\begin{equation}\label{E4.1}
\wt M(\lambda,J)=\bigsqcup\limits_{\xi\in\mathbb{Z}_{\geq 0}\Pi_{J^c}}\wt L_J(\lambda-\xi).
\end{equation}
Note that as $\frac{M(\lambda)}{N(\lambda,J)}$ is $\mathfrak{g}_J$-integrable, $M(\lambda,J)\twoheadrightarrow{ }\frac{M(\lambda)}{N(\lambda,J)}$, and therefore $\wt \frac{M(\lambda)}{N(\lambda,J)}\subseteq \wt M(\lambda,J)$. We now show that $\wt \frac{M(\lambda)}{N(\lambda,J)}$ contains every weight on the right hand side of the above equation. This proves that $\wt \frac{M(\lambda)}{N(\lambda,J)}=\wt M(\lambda,J)$, which immediately proves the reverse implication in the statement of the theorem, as $M(\lambda,J)\twoheadrightarrow{ }V$ whenever $\sum_{j\in J}U(\mathfrak{g})f_{j}^{\langle\lambda,\alpha_j^{\vee}\rangle+1}m_{\lambda}\subset N_V$. For this, we consider $\big(\wt \frac{M(\lambda)}{N(\lambda,J)}\big)\cap(\lambda-\mathbb{Z}_{\geq 0}\Pi_{J^c})$, and the $\mathfrak{g}_{J^c}$-module $N(\lambda,J)\cap  U(\mathfrak{g}_{J^c})m_{\lambda}$.

Observe that $U(\mathfrak{g}_{J^c})m_{\lambda}$ is isomorphic to the Verma module over $\mathfrak{g}_{J^c}$ with highest weight $\lambda$ (or $\lambda\big|_{\mathfrak{h}_{J^c}^*}$). Now analogous to $N(\lambda)$, let $N_{J^c}(\lambda)$ be the largest submodule of $U(\mathfrak{g}_{J^c})m_{\lambda}$ with respect to the property (P2) (over $\mathfrak{g}_{J^c}$). By Proposition \ref{P2.7} (with $\mathfrak{g}_{J^c}$ in place of $\mathfrak{g}$), we have $\wt\frac{U(\mathfrak{g}_{J^c})m_{\lambda}}{N_{J^c}(\lambda)}=\lambda-\mathbb{Z}_{\geq0}\Pi_{J^c}$. Now, note by the definition of $N(\lambda,J)$ that if $\mu\in(\wt N(\lambda,J))\cap (\lambda-\mathbb{Z}_{\geq 0}\Pi_{J^c})$, then the Dynkin subdiagram on $\supp(\lambda-\mu)$ has at least one edge. So, the $\mathfrak{g}_{J^c}$-submodule $N(\lambda,J)\cap U(\mathfrak{g}_{J^c})m_{\lambda}$ of $U(\mathfrak{g}_{J^c})m_{\lambda}$ satisfies the property (P2) (over $\mathfrak{g}_{J^c}$). Therefore, \begin{equation}\label{E4.2}
N(\lambda,J)\cap U(\mathfrak{g}_{J^c})m_{\lambda}\subset N_{J^c}(\lambda),\text{ so that }N(\lambda,J)_{\bar{\mu}}\subset N_{J^c}(\lambda)_{\bar{\mu}}\text{ }\forall\text{ }\bar{\mu}\in\lambda-\mathbb{Z}_{\geq 0}\Pi_{J^c}.\end{equation}
Recall for any submodule $K$ of $M(\lambda)$ and any $\mu''\in\wt M(\lambda)$ that $\big(\frac{M(\lambda)}{K}\big)_{\mu''}$ and $\frac{M(\lambda)_{\mu''}}{K_{\mu''}}$ are isomorphic as vector spaces; $K$ being a weight module. Now, in view of equation \eqref{E4.2} and the previous line, for $\bar{\mu}\in\lambda-\mathbb{Z}_{\geq 0}\Pi_{J^c}$ one observes the following implications.  
\[ \bigg[\frac{U(\mathfrak{g}_{J^c})m_{\lambda}}{N_{J^c}(\lambda)}\bigg]_{\bar{\mu}}\neq\{0\}\implies\frac{\big( U(\mathfrak{g}_{J^c})m_{\lambda}\big)_{\bar{\mu}}}{N_{J^c}(\lambda)_{\bar{\mu}}}\neq\{0\}\implies \bigg[\frac{M(\lambda)}{N(\lambda,J)}\bigg]_{\bar{\mu}}\simeq \bigg[\frac{U(\mathfrak{g}_{J^c})m_{\lambda}}{N(\lambda,J)\cap U(\mathfrak{g}_{J^c})m_{\lambda}}\bigg]_{\bar{\mu}}\neq \{0\}.\]
Hence, $\wt \frac{M(\lambda)}{N(\lambda, J)}\supset\lambda-\mathbb{Z}_{\geq 0}\Pi_{J^c}$. Now, note for any $\xi\in\mathbb{Z}_{\geq0}\Pi_{J^c}$ that every non-zero vector in the weight space $\big(\frac{M(\lambda)}{N(\lambda, J)}\big)_{\lambda-\xi}$ is a maximal vector for the action of $\mathfrak{g}_J$. So,
\[\wt\frac{M(\lambda)}{N(\lambda,J)}\supset \wt\Big[U(\mathfrak{g}_J)\Big(\frac{M(\lambda)}{N(\lambda,J)}\Big)_{\lambda-\xi}\Big]\supset \wt L_J(\lambda-\xi)\text{ } \forall\text{ } \xi\in\mathbb{Z}_{\geq0}\Pi_{J^c},\]
and therefore, $\wt\frac{M(\lambda)}{N(\lambda, J)}\supset\bigsqcup_{\xi\in\mathbb{Z}_{\geq 0}\Pi_{J^c}}\wt L_J(\lambda-\xi)$. Hence, the proof of Theorem \ref{thmB} is complete.
\end{proof}
\subsection{Proof of Theorem \ref{thmC}} We begin this subsection with the proof of our next main result. 
\begin{proof}[\textnormal{\textbf{Proof of Theorem \ref{thmC}}}]
Let $\lambda\in\mathfrak{h}^*$ and $M(\lambda)\twoheadrightarrow{ }V$. We begin by proving equation \eqref{E2.11}. Note that the result is obvious in the extreme cases where (i) $J_{\lambda}=\emptyset$ and (ii) $J_{\lambda}=\mathcal{I}$; as $\lambda-\mathbb{Z}_{\geq 0}\Pi=\wt L(\lambda)\subset \wt V$, and respectively $\mathbb{Z}_{\geq 0}\Delta_{J^c_{\lambda},1}=\{0\}$, in the cases (i), and respectively (ii). So, we assume throughout the proof that $\emptyset\neq J_{\lambda}\subsetneq \mathcal{I}$. We prove for $\gamma_1,\ldots, \gamma_n\in\Delta_{J_{\lambda}^c,1}$, $n\in\mathbb{N}$, that 
\[ \wt_{J_{\lambda}}V-\sum\limits_{t=1}^n\gamma_t\subset \wt V\quad
\text{by induction on }\text{ } \height_{J_{\lambda}}\Big(\sum_{t=1}^n\gamma_t\Big)\geq 0.\]
This proves that $\wt_{J_{\lambda}}V-\mathbb{Z}_{\geq 0}\Delta_{J_{\lambda}^c,1}\subset\wt V$. The reverse inclusion in the previous line holds true by Corollary \ref{C3.2} (b), and hence the proof of equation \eqref{E2.11} will be complete.\\
Base step: $\height_{J_{\lambda}}(\sum_{t=1}^n\gamma_t)=0$, and so $\gamma_t\in\Pi_{J_{\lambda}^c}$ $\forall t$. Therefore, the result follows by Lemma \ref{L4.2}.\\
Induction step: Let $\mu\in\wt_{J_{\lambda}}V$, and $\gamma_1,\ldots,\gamma_n\in\Delta_{J_{\lambda}^c,1}$ such that $\height_{J_{\lambda}}(\sum_{t=1}^n\gamma_t)$ $>0$. The result follows once we prove that $\mu-\sum_{t=1}^n\gamma_t\in\wt V$. Without loss of generality we will assume that $\height_{J_{\lambda}}(\gamma_1)>0$. Pick $j\in J_{\lambda}$ such that $\gamma_1-\alpha_j\in\Delta_{J_{\lambda}^c,1}$, which exists by Lemma \ref{L3.4}. Note by the induction hypothesis applied to $\gamma_1-\alpha_j+\sum_{t= 2}^n\gamma_t$ that $\wt_{J_{\lambda}}V-(\gamma_1-\alpha_j+\sum_{t= 2}^n\gamma_t)\subset \wt V$ (when $n=1$, treat the term $\sum_{t=2}^n\gamma_t$ as 0 throughout the proof). Now, if $\mu-\alpha_j\in\wt_{J_{\lambda}} V$, then by the previous sentence $\mu-\sum_{t=1}^n\gamma_t=\mu-\alpha_j-(\gamma_1-\alpha_j+\sum_{t= 2}^n\gamma_t)\in\wt V$, and therefore we are done.\\
So, we assume that $\mu-\alpha_j\notin\wt V$. Note that this forces $f_jV_{\mu}=\{0\}$. So, now one can easily check that the $\mathfrak{g}_{\{j\}}$-module $U(\mathfrak{g}_{\{j\}})V_{\mu}$ is finite-dimensional by $\mathfrak{sl}_2$-theory, and hence is $\mathfrak{g}_{\{j\}}$-integrable. Define $\Tilde{\mu}$ and $\Tilde{\gamma_t}$ as follows.
    \begin{align*}
    \begin{aligned}
    \Tilde{\mu}=
    \begin{cases}
    \mu&\text{if } \langle\mu,\alpha_j^{\vee}\rangle\geq 0\\
    s_j\mu &\text{if }\langle\mu,\alpha_j^{\vee}\rangle<0
    \end{cases}
    \end{aligned}\quad\quad
    \begin{aligned}
    \Tilde{\gamma_1}=
    \begin{cases}
    s_j\gamma_1&\text{if }\langle\gamma_1,\alpha_j^{\vee}\rangle>0\\
    \gamma_1-\alpha_j&\text{if }\langle\gamma_1,\alpha_j^{\vee}\rangle\leq0
    \end{cases}
    \end{aligned}\quad\quad
    \begin{aligned}[t]
    &\Tilde{\gamma_t}=
    \begin{cases}
    \gamma_t &\text{if }\langle\gamma_t,\alpha_j^{\vee}\rangle\leq 0\\
    s_j\gamma_t &\text{if }\langle\gamma_t,\alpha_j^{\vee}\rangle> 0 
    \end{cases}\\
    &\text{ }\text{when }n\geq 2\text{ and }2\leq t\leq n.
    \end{aligned}
\end{align*}
One now immediately checks the following about $\tilde{\mu}$ and $\tilde{\gamma_t}$.
\begin{itemize}
    \item[(a)] $s_j\mu\in\wt V$, by the $\mathfrak{g}_{\{j\}}$-integrability of $U(\mathfrak{g}_{\{j\}})V_{\mu}$. Therefore, $\Tilde{\mu}\in\wt V$.
    \item[(b)] $\Tilde{\gamma_t}\preceq \gamma_t\in\Delta_{J_{\lambda}^c,1}$ $\forall$ $t$, and $\Tilde{\gamma_1}\precneqq\gamma_1$. Therefore, $\height_{J_{\lambda}}(\sum_{t=1}^n\Tilde{\gamma_t})<\height_{J_{\lambda}}(\sum_{t=1}^n\gamma_t)$.
    \item[(c)] $\langle\Tilde{\mu},\alpha_j^{\vee}\rangle\geq 0$, $\langle\Tilde{\gamma_1},\alpha_j^{\vee}\rangle<0$, and $\langle\Tilde{\gamma_t},\alpha_j^{\vee}\rangle\leq 0$ $\forall$ $t\geq 2$. 
    \item[(d)] $\mu\in [s_j\Tilde{\mu},\Tilde{\mu}]$, and $\gamma_t\in[\Tilde{\gamma_t},s_j\Tilde{\gamma_t}]$ $\forall$ $t$.
\end{itemize}
Consider $\Tilde{\mu}-\sum_{t=1}^n\Tilde{\gamma_t}$. In view of points (a) and (b), the induction hypothesis yields $\Tilde{\mu}-\sum_{t=1}^n\tilde{\gamma_t}\in\wt V$. By the $\mathfrak{g}_{\{j\}}$-action on $V_{\Tilde{\mu}-\sum_{t=1}^n\Tilde{\gamma_t}}$, $\mathfrak{sl}_2$-theory and by point (d), it can be checked that \[\mu-\sum_{t=1}^n\gamma_t\in \Big[s_j\Big(\Tilde{\mu}-\sum_{t=1}^n\Tilde{\gamma_t}\Big),\text{ }\Tilde{\mu}-\sum_{t=1}^n\Tilde{\gamma_t}\Big]\subset \wt V.\]
We finally have $\mu-\sum_{t=1}^n\gamma_t\in\wt V$, completing the proof of equation \eqref{E2.11}.

    Observe that $J_{\lambda}$ played no role in the proof of equation \eqref{E2.11}, except in the base step of the proof. The proof in the base step, which shows that $\wt_{J_{\lambda}}V-\mathbb{Z}_{\geq 0}\Pi_{J^c_{\lambda}}\subset\wt V$, just followed using Lemma \ref{L4.2}. Thus, for any $J\subset\mathcal{I}$ the above proof of equation \eqref{E2.11} (with $J$ in the place of $J_{\lambda}$) proves equation \eqref{E2.12}, which says
    \[\wt_JV-\mathbb{Z}_{\geq 0}\Pi_{J^c}\subset \wt V\quad \iff \quad \wt V=\wt_JV-\mathbb{Z}_{\geq 0}\Delta_{J^c,1}.\]
\end{proof}
We conclude this subsection with the following remark which relates the Minkowski difference formula for $\wt V$ in equation \eqref{E2.11} with (similar) Minkowski difference formulas that might exist inside $\wt_{J_{\lambda}}V$.
\begin{remark}
Let $\lambda\in\mathfrak{h}^*$ and $M(\lambda)\twoheadrightarrow{ }V$. Suppose there is a subset $J\subset J_{\lambda}$ such that the set of weights of the $\mathfrak{g}_{J_{\lambda}}$-module $U(\mathfrak{g}_{J_{\lambda}})V_{\lambda}$, which is precisely $\wt_{J_{\lambda}}V$, has the Minkowski difference formula:
\[
\wt_{J_{\lambda}}V= \wt_{J}V-\mathbb{Z}_{\geq 0}\big(\Delta_{J^c,1}\cap \Delta_{J_{\lambda}}\big).
\]
Then by the parabolic-PSP it can be easily seen that $\Delta_{J^c_{\lambda},1}\subset \mathbb{Z}_{\geq0}\Delta_{J^c,1}$, and this immediately results in the following Minkowski difference decomposition for $\wt V$: 
\[
\wt V= \wt_{J}V-\mathbb{Z}_{\geq 0}\Delta_{J^c,1}=\wt_{J}V-\mathbb{Z}_{\geq 0}(\Delta^+\setminus\Delta_J^+).
\]
\end{remark}
\section{Moving between comparable weights of representations}
			\subsection{Proof of Theorem \ref{thmD}}
		    In this subsection, we prove Theorem \ref{thmD} by proving (\ref{thmD}1) and (\ref{thmD}2) separately for a Kac--Moody algebra $\mathfrak{g}=\mathfrak{g}(A)$.\\
		    Throughout this subsection, we use without further mention the following fact.
		    \begin{fact}
		    Let $\mathfrak{g}$ be a Kac--Moody algebra and $M$ be a weight module of $\mathfrak{g}$. Fix a weight $\mu\in\wt M$ and a real root $\alpha$. Suppose $M$ is $\mathfrak{sl}_{\alpha}$-integrable, where $\mathfrak{sl}_{\alpha}=\mathfrak{g}_{-\alpha}\oplus\mathbb{C}\alpha^{\vee}\oplus\mathfrak{g}_{\alpha}$. Then by $\mathfrak{sl}_2$-theory \[\langle\mu,\alpha^{\vee}\rangle>0 \implies \mu-\alpha\in [s_{\alpha}\mu,\mu]\subset\wt M,\quad
		   \text{similarly, } \langle\mu,\alpha^{\vee}\rangle<0 \implies \mu+\alpha\in [\mu,s_{\alpha}\mu]\subset\wt M.\]
		   \end{fact}
	  \begin{proof}[\normalfont{\textbf{Proof of (\ref{thmD}1) (for adjoint representation)}}]
			Relabel $\mu_0$ and $\mu$ by $\beta_0$ and $\beta$ respectively for convenience. We prove (\ref{thmD}1) by induction on $\height(\beta -\beta_0)\geq 1$. In the base step $\height(\beta -\beta_0)=1$, (\ref{thmD}1) is trivial.\\
			Induction step: Assume $\height(\beta-\beta_0)>1$, and observe that if there exists an $\tilde{\alpha}\in \pm\Pi\sqcup \{0\}$ such that $\beta_0\precneqq \tilde{\alpha}\precneqq\beta$, then we will be immediately done by the induction hypothesis applied to $\beta_0\precneqq \tilde{\alpha}$ and $\tilde{\alpha}\precneqq\beta$.
			So, we assume throughout the proof that both $\beta_0$ and $\beta$ are positive roots (by symmetry, the proof for the leftover case where both $\beta_0$ and $\beta$ are negative just follows from this). We show in steps that there exists an $\alpha \in \Pi$ such that $\alpha \prec \beta-\beta_0$, and either $\beta -\alpha$ or $\beta_0 +\alpha$ is a root. The induction hypothesis then completes the proof.\\
			Write $\beta =\sum\limits_{i\in I}a_i \alpha_i$ and $\beta_0 =\sum\limits_{i'\in I_0}b_{i'} \alpha_{i'}$ for some $I_0 \subseteq I\subseteq \mathcal{I}$, and positive integers $a_i$ and $b_{i'}\leq a_{i'}$, $\forall$ $i\in I$ and $i'\in I_0$. We proceed in several steps below.\newline\newline
			\textbf{Step 1:} If $I_0 \subsetneq I$, then, as the Dynkin subdiagram on $I=\supp(\beta)$ is connected, we must have a simple root $\alpha' \in \Pi_{I}\setminus \Pi_{I_0}$ such that $\langle\beta_0 , \alpha'^{\vee}\rangle<0\implies \beta_0 +\alpha'\prec \beta\in\Delta_I$.\\ So, we assume now that $I_0 =I$. Define $J:=\supp(\beta-\beta_0)\subset I$ and write $\beta-\beta_0 =\sum_{j\in J}(a_j -b_j)\alpha_j$. Note that $b_j <a_j$ $\forall$ $j\in J$.\\
			If $J=I$, then by the PSP there exists an $\alpha''\in\Pi$ such that $\beta -\alpha'' \in \Delta$, and we will be done.\\
			So, we also assume now that $J\subsetneq I$. Write $\beta =\sum\limits_{i\in I\setminus J}a_i \alpha_i +\sum\limits_{j\in J}a_j \alpha_j$ and $\beta_0 =\sum\limits_{i\in I\setminus J}a_i \alpha_i +\sum\limits_{j\in J}b_j \alpha_j$.\newline\newline
			\textbf{Step 2:} Consider the decomposition of the submatrix $A_{J\times J}$ of $A$ into indecomposable blocks. Pick an indecomposable block $A_{K\times K}$ of $A_{J\times J}$ for some $K\subset J$.\\
			If $A_{K\times K}$ (equivalently $A^t_{K\times K}$) is of finite type, then by \cite[Theorem 4.3]{Kac} there exists a vector \[ Y=\sum_{k\in K}p_k\alpha_k^{\vee}\in\mathfrak{h}\text{ such that }p_k\in\mathbb{R}_{>0}\text{ and }\langle\alpha_k,Y\rangle>0 \text{ }\forall\text{ } k\in K.\] 
			Check that $\langle\alpha_j,Y\rangle\geq 0$ $\forall$ $j\in J$, and so $\langle\beta-\beta_0,Y\rangle>0$. By the previous line, there exists some $k_0\in K\subset J$ such that $\langle\beta -\beta_0 ,\alpha_{k_0}^{\vee}\rangle>0$. This implies either $\langle\beta, \alpha_{k_0}^{\vee}\rangle>0$ or $\langle\beta_0 ,\alpha_{k_0}^{\vee}\rangle<0$, yielding $\beta -\alpha_{k_0} \succ\beta_0 \in\Delta$ or respectively $\beta_0 +\alpha_{k_0} \prec \beta \in\Delta$, as required.\\
			So, we assume now that $A_{K\times K}$ is not of finite type. This means $A_{K\times K}$ (equivalently $A^t_{K\times K}$) must be of affine or indefinite type, and once again by \cite[Theorem 4.3]{Kac} there exists a vector
			\[ X=\sum_{k\in K}q_k\alpha_k^{\vee}\in\mathfrak{h}\text{ such that }q_k\in\mathbb{R}_{>0}\text{ and }\langle\alpha_k,X\rangle\leq0\text{ } \forall\text{ } k\in K.\]
			Check that $X$ also satisfies  $\langle\alpha_j,X\rangle\leq 0$ $\forall$ $j\in J$. Note that (i) the Dynkin subdiagram on $I$ is connected, and (ii) whenever $K\subsetneq J$, the subdiagrams on $K$ and $J\setminus K$ are disconnected. Thus, in view of the previous line, there must exist a pair of nodes $t_1\in K$ and $t_2\in I\setminus J$ such that $t_1$ and $t_2$ are connected by at least one edge in the Dynkin diagram, i.e. $\langle\alpha_{t_2},\alpha_{t_1}^{\vee}\rangle<0$. This yields
			\begin{align*}
			\langle\sum\limits_{i\in I\setminus J}a_i \alpha_i ,X\rangle<0&
			\implies \langle\sum\limits_{i\in I\setminus J}a_i \alpha_i + \sum\limits_{j\in J}b_j \alpha_j ,X\rangle= \langle\beta_0 ,X\rangle<0\\
			&\implies \exists\text{ }k'\in K\subset J \text{ such that } \langle\beta_0 ,\alpha_{k'}^{\vee}\rangle<0\implies \beta_0 +\alpha_{k'} \in \Delta.
			\end{align*}
			($\langle\sum_{i\in I\setminus J}a_i \alpha_i ,X\rangle<0$ as $c_i>0$ $\forall$ $i\in I\setminus J$.) Hence, the proof of (\ref{thmD}1) is complete.
			\end{proof}
			Before proving (\ref{thmD}2), we prove the following corollary of (\ref{thmD}1) which is similar to Corollary \ref{C3.5}.\\Recall from equation \eqref{E2.2} that $\Delta_{\alpha ,J}:=\{\beta \in \Delta^+$ $|$ $\supp(\beta -\alpha)\subset J\}$ for $\alpha \in \Delta^+$ and $J\subset \mathcal{I}$.		\begin{cor}\label{C5.1}
				Let $\mathfrak{g}$ be a Kac--Moody algebra. Fix a real root $\alpha \in \Delta^+$ such that $\supp(\alpha)\subsetneq \mathcal{I}$, and let $J$ be a non-empty subset of $\mathcal{I}\setminus \supp(\alpha)$. Suppose $\beta \in \Delta_{\alpha,J}$ is such that $\beta -\alpha \in \Delta_J^+$. Then there exists a sequence of roots $\beta_i \in \Delta_{\alpha,J}$, $0\leq i\leq n=\height(\beta-\alpha)$, such that
				\[
				\alpha=\beta_0 \prec \cdots \prec\beta_i \prec \cdots \prec \beta_n =\beta \in \Delta_{\alpha ,J}\text{ }\forall i\quad \text{and also }\beta_i -\alpha\in \Delta^+_J\text{ }\forall\text{ }i\geq 1.
				\]
				\end{cor}
				\begin{proof}
				We prove this Corollary by induction on $\height(\beta-\alpha)\geq 1$. In the base step $\height(\beta-\alpha)=1$, there is nothing to prove. Induction step: Assume $\height(\beta-\alpha)>1$, pick $j\in \supp(\beta-\alpha)\subset J$ such that $\langle\alpha,\alpha_j^{\vee}\rangle<0$ (such a node $j$ exists as the Dynkin subdiagram on $\supp(\beta)$ is connected). By (\ref{thmD}1) applied to $\alpha_j \prec \beta-\alpha\in\Delta^+$, we get a root $\beta'\in \Delta_{J}^+$ such that $(\beta-\alpha)-\beta' \in \Pi_{J}$ and $\alpha_j \prec \beta'$. Observe that we must have $\langle\beta' ,\alpha^{\vee}\rangle<0$ (as $j\in\supp(\beta')$). This implies $\beta'+\alpha\prec\beta \in \Delta_{\alpha,J}$, and the induction hypothesis applied to $\beta'+\alpha$ now finishes the proof.
			\end{proof} 
		   \begin{proof}[\normalfont{\textbf{Proof of (\ref{thmD}2) (for submodules of $M(\lambda,J)$)}}]
		    Let $V$ be a submodule of $M(\lambda,J)$ and $\mu_0\prec\mu\in\wt V$. We prove (\ref{thmD}2) by induction on $\height(\mu-\mu_0)\geq 1$. In the base step $\height(\mu-\mu_0)=1$, (\ref{thmD}2) is trivial.\\
		    Induction step: Assume $\height(\mu-\mu_0)>1$, and let $J_1:=\supp(\mu-\mu_0)$, $I:=\supp(\lambda-\mu)\setminus J_1$ and $m_{\lambda}$ span $M(\lambda,J)_{\lambda}$. When $I\neq\emptyset$, write $\mu=\lambda-\sum\limits_{i\in I}c_i\alpha_i-\sum\limits_{j\in J_1}c'_j\alpha_j$ for some $c_i\in\mathbb{Z}_{>0}$ and $c'_j\in\mathbb{Z}_{\geq0}$. Firstly, note:
		    \begin{itemize}
		    \item[(a)] Let $\mathfrak{p}_J$ be the parabolic Lie subalgebra of $\mathfrak{g}$ corresponding to $J\subset J_{\lambda}\subset\mathcal{I}$, and let $L^{\max}_J(\lambda)$ be the largest integrable highest weight $\mathfrak{p}_J$-module with highest weight $\lambda$. Then $M(\lambda,J)\simeq U(\bigoplus\limits_{\beta\in\Delta^{-} \setminus\Delta_J^-}\mathfrak{g}_{\beta})\otimes L^{\max}_J(\lambda)$. Note that $\bigoplus\limits_{\beta\in\Delta^{-} \setminus\Delta_J^-}\mathfrak{g}_{\beta}=\bigoplus\limits_{n\in\mathbb{Z}_{<0}}\mathfrak{g}_{J^c,n}$, and $\bigoplus\limits_{n\in\mathbb{Z}_{<0}}\mathfrak{g}_{J^c,n}$ can be easily checked to be a Lie algebra.
		    \item[(b)] $M(\lambda,J)$ is torsion free over $U(\bigoplus\limits_{\beta\in\Delta^-\setminus\Delta^-_J}\mathfrak{g}_{\beta})$, and so is $V$.
		    \end{itemize}
		    We show via several cases that there exists a simple root $\alpha\prec\mu-\mu_0$ such that either $\mu-\alpha$ or $\mu_0+\alpha$ belongs to $\wt V$. The induction hypothesis then completes the proof.\\
		    If $J_1\cap J^c\neq \emptyset$, then for $i\in J_1\cap J^c$ (a) and (b) yield $f_iV_{\mu}\neq \{0\}$, which implies $\mu-\alpha_i\succ\mu_0\in\wt V$.\\
            So, we assume for the rest of the proof that $J_1\subset J$. If $\langle\mu-\mu_0,\alpha_{j'}^{\vee}\rangle>0$ for some $j'\in J_1$, then we must have $\langle\mu,\alpha_{j'}^{\vee}\rangle>0$ or $\langle\mu_0,\alpha_{j'}^{\vee}\rangle<0$, in which case we are done by the existence of the injective mappings $V_{\mu}\xhookrightarrow{f_{j'}}V_{\mu-\alpha_{j'}}$ or respectively $V_{\mu_0}\xhookrightarrow{e_{j'}}V_{\mu_0+\alpha_{j'}}$ (by \cite[Proposition 3.6]{Kac} as $M(\lambda,J)$, and hence $V$, is $\mathfrak{g}_J$-integrable).\\
So, we also assume now that $\langle\mu-\mu_0,\alpha_j^{\vee}\rangle\leq 0$ $\forall$ $j\in J_1$. Fix an indecomposable block $A_{K\times K}$ of $A_{J_1\times J_1}$ for some $K\subset J_1$, and note that (by the assumption in the previous line and \cite[Theorem 4.3]{Kac}) $A_{K\times K}$ (equivalently $A^t_{K\times K}$) must be of either affine or indefinite type. Thus, we have a vector \[X=\sum_{k\in K}p_k\alpha_k^{\vee}\in\mathfrak{h}\text{ such that }p_k\in\mathbb{R}_{>0}\text{ and }\langle\alpha_k,X\rangle\leq0\text{ }\forall\text{ }k\in K.\] Note that $X$ also satisfies $\langle\alpha_j,X\rangle\leq 0$ $\forall$ $j\in J_1$. We now proceed in two cases below. We show in both the cases that $\langle\mu,X\rangle>0$. This implies that $\langle\mu,\alpha_{j_0}^{\vee}\rangle>0$ for some $j_0\in K\subset J_1$, and then we will be done as above.\newline\newline
1) $I\neq \emptyset$, and also $J_1$ and $I$ are connected by at least one edge in the Dynkin diagram: Firstly, without loss of generality, we may assume that the subset $K$ of $J_1$ we started with has the property that the Dynkin subdiagram on $K\sqcup I$ is connected. Thus, $\langle\sum_{i\in I}c_i\alpha_i,X\rangle<0$ as $c_i>0$ $\forall$ $i\in I$. Note that $\langle\lambda,\alpha_j^{\vee}\rangle\geq0$ $\forall$ $j\in J_1\subset J$, and in particular, $\langle\lambda,X\rangle\geq 0$. The observation that $\langle\alpha_j,X\rangle\leq0$ $\forall$ $j\in J_1$ and the previous two lines together imply that $\langle\mu,X\rangle=\langle\lambda-\sum\limits_{i\in I}c_i\alpha_i-\sum\limits_{j\in J_1}c'_j\alpha_j,X\rangle>0$, as required.\newline\newline
2) $J_1$ and $I$ are not connected by any edge in the Dynkin diagram (with $I$ possibly empty): Hence, the Dynkin subdiagrams on $K, J_1\setminus K\text{ and }I$ are pairwise disconnected. So, any vector in $V_{\mu_0}$ can be expressed in the form $\sum\limits_{p=1}^{l}F_pG_pH_p m_{\lambda}$ for some $F_p\in U(\mathfrak{n}_I^-)$, $G_p\in U(\mathfrak{n}_{J_1\setminus K}^-)$, and $H_p\in U(\mathfrak{n}_{K}^-)$, $1\leq p\leq l$. As $\height_{K}(\lambda-\mu_0)>0$, in the previous line each $H_p$ may be assumed to be a non-zero (and also non-scalar) element of the direct sum of the graded pieces $U(\mathfrak{n}^-_K)_{\nu}$ of $U(\mathfrak{n}^-_K)$ where $\nu\in(\mathbb{Z}_{\leq 0}\Pi_K)\setminus\{0\}$.
Next, if $\langle\lambda,\alpha_{j''}^{\vee}\rangle=0$ for some $j''\in J$, then $f_{j''}m_{\lambda}=0$, as $M(\lambda,J)$ is $\mathfrak{g}_J$-integrable.\\
In view of the previous three lines, observe that $\langle\lambda,X\rangle>0$. Finally, $\langle\alpha_j,X\rangle\leq 0$ $\forall$ $j\in J_1$, $\langle\alpha_i,X\rangle=0$ $\forall$ $i\in I$ (when $I\neq\emptyset$) and the previous line together imply $\langle\mu,X\rangle=\langle\lambda-\sum\limits_{i\in I}c_i\alpha_i-\sum\limits_{j\in J_1}c'_j\alpha_j,X\rangle>0$, once again as required. Hence, the proof of (\ref{thmD}2) is complete.
    \end{proof}
    \begin{remark}\label{R5.2}
    With the notation as in (\ref{thmD}2), observe that in the above proof of (\ref{thmD}2) we only made use of the $\mathfrak{g}_J$-integrability of $M(\lambda,J)$, which is a consequence of the $\mathfrak{g}_J$-integrability of $L^{\max}_J(\lambda)$. Recall that it is not known if $L^{\max}_J(\lambda)$ over a general Kac--Moody algebra $\mathfrak{g}_J$ is simple. Let $L'_J(\lambda)$ be an integrable highest weight $\mathfrak{g}_J$-module with highest weight $\lambda$, and define $N:=U(\bigoplus\limits_{\beta\in\Delta^{-} \setminus\Delta_J^-}\mathfrak{g}_{\beta})\otimes L'_J(\lambda)$. Note that $N$ is a $\mathfrak{g}$-module, and $L'_J(\lambda)$ is a quotient of $L^{\max}_J(\lambda)$. Recall by the explicit description of the set of weights of an integrable highest weight module over a Kac--Moody algebra in \cite[Proposition 11.2]{Kac}, that $\wt L'_J(\lambda)=\wt L^{\max}_J(\lambda)$. This implies $\wt N=\wt M(\lambda,J)$. Thus, the above proof of (\ref{thmD}2) also holds for the module $\mathfrak{g}$-module $N$.
    \end{remark}
    It is natural to ask if Theorem \ref{thmD} holds true ``at the level of weight vectors'', made precise below, and similar to the parabolic-PSP. The following remark provides a negative answer. 
    \begin{remark}\label{R5.3}
    A strengthened version of Theorem \ref{thmD} one would naturally expect at the level of weight vectors does not hold true. More precisely, if $\mu_0\precneqq\mu\in\wt V$, then there need not exist a non-zero $\nu\in V_{\mu}$, and $f_{i_j}\in\mathfrak{g}_{-\alpha_{i_j}}$, $1\leq j\leq n=\height(\mu-\mu_0)$, such that \[\mu-\sum\limits_{j=1}^{n}\alpha_{i_j}=\mu_0\text{ and }0\neq\prod\limits_{j=1}^{n}f_{i_{j}}\nu\in V_{\mu_0}.\]
    This can be easily checked for $V$ the adjoint representation when $\mathfrak{g}=\mathfrak{sl}_{4}(\mathbb{C})$ and $\mathcal{I}=\{1,2,3\}$, where the nodes 1 and 3 are the leaves in the Dynkin diagram, for the pair $(-\alpha_1) \prec \alpha_3\in \wt \mathfrak{g}$. \end{remark}
    It will be interesting to investigate if Theorem \ref{thmD} holds true, even in finite type, for: (1) non-integrable highest weight modules whose set of weights are not those of any parabolic Verma module and (2) more general weight modules which are not necessarily highest weight modules, for example modules in the category $\mathcal{O}$ over $\mathfrak{g}$.\\
    For instance, Corollary \ref{C5.5} below proves Theorem \ref{thmD} for any integrable module (which need not be a highest weight module, or need not have finite-dimensional weight spaces) over semisimple $\mathfrak{g}$.
    
    We end this subsection by extending Theorem \ref{thmD} (\ref{thmD}2), which is for the set of weights of highest weight modules, to any saturated subset $U$ over semisimple $\mathfrak{g}$. Let $\Lambda$ be the weight lattice. Recall that $U\subset\Lambda$ is a saturated subset of $\Lambda$ if for every $\mu\in U$ and $\alpha \in \Pi$, $\mu-t\alpha \in U$ $\forall$ $t= 0,\ldots,\langle\mu,\alpha^{\vee}\rangle$, see \cite[\S 13.4]{Hump}. 
		\begin{lemma}\label{L5.4}
			Let $\mathfrak{g}$ be semisimple and $U$ be a saturated subset, and suppose $\mu_0 \prec\mu\in U$. Then there exists a sequence of weights $\mu_i\in U$, $1\leq i\leq n=\height(\mu-\mu_0), $ such that
			\[
			\mu_0 \prec \cdots \prec \mu_i \prec \cdots \prec \mu_n=\mu\in U \quad\text{and}\quad \mu_i-\mu_{i-1}\in \Pi\text{ }\forall i.
			\]
		\end{lemma}
		\begin{proof}
			We prove the lemma by induction on $\height(\mu -\mu_0)\geq 1$. In the base step $\height(\mu-\mu_0)=1$, there is nothing to prove. To show the induction step, recall that when $\mathfrak{g}$ is semisimple the symmetric invariant (Killing) form on $\mathfrak{h}^*$ is positive definite, i.e. $(x,x)>0$ $\forall$ $0\neq x\in\mathfrak{h}^*$. Now let $\mu -\mu_0 = \sum\limits_{i\in \mathcal{I}}c_i \alpha_i$ for some $c_i\in\mathbb{Z}_{\geq0}$ and consider $0<(\mu -\mu_0 ,\mu -\mu_0 )= (\sum_{i\in \mathcal{I}}c_i \alpha_i ,\mu ) - (\sum_{i\in\mathcal{I}}c_i \alpha_i , \mu_0)$. By the previous line, we must have 
			\begin{align*}
			&\text{either } (\sum_{i\in\mathcal{I}}c_i \alpha_i ,\mu )>0\text{ or } (\sum_{i\in\mathcal{I}}c_i\alpha_i,\mu_0)<0\\
			\implies& \langle\mu,\alpha^{\vee}\rangle>0\text{ or resp. } \langle\mu_0,\alpha^{\vee}\rangle<0 \text{ for some simple root }\alpha \prec\mu-\mu_0\\
			\implies &\mu_0\prec\mu -\alpha\in U \text{ or resp. }\mu_0+\alpha\prec \mu\in U \text{ (by the definition of $U$)}.
			\end{align*}
			The induction hypothesis then completes the proof.  
		\end{proof}
		\begin{cor}\label{C5.5}
		\begin{itemize}
			\item[1)] Theorem \ref{thmD} for finite-dimensional simple modules, proved by S. Kumar, now holds true for any integrable module over semisimple $\mathfrak{g}$.
			\item[2)] Let $\mathfrak{g}$ be a Kac--Moody algebra, $J\subset\mathcal{I}$. Suppose $\mathfrak{g}_J$ is semisimple and $M$ is a finite-dimensional $\mathfrak{g}_J$ (or $\mathfrak{p}_J$)-module. Then:
			\begin{itemize}
			\item[(a)] The $\mathfrak{g}$-module $N:=U(\mathfrak{g})\otimes_{U(\mathfrak{p}_J)}M$, defined analogous to $M(\lambda,J)$, is $\mathfrak{g}_J$-integrable.
			\item[(b)] Theorem \ref{thmD} holds true for $\wt N$.
			\end{itemize}
			\end{itemize}
		\end{cor}
		\begin{proof}
		1) holds by Lemma \ref{L5.4}, as the set of weights of an integrable module is a saturated subset. 2) (a) is a trivial check. 2) (b) can be proved by combining the proofs of (\ref{thmD}2) and Lemma \ref{L5.4}.
		\end{proof}
		  	  \begin{remark}\label{R5.6}
		  	  	Observe that in this paper we have answered Question \ref{Q2} of Khare for a large class of highest weight modules, which contains all simple highest weight modules, over general Kac--Moody algebras. However, Question \ref{Q2} still stands open for general highest weight modules even in the semisimple case. In this paper, we essentially worked with the integrability and the formulas for the set of weights of (parabolic Verma and) simple highest weight modules given by \cite{Dhillon_arXiv}. Working with the integrabilities of the Jordan--H{\"o}lder factors of a highest weight module, it might be possible to extend Theorem \ref{thmD} for all/more highest weight modules at least in the semisimple case. 
		  	  \end{remark} 
	\subsection{Moving between comparable weights in steps of $\Delta_{I,1}$ in the semisimple case}\label{S5.2}
	In this subsection, we discuss the ``parabolic'' generalizations of parts (\ref{thmD}1) and (\ref{thmD}2) of Theorem \ref{thmD}---i.e. moving between comparable roots and weights in steps of $\Delta_{I,1}$. See Proposition \ref{P5.8} below. As a ``warmup'' to our main result in this subsection, we first note the following immediate consequence of Corollary \ref{C3.2} applied to $U(\mathfrak{n}^-)$.
		\begin{prop}\label{P5.7}
			Let $\mathfrak{g}$ be a Kac--Moody algebra, $\lambda\in\mathfrak{h}^*$, $\emptyset\neq I\subset\mathcal{I}$, and $M(\lambda)\twoheadrightarrow V$. Suppose $\mu\prec\lambda\in\wt V$ such that $n=\height_I(\lambda-\mu)>0$. Then there exist sequences of weights $\mu_i$ and $\mu'_i\in\wt V$ such that 
	       \[ \text{a) }\mu=\mu_0\prec\cdots\prec\mu_i\prec\cdots\prec\mu_n\preceq \lambda\text{ and }\mu_i-\mu_{i-1}\in\Delta_{I,1} \text{ }\forall\text{ }1\leq i\leq n.\]
		    \[\text{b) }\mu\preceq \mu'_0\prec\cdots\prec\mu'_i\prec\cdots\prec\mu'_n=\lambda\text{ and }\mu'_i-\mu'_{i-1}\in\Delta_{I,1}\text{ }\forall\text{ }1\leq i\leq n.\]
		  \end{prop}
		  	Notice that Proposition \ref{P5.7} discusses ``moving between comparable weights in steps of $\Delta_{I,1}$'' for arbitrary for $M(\lambda)\twoheadrightarrow V$ over Kac--Moody $\mathfrak{g}$ for the pair $\mu_0\prec\lambda\in \wt V$, and this cannot be generalized to $\mu_0\prec\mu\in\wt V$ for $\mu\precneqq \lambda$. Importantly, note the inequalities in the starting/ending of the chains in the statement of Proposition \ref{P5.7} and also that of all the results of this subsection. These two lines will be justified in parts (ii) and (i) of Remark \ref{R5.11} (1), respectively.
		
		We now state the main result of this subsection, which is also an improvement of the above proposition in the semisimple case.
		\begin{prop}\label{P5.8}
		Let $\mathfrak{g}$ be semisimple, $\emptyset\neq I\subset\mathcal{I}$, and $\beta \prec \beta' \in \wt\mathfrak{g}=\Delta \sqcup \{0\}$ such that $n=\height_{I}(\beta'-\beta) >0$.
		\begin{itemize}
		\item[a)] If either $\beta\succeq 0$, or $\beta\prec 0$ and $\height_{I}(\beta)< 0$, then there exists $\beta_{i}\in \wt\mathfrak{g}$ such that \[\beta=\beta_0\prec \cdots\prec\beta_i\prec\cdots\prec\beta_n\preceq \beta' \in \wt\mathfrak{g}\quad \text{and}\quad \beta_i-\beta_{i-1}\in\Delta_{I,1}\text{ } \forall\text{ } 1\leq i\leq n.\] 
		\item[b)] If either $\beta'\preceq 0$, or $\beta'\succ 0$ and $\height_{I}(\beta')> 0$, then there exists $\beta_{i}' \in \wt\mathfrak{g}$ such that \[\beta\preceq\beta'_0\prec\cdots\prec\beta'_i\prec\cdots\prec\beta_n'=\beta' \in \wt\mathfrak{g}\quad\text{and}\quad \beta'_i-\beta'_{i-1}\in\Delta_{I,1}\text{ }\forall\text{ } 1\leq i \leq n.\]
		\end{itemize} 
	\end{prop}
	We first prove two preliminary lemmas needed in the proof of Proposition \ref{P5.8}. In the rest of this subsection, when $\mathfrak{g}$ is semisimple, we fix $\{e_{\alpha}\in\mathfrak{g}_{\alpha}$ $|$ $e_{\alpha}\neq 0$ and $\alpha\in\Delta\}\sqcup \{\alpha_i^{\vee}\}_{i\in\mathcal{I}}$ to be the Chevalley basis of $\mathfrak{g}$, and also we repeatedly use without mention the following fact. 
		\begin{fact} Let $\mathfrak{g}$ be semisimple,
		$\alpha\in\Pi$ and $\gamma\in\Delta$. If $\alpha+\gamma\in\Delta$, then $[e_{\alpha},\mathfrak{g}_{\gamma}]=\mathfrak{g}_{\alpha+\gamma}$.
		\end{fact}
		Observe that this fact holds true by \cite[Proposition 3.6 (iv)]{Kac} and the basic fact that the root spaces of $\mathfrak{g}$ are one-dimensional.
	\begin{lemma}\label{L5.9}
		Let $\mathfrak{g}$ be semisimple, and $\emptyset \neq I \subset \mathcal{I}$. Suppose $\beta \prec \beta' \in \Delta^+$ such that $\height_{I}(\beta' -\beta)>0$. Then there exists a sequence of roots $\gamma_i \in \Delta_{I,1}$, $1\leq i\leq n=\height_{I}(\beta' -\beta)$, such that
		\[
			\Big[e_{\gamma_n},\big[\cdots,[e_{\gamma_1},e_{\beta}]\cdots\big]\Big] \neq 0\text{  and
			 }\beta \prec \cdots\prec \beta + \sum\limits_{j=1}^{i}\gamma_{j}\prec \cdots \prec \beta +\sum\limits_{j=1}^{n} \gamma_j \preceq \beta' \in \Delta\text{ }\forall i.
		\]
		\end{lemma}
	\begin{proof}
		We prove the lemma by induction on $\height(\beta' -\beta)\geq 1$. Base step: $\height(\beta' -\beta)=1$ forces $\beta' -\beta\in \Pi_{I}$, and we have $[e_{\beta'-\beta}, \mathfrak{g}_{\beta}] = \mathfrak{g}_{\beta'}$ as desired.\\
		Induction step: Assume $\height(\beta'-\beta)>1$, and consider $0<(\beta' -\beta ,\beta' -\beta)=(\beta' -\beta ,\beta')-(\beta' -\beta ,\beta)$, giving the following two cases.\newline\newline
		(1) $(\beta' -\beta,\beta' )>0$: There must exist a simple root $\alpha' \prec \beta'-\beta$ such that $(\alpha' ,\beta')>0$, implying $\beta' -\alpha' \in \Delta^+$.\\
		If $\height_{I}(\beta' -\alpha')>\height_{I}(\beta)$, then apply the induction hypothesis to the pair $\beta\prec \beta'-\alpha'$ to get a chain of roots $\beta\prec\cdots\prec\beta^{(1)}\preceq\beta'-\alpha'\prec\beta'$ as in the statement. Now, we are done if $\alpha'\notin\Pi_I$. Otherwise, by applying the induction hypothesis further to the pair $\beta^{(1)}\prec\beta'$, we will be done.\\
		So, we assume now that $\height_{I}(\beta' -\alpha')=\height_I(\beta)$, which implies that $\height_I(\beta'-\beta)=1$ and $\alpha'\in\Pi_I$. Consider $\beta \prec \beta' -\alpha'\in\Delta$. By Theorem \ref{thmD} (\ref{thmD}1) and \cite[Proposition 3.6]{Kac}, we get a sequence of simple roots $\alpha_1,\ldots,\alpha_m\in \Pi_{I^c}$, $m=\height(\beta' -\beta)-1$, such that 
		\[ y:= \Big[e_{\alpha'},\big[e_{\alpha_m},\big[\cdots,[e_{\alpha_1},e_{\beta}]\cdots\big]\big]\Big] \in \mathfrak{g}_{\beta'}\text{ and }y\neq 0.\]
		If $\beta' -\alpha_j \in \Delta$ for some $j\in [m]$, then we are done by the induction hypothesis applied to the pair $\beta \prec \beta'-\alpha_j$. Else, observe by the Jacobi identity that
		\begin{align*}
			\bigg[e_{\alpha'},\Big[e_{\alpha_m},\big[\cdots,[e_{\alpha_1},e_{\beta}]\cdots\big]\Big]\bigg]=& -\Big[[e_{\alpha_m} ,e_{\alpha'}], \big[\cdots ,[e_{\alpha_1},e_{\beta}]\cdots\big]\Big]\\&+\underbrace{\Big[e_{\alpha_m},\Big[e_{\alpha'},\big[\cdots, [e_{\alpha_1},e_{\beta}]\cdots\big]\Big]\Big]}_{0}\\=& +\Big[\big[e_{\alpha_{m-1}},[e_{\alpha_m}, e_{\alpha'}]\big],\big[\cdots,[e_{\alpha_1},e_{\beta}]\cdots\big]\Big]+0\\
			\vdots\\=& (-1)^{m}\bigg[\underbrace{\Big[e_{\alpha_1},\big[\cdots,[e_{\alpha_m},e_{\alpha'}]\cdots\big]\Big]}_{=:z},e_{\beta}\bigg]+0. 
		\end{align*}
		By the assumption $\beta-\alpha_j\notin\Delta$ $\forall$ $j\in [m]$, all the second terms in each line on the right hand side of the above equations are zero. Observe that $z$ defined in the last line of the above equation belongs to $\mathfrak{g}_{\beta'-\beta}$. As $y\neq0$, we must have $[z,e_{\beta}]\neq0\implies z\neq0\implies \beta'-\beta\in\Delta$. Now, $\height_{I}(\beta'-\beta)=1$ implies $\beta'-\beta\in\Delta_{I,1}$. Putting all of this together, we have $\beta'-\beta\in\Delta_{I,1}$ and $[\mathfrak{g}_{\beta'-\beta},e_{\beta}]\neq \{0\}$ as desired, completing the proof in this case. \newline\newline
		(2) $(\beta' -\beta ,\beta' )\leq 0$: In this case we must have $(\beta' -\beta ,\beta)<0$. This implies there exists a simple root $\alpha \prec \beta' -\beta$ such that $(\alpha ,\beta)<0\implies\beta+\alpha\in\Delta\implies [e_{\alpha},e_{\beta}]\neq0$.\\
		If $\alpha\in\Pi_I$ and $\height_I(\beta'-\beta)=1$, then we are done by $[e_{\alpha},e_{\beta}]\neq0$. Else if $\alpha\in\Pi_I$ and  $\height_I(\beta'-\beta)>1$, then we are done by $[e_{\alpha},e_{\beta}]\neq0$, and by the induction hypothesis applied to the pair $\beta+\alpha\prec\beta'$.\\
		So, we assume now that $\alpha\notin\Pi_I$. By the
		induction hypothesis applied to the pair
		$\beta+\alpha\prec\beta'$, we get a root $\gamma \in
		\Delta_{I,1}$ such that
		$\gamma\prec\beta'-(\beta+\alpha)$ and
		$\big[e_{\gamma},[e_{\alpha},e_{\beta}]\big]\neq 0$. By
		the Jacobi {\color{blue}identity,}
		\[
		0\neq\big[e_{\gamma},[e_{\alpha},e_{\beta}]\big]= \big[\underbrace{[e_{\gamma},e_{\alpha}]}_{=:z_1},e_{\beta}\big]+ \big[e_{\alpha},\underbrace{[e_{\gamma},e_{\beta}]}_{=:z_2}\big].\]
		 If $\height_I(\beta'-\beta)=1$, then we are done as either $[z_1,e_{\beta}]$ or $z_2$ in the above equation must be non-zero. Else, for the same reason, we are done by the induction hypothesis applied to either $\beta+\alpha+\gamma\prec \beta'$ or respectively $\beta+\gamma\prec\beta'$.
		\end{proof}
	\begin{lemma}\label{L5.10}
		Let $\mathfrak{g}$ be semisimple and $\emptyset \neq I \subset \mathcal{I}$. Suppose $\beta \prec \beta' \in \Delta^+$ such that $\height_{I}(\beta' -\beta)>0$. Then there exists a sequence of roots $\gamma_i' \in \Delta_{I,1}$, $1\leq i\leq n= \height_{I}(\beta' -\beta)$, such that  
		\begin{equation*}
		\begin{aligned}
		\beta''&:=\beta'-\sum\limits_{j=1}^{n}\gamma_j'\in \Delta, \quad \Big[e_{\gamma^{'}_{n}},\big[\cdots , [e_{\gamma^{'}_{1}},e_{\beta''}]\cdots \big]\Big]\neq 0\\ &\text{and}\quad \beta \preceq \beta''\prec \cdots \prec\beta' -\sum\limits_{j=1}^{i}\gamma'_j\prec\cdots \prec \beta' \in \Delta\text{ }\forall i.	\end{aligned}
		\end{equation*}
	\end{lemma}
	\begin{proof}
		We prove the lemma by induction on $\height(\beta'-\beta)\geq 1$. Base step: $\height(\beta'-\beta)=1$ forces $\beta'-\beta \in \Pi_{I}$, and we have $[e_{\beta'-\beta},\mathfrak{g}_{\beta}]=\mathfrak{g}_{\beta'}$ as desired.\\
		Induction step: Assume that $\height(\beta'-\beta)>1$. Now, $0<(\beta' -\beta ,\beta' -\beta)=(\beta'-\beta,\beta')-(\beta'-\beta,\beta)$, which leads to the following two cases.\newline\newline
		(1) $(\beta' -\beta, \beta)<0$: There must exist a simple root $\alpha\prec \beta'-\beta $ such that $(\beta ,\alpha)<0$. This implies $[e_{\alpha},e_{\beta}]\neq0$, implying $\beta +\alpha \in \Delta$. If $\height_{I}(\beta +\alpha)<\height_{I}(\beta')$, then apply the induction hypothesis to the pair $\beta+\alpha\prec \beta'$ to get a chain of roots $\beta+\alpha\preceq \beta^{(1)}\prec\cdots\prec\beta'$ as in the statement. We are done if $\alpha\notin\Pi_{I}$. Otherwise, by applying the induction hypothesis further to the pair $\beta\prec \beta^{(1)}$, we will be done.\\
		So, we assume now that $\height_{I}(\beta +\alpha)=\height_{I}(\beta')$, which implies that $\height_I(\beta'-\beta)=1$ and $\alpha\in\Pi_I$. Consider $\beta+\alpha\prec\beta'$. By Theorem \ref{thmD} (\ref{thmD}1) and \cite[Proposition 3.6]{Kac}, we get a sequence of simple roots $\alpha_1,\ldots,\alpha_{m} \in \Pi_{I^c}$, where $m=\height(\beta' -\beta)-1$, such that 
		\[ y:=\Big[e_{\alpha_m},\big[\cdots ,\big[e_{\alpha_1},[e_{\alpha},e_{\beta}]\big]\cdots\big] \Big]\in\mathfrak{g}_{\beta'} \text{ and }y\neq 0.\]
		If $\beta+\alpha_j\in\Delta$ for some $j\in [m]$, then we are done by the induction hypothesis applied to the pair $\beta+\alpha_j\prec\beta'$. Else if $\beta+\alpha_j\notin\Delta$ $\forall$ $j\in [m]$, then $[e_{\alpha_j},e_{\beta}]=0$ $\forall$ $j\in[m]$, and by the Jacobi identity we have
		\begin{align*}
			\bigg[e_{\alpha_m},\Big[\cdots ,\big[e_{\alpha_1},[e_{\alpha},e_{\beta}]\big]\cdots \Big]\bigg]=& \bigg[e_{\alpha_m},\Big[\cdots,\big[[e_{\alpha_1},e_{\alpha}],e_{\beta}\big]\cdots\Big]\bigg]\\
			\vdots\\ =&\bigg[\underbrace{\Big[e_{\alpha_m},\big[\cdots ,[e_{\alpha_1},e_{\alpha}]\cdots\big]\Big]}_{=:z},e_{\beta}\bigg].
		\end{align*}
		Observe that $z$ (defined in the last line of the above equation) belongs to $\mathfrak{g}_{\beta'-\beta}$. As $y\neq0$, we must have $[z,e_{\beta}]\neq0\implies z\neq0\implies \beta'-\beta\in\Delta$. Moreover, $\height_{I}(\beta'-\beta)=1\implies\beta'-\beta\in\Delta_{I,1}$. Thus, we have $\beta'-\beta\in\Delta_{I,1}$ and $[\mathfrak{g}_{\beta'-\beta},e_{\beta}]\neq \{0\}$ as desired, completing the proof in this case.\newline\newline
		(2) $(\beta' -\beta, \beta)\geq 0$: In this case we must have $(\beta' -\beta ,\beta')>0$. Thus, there exists a simple root $\alpha' \prec \beta'-\beta$ such that $(\beta',\alpha')>0$. This implies $\beta'-\alpha' \in \Delta$, implying $[e_{\alpha'},e_{\beta'-\alpha'}]\neq0$.\\
		If $\alpha' \in \Pi_{I}$ and $\height_I(\beta'-\beta)=1$, then we are done by $[e_{\alpha'},e_{\beta'-\alpha'}]\neq0$ once we set $\beta''=\beta'-\alpha'$. Else if $\alpha'\in\Pi_I$ and $\height_I(\beta'-\beta)>1$, then we are done by $[e_{\alpha'},e_{\beta'-\beta}]\neq0$, and by the induction hypothesis applied to the pair $\beta\prec \beta'-\alpha'$.\\
		So, we assume now that $\alpha'\notin \Pi_I$. By the
		induction hypothesis applied to the pair
		$\beta\prec\beta'-\alpha'$ we get a root $\gamma' \in
		\Delta_{I,1}$ such that
		$\beta\prec\beta'-\alpha'-\gamma'\in\Delta^+$ and
		$[e_{\gamma'},e_{\beta'-\alpha'-\gamma'}]\neq 0$. By the
		Jacobi {\color{blue}identity,}
		\[0\neq \big[e_{\alpha'},[e_{\gamma'},e_{\beta'-\alpha'-\gamma'}]\big]=\big[\underbrace{[e_{\alpha'},e_{\gamma'}]}_{=:z_1},e_{\beta'-\alpha'-\gamma'}\big]+ \big[e_{\gamma'},\underbrace{[e_{\alpha'},e_{\beta'-\alpha'-\gamma'}]}_{=:z_2}\big].\]
		If $\height_I(\beta'-\beta)=1$, then as either $[z_1,e_{\beta'-\alpha'-\gamma'}]$ or $[e_{\gamma'},z_2]$ in the above equation must be non-zero, we will be done once we set $\beta''=\beta'-\alpha'-\gamma'$ or respectively $\beta'-\gamma'$. Else, as either $[z_1,e_{\beta'-\alpha'-\gamma'}]$ or $[e_{\gamma'},z_2]$ is non-zero, we will be done by the induction hypothesis applied to either $\beta\prec \beta'-\alpha'-\gamma'$ or respectively $\beta\prec \beta'-\gamma'$.
	\end{proof}
	We are now able to discuss ``moving between comparable roots in steps of $\Delta_{I,1}$'' when both the roots are not necessarily positive. 	\begin{proof}[\textnormal{\textbf{Proof of Proposition \ref{P5.8}}}]
		We prove the proposition in cases below.\\
	1) $\beta=0$ or $\beta'=0$: In this case the proposition follows by the parabolic-PSP applied to $\beta'$ or $\beta$ respectively.\\
	2) $\beta\succneqq 0$: In this case a) and b) directly follow by Lemmas \ref{L5.9} and \ref{L5.10} respectively.\\
    3) $\beta'\precneqq0$: In this case a) and b) once again directly follow by respectively applying Lemmas \ref{L5.10} and \ref{L5.9} to the pair $-\beta' \prec -\beta \in \Delta^+$.\\
	4) $\beta\precneqq0\precneqq\beta'$ and $\height_{I}(\beta)<0$: We prove a); the proof of b) is similar. By the parabolic-PSP we get a chain of roots between $\beta\prec 0$. When $\height_I(\beta')=0$, we are done by the previous line. Else when $\height_{I}(\beta')>0$, by the parabolic-PSP further applied to $0\prec \beta'$ we will be done. Notice that only in this case do we use the assumption made in part a), and this assumption is made in view of Remark \ref{R5.11} (1) part (ii) below.
	\end{proof}
	  	    In view of Remark \ref{R5.11} below, observe that Propositions \ref{P5.8} and \ref{P5.7} prove the parabolic-generalizations of (\ref{thmD}1) and (\ref{thmD}2) to the ``best possible'' extent.
		\begin{remark}\label{R5.11}
		(1) Let $\mathfrak{g}$ be of type $A_{3}$, and $\mathcal{I}=\{ 1,2,3\}$ where the nodes 1 and 3 are the leaves in the Dynkin diagram. Note that $\mathfrak{g}\simeq L(\alpha_1+\alpha_2+\alpha_3)$ as $\mathfrak{g}$-modules.
		\begin{itemize}
			\item[(i)] Let $I=\{1\}$ and consider $\alpha_1+\alpha_2+\alpha_3\succ\alpha_2$. Check that there does not exist a root $\gamma\in\Delta_{I,1}$ such that $\alpha_1+\alpha_2+\alpha_3-\gamma= \alpha_2\in\wt\mathfrak{g}$ or $\alpha_1+\alpha_2+\alpha_3=\alpha_2+\gamma\in\wt\mathfrak{g}$.
			\item[(ii)] Let $I_1=\{ 1,2\}$ and $I_2=\{2,3\}$, and consider $\alpha_3 \succ -\alpha_1$. Check that there does not exist a root $\gamma_1\in\Delta_{I_1,1}$ such that $\alpha_3-\gamma_1\succeq-\alpha_1\in\wt\mathfrak{g}$. Similarly, check that there does not exist a root $\gamma_2\in\Delta_{I_2,1}$ such that $\alpha_3\succeq-\alpha_1+\gamma_2\in\wt\mathfrak{g}$.
			\end{itemize} 
		(2) Observe that one can construct similar examples as in point (1) when $\mathfrak{g}$ is any semisimple Lie algebra of rank $\geq 3$, and even when $\mathfrak{g}$ is of type $A_1\times A_1$.\\ 
		(3) Let $\mathfrak{g}$ be an affine Kac--Moody algebra of type $X_{\ell}^{(r)}$ with rank $\ell\geq 3$, $r\in\{1,2,3\}$ and $\mathcal{I}=\{0,1,\cdots,\ell\}$, see \cite[Table Aff 1--3]{Kac}. Let $i,j\in\mathcal{I}\setminus\{0\}$ such that $i$ and $j$ are not connected by any edge in the Dynkin diagram. Set $I_1=\{j\}$ and $I_2=\{i\}$.  Let $\delta$ be the smallest positive imaginary root. Consider $r\delta+\alpha_i\succ r\delta-\alpha_j\in\Delta^+$, note both of these roots are real by \cite[Proposition 6.3]{Kac}. This result also implies that there does not exist a root $\gamma_1\in\Delta_{I_1,1}$ such that $r\delta+\alpha_i-\gamma_1\succ r\delta-\alpha_j\in\Delta$. Similarly, check that there does not exist a root $\gamma_2\in\Delta_{I_2,1}$ such that $r\delta+\alpha_i\succ r\delta-\alpha_j+\gamma_2\in\Delta$.\\
		Thus, Proposition \ref{P5.8} cannot be extended to affine root systems, even to move between two comparable positive roots.
		\end{remark} \section{Further observations on the sets $\Delta_{I,1}$}
	    In this section, we answer some questions related to the finiteness of the sets $\Delta_{I,1}$, and also the minimal generators of $\conv_{\mathbb{R}}\Delta_{I,1}$ for $\emptyset\neq I\subset \mathcal{I}$.
	 \subsection{Minimal generators of $\conv_{\mathbb{R}}\Delta_{I,1}$}\label{S6.1}
	  	   Let $C$ be a convex subset of a real vector space, and $B\subset C$. Recall, $B$ is said to generate $C$ if $\conv_{\mathbb{R}}B=C$. Similarly, $B$ is said to be a minimal generating set of $C$ if there does not exists $B_1\subsetneq B$ such that $\conv_{\mathbb{R}}B_1=C$. Note that these notions can as well be thought over an arbitrary subfield $F$ of $\mathbb{R}$ (with $F_{\geq 0}$ in the place of $\mathbb{R}_{\geq 0}$).\\
	  	   In this subsection, we define for $I\subset\mathcal{I}$ and $n\in\mathbb{Z}_{>0}$,
	    \begin{equation}\label{E6.1}
	   S_{I,n}:=\{\beta\in\Delta_{I,n}\text{ }|\text{ } \nexists\text{ } \alpha\in\Delta_{I,n}\text{ such that }\alpha\precneqq\beta\}.
	    \end{equation}
	   We call the elements of $S_{I,n}$ the minimal elements of $\Delta_{I,n}$. The goal of this subsection is to prove the following Lemma, which is also needed in the proof of Proposition \ref{P2.11} in the Appendix. \begin{lemma}\label{L6.1}
	    	Let $\mathfrak{g}$ be a Kac--Moody algebra, and fix $\emptyset\neq I\subset \mathcal{I}$ and $n\in\mathbb{Z}_{>0}$. Then:
	    	\begin{itemize}
	    	\item[(a)] $\conv_{\mathbb{R}}\Delta_{I, n}$ is generated by $W_{I^c}S_{I, n}$. 
	      	\item[(b)] For $\alpha\in\Delta^{+}$ and $J\subset \mathcal{I}\setminus \supp(\alpha)$, $\conv_{\mathbb{R}}\Delta_{\alpha ,J}$ is minimally generated by $W_{J}\alpha$.
	      	\item[(c)] In particular, $\conv_{\mathbb{R}}\Delta_{I, 1}$ is minimally generated by $ W_{I^c}\Pi_{I}$.
	    	\end{itemize}
	    \end{lemma}	
	    \begin{proof}
	    	To prove (a), we show that $\Delta_{I,n}\subset \conv_{\mathbb{R}}W_{I^c}S_{I,n}$ by the method of contradiction. This proves that $\conv_{\mathbb{R}}\Delta_{I,n}=\conv_{\mathbb{R}}W_{I^c}S_{I,n}$.\\
	    	Suppose $\Delta_{I,n}\not\subset \conv_{\mathbb{R}}W_{I^c}S_{I,n}$. Then there must exist a root $\beta\in\Delta_{I,n}$ such that $\beta\notin \conv_{\mathbb{R}}W_{I^c}S_{I,n}$. Without loss of generality assume that $\beta$ is of least height such that $\beta\in\Delta_{I,n}\setminus \conv_{\mathbb{R}}W_{I^c}S_{I,n}$. Note that $\beta$ cannot be a minimal element of $\Delta_{I,n}$ (i.e. an element of $S_{I,n}$) by the choice. This implies that there exists a root $\eta\in\Delta_{I,n}$ such that $\eta\precneqq\beta$. As $\height_I(\beta)=\height_I(\eta)=n$ and $\eta\prec\beta$ observe that $\supp(\beta-\eta)\subset I^c$. Now, by Theorem \ref{thmD} (\ref{thmD}1) applied to $\eta \precneqq \beta$, we have $\beta-\alpha_{j_0}\in\Delta_{I,n}$ for some $j_0\in\supp(\beta-\eta)\subset I^c$.\\
	    	If $\langle \beta ,\alpha_{j_0}^{\vee}\rangle>0$, then as $\height(s_{j_0}\beta)<\height(\beta)$ we have $s_{j_0}\beta\in \conv_{\mathbb{R}}W_{I^c}S_{I,n}$. Now, the $W_{I^c}$-invariance of $\conv_{\mathbb{R}}W_{I^c}S_{I,n}$ implies that $\beta=s_{j_0}(s_{j_0}\beta)\in \conv_{\mathbb{R}}W_{I^c}S_{I,n}$, contradicting the choice of $\beta$.\\
	    	So, we assume now that $\langle\beta,\alpha^{\vee}_{j_0}\rangle\leq 0$. This implies 
	    	\[
	    	\epsilon':= -1-\langle \beta-\alpha_{j_0} , \alpha_{j_0}^{\vee}\rangle \geq 1.
	    	\]
	    	As $\height(\beta-\alpha_{j_0})<\height(\beta)$ we have \[\beta-\alpha_{j_0}\in \conv_{\mathbb{R}}W_{I^c}S_{I,n} \implies s_{j_0}(\beta-\alpha_{j_0})\in \conv_{\mathbb{R}}W_{I^c}S_{I,n}.
	    	\]
	    	Observe then that we can write 
	    	\[
	    	\beta = \frac{\epsilon'}{\epsilon' + 1}(\beta -\alpha_{j_0})+\frac{1}{\epsilon' +1}s_{j_0}(\beta -\alpha_{j_0}).
	    	\]
	    	This implies that $\beta\in \conv_{\mathbb{R}}W_{I^c}S_{I,n}$, once again contradicting the choice of $\beta$. Thus, $\conv_{\mathbb{R}}\Delta_{I,n}=\conv_{\mathbb{R}}W_{I^c}S_{I,n}$.\\
	    	We now prove (b). In view of Theorem \ref{thmD} part (\ref{thmD}1) note that $\alpha$ is the only minimal element in $\Delta_{\alpha,J}$. It can be proved very similar to (a) that $W_J\alpha$ generates $\conv_{\mathbb{R}}\Delta_{\alpha,J}$. So, we only have to show that $W_J\alpha$ is minimal (in generating $\conv_{\mathbb{R}}\Delta_{\alpha,J}$).\\ 
	    	Suppose 
	    	\[ \omega \alpha = r_1 \omega_1 \alpha + \cdots + r_k \omega_k \alpha \text{ for some }r_1 ,\ldots ,r_k \in \mathbb{R}_{>0}\text{ and }\omega ,\omega_1, \ldots ,\omega_k \in W_J.\] Then we have 
	    	\[\alpha = r_1\omega^{-1}\omega_1 \alpha + \cdots+ r_k \omega^{-1}\omega_{k}\alpha \implies \supp(\omega^{-1}\omega_i \alpha)\subset \supp(\alpha)\text{ }\forall\text{ }i\in [k], \text{ as } \omega^{-1}\omega_i\alpha\in\Delta_{\alpha,J}.\] 
	    	Observe then that, as $\omega^{-1}\omega_i\alpha\succeq\alpha$, we must have $\omega \alpha = \omega_i \alpha$ $\forall$ $i\in [k]$, which proves the minimality.\\
	        (c) follows from (b) by noting: (i) $\Delta_{I,1}=\bigsqcup\limits_{i\in I}\Delta_{\alpha_i, I^c}$, and (ii) if \[
	        \omega' \alpha_{i_0}= r_1 \omega'_1 \alpha_{i_1}+\cdots +r_l \omega'_l \alpha_{i_l}\text{ for some }r_1 ,\ldots ,r_l \in \mathbb{R}_{>0},\text{ } i_0 ,\ldots ,i_l \in I\text{ and }\omega' ,\omega'_1 ,\ldots ,\omega'_l \in W_{I^c},\] 
	        then $i_t = i_0$ and $\omega'_t\alpha_{i_t} =\omega'\alpha_{i_0}$ $\forall$ $1\leq t\leq l$.\\
	        This also proves that the cone $\mathbb{R}_{\geq 0}(\Delta^+\setminus\Delta^+_{I^c})=\mathbb{R}_{\geq0}\Delta_{I,1}$ is minimally generated by $W_{I^c}\Pi_I$. 
	    \end{proof}
	  \begin{remark}  Observe that the proof of Lemma \ref{L6.1} above holds true even if we replace $\mathbb{R}$ everywhere in the proof by an arbitrary subfield of $\mathbb{R}$. Thus, the assertions on convex hulls in Lemma \ref{L6.1} are more generally true over any subfield of $\mathbb{R}$. 
	   \end{remark} \subsection{Finiteness of $\Delta_{I,1}$}\label{S6.2}
	    	 In this subsection, we give necessary and sufficient conditions for the sets $\Delta_{\alpha,J}$ and $\Delta_{I,1}$ to be finite, see Proposition \ref{P6.3} and Corollary \ref{C6.4} below. In this we invoke an interesting result of Deodhar \cite{VVD}, which proves the equivalence of the finiteness of a Coxeter group (in our situation, a parabolic subgroup of the Weyl group of the Kac--Moody algebra $\mathfrak{g}$) and that of the quotients by its parabolic subgroups.\\
	    	 We first note the following observation which motivates, and proves (independently), Proposition \ref{P6.3} below in the special case when $\mathfrak{g}$ is an affine Kac--Moody algebra.
	    	 \begin{observation}\label{O6.2}
	    	 Let $\mathfrak{g}$ be an affine Kac-Moody algebra, and let $\alpha\in\Delta^+$ be such that $\supp(\alpha)\subsetneq \mathcal{I}$. Fix $\emptyset\neq J\subset \mathcal{I}\setminus\supp(\alpha)$, and $\emptyset\neq I\subsetneq\mathcal{I}$. Then:  \begin{itemize} 
	    	 \item[(1)] $\mathfrak{g}_{I^c}$ and $\mathfrak{g}_J$ are semisimple, or equivalently every connected component of the Dynkin subdiagram on $I^c$ and respectively on $J$ is of finite type.
	    	 \item[(2)] It can be easily checked by the explicit description of $\Delta$ in Proposition 6.3 of Kac's book \cite{Kac} that $\Delta_{I,1}$ and $\Delta_{\alpha,J}$ are finite.
	    	 \end{itemize}
	    	 \end{observation}
	    	\begin{prop}\label{P6.3}
	    Let $\mathfrak{g}$ be a Kac--Moody algebra, $\alpha\in\Delta^+$, and $J\subset\mathcal{I}\setminus \supp(\alpha)$. Then the following are equivalent:
	    \begin{itemize}
	          \item[(1)] $\Delta_{\alpha,J}$ is finite.
	          \item[(2)] $W_J\alpha$ is finite.
	          \item[(3)] The quotient $W_J$ modulo the parabolic subgroup $\big\langle s_j\text{ }\big|\text{ }j\in J\text{ and }\langle\alpha,\alpha_j^{\vee}\rangle=0\big\rangle$ is finite.
	          \item[(4)] When $J\neq\emptyset$, there exists $J'\subset J$ such that $\mathfrak{g}_{J'}$ is semisimple and $\Delta_{\alpha,J}=\Delta_{\alpha,J'}$.	        \end{itemize}
            When the Dynkin subdiagram on $J$ is connected and when $\Delta_{\alpha,J}\supsetneq \{\alpha\}$, $J'$ in (4) is equal to $J$. 
	    \end{prop}
	    \begin{proof}
	    Firstly, when $J=\emptyset$, note that the equivalence of (1), (2), (3) is trivial as $\Delta_{\alpha,J}=\{\alpha\}$. Similarly, when $J\neq \emptyset$ and $\Delta_{\alpha,J}=\{\alpha\}$---which happens if and only if $\big\langle s_j\text{ }\big|\text{ }j\in J\text{ and }\langle\alpha,\alpha_j^{\vee}\rangle=0\big\rangle=W_J$---by choosing $J'$ to be any singleton subset of $J$ it can be easily checked that (1), (2), (3), (4) are all equivalent. 
	    
	    So, we assume for the rest of the proof that $\Delta_{\alpha,J}\supsetneq \{\alpha\}$. Now, (1) $\implies$ (2) is obvious as $W_J\alpha\subset \Delta_{\alpha,J}$. It can be easily checked that (2) $\implies$ (1) follows by Lemma \ref{L6.1} part (b) and the fact that $\Delta_{\alpha,J}$ is a discrete subset of $\conv_{\mathbb{R}}W_J\alpha$.
	    	For (2) $\iff$ (3), consider the action of $W_J$ on $-\alpha$. Observe that $-\alpha$ belongs to the fundamental chamber of the (dual) $J$-Tits cone---i.e. $\langle-\alpha,\alpha_j^{\vee}\rangle\geq0$ $\forall$ $j\in J$. So, by \cite[Proposition 3.12]{Kac} we must have that the isotropy group $(W_J)_{-\alpha}:=\{w\in W_J\text{ }|\text{ }w(-\alpha)=-\alpha\}$ (which is same as $(W_J)_{\alpha}$) equals the parabolic subgroup $\big\langle s_j\text{ }\big|\text{ }j\in J\text{ and }\langle\alpha,\alpha_j^{\vee}\rangle=0\big\rangle$. By the bijection which exists between the set of left cosets $W_J\big/(W_J)_{\alpha}$ and the orbit $W_J\alpha$, (2) $\iff$ (3) follows.\\
	    	For (3) $\implies$ (4), let $J=J_1\sqcup \cdots \sqcup J_l$ give the decomposition of the Dynkin subdiagram on $J$ into connected components ($l=1$ when the Dynkin subdiagram on $J$ is connected). Let $K$ be the subset of $J$ such that $W_K=(W_J)_{\alpha}$. Observe by the assumption $\Delta_{\alpha,J}\supsetneq \{\alpha\}$ that $K\subsetneq J$. Note that $W_J\simeq W_{J_1}\times\cdots\times W_{J_l}$ and $W_K\simeq W_{J_1\cap K}\times\cdots\times W_{J_l\cap K}$ (direct products of groups). Thus, there exists a bijection between $W_J\big/ (W_J)_{\alpha}=W_J\big/ W_K$ and the Cartesian product $\prod\limits_{t=1}^{l}\big(W_{J_t}\big/W_{J_t\cap K}\big)$ of the sets of left cosets $W_{J_t}\big/W_{J_t\cap K}$ $\forall$ $1\leq t\leq l$. The assumption in (3) that $W_J/W_K$ is finite, and the previous line together imply that $W_{J_t}\big/W_{J_t\cap K}$ is finite $\forall$ $1\leq t\leq l$. Now, by \cite[Proposition 4.2]{VVD} observe that the Dynkin subdiagram on $J_t$ must be of finite type whenever $K\cap J_t\subsetneq J_t$. Now, set $J'$ to be the union of all the subsets $J_t$ such that $K\cap J_t\subsetneq J_t$. Observe that $J'\neq \emptyset$ as $K\subsetneq J$, and also $\mathfrak{g}_{J'}$ is semisimple by the previous two sentences. Moreover, it can be easily checked that $W_J\alpha= W_{J'}\alpha$. Now, Lemma \ref{L6.1} (b) yields $\Delta_{\alpha,J}=\Delta_{\alpha,J'}$. \\
	    	As in (4), suppose $\Delta_{\alpha,J}=\Delta_{\alpha,J'}$ for some $J'\subset J$ such that $\mathfrak{g}_{J'}$ is semisimple. Then $W_{J'}$ is finite as $\mathfrak{g}_{J'}$ is semisimple. This implies $W_{J'}\alpha$ is finite. By (2) $\implies$ (1) (for $J'$ in the place of $J$) we get that $\Delta_{\alpha,J'}$ is finite. Thus, $\Delta_{\alpha,J}$ is finite. This proves (4) $\implies$ (1), completing the proof of the proposition.  \end{proof}
	    	\begin{cor}\label{C6.4}
	    	    Let $\mathfrak{g}$ be a Kac--Moody algebra, and $\emptyset\neq I\subsetneq\mathcal{I}$. Then the following are equivalent: \begin{itemize}
	    	       \item[(1)] $\Delta_{I,1}$ is finite
	         \item[(2)] $W_{I^c}\Pi_I$ is finite.
	         \item[(3)] For each $i\in I$ the quotient $W_{I^c}$ modulo the parabolic subgroup $\big\langle s_t\text{ }\big|\text{ }t\in I^c\text{ and }\langle\alpha_i,\alpha_t^{\vee}\rangle=0\big\rangle$ is finite.
	         \item[(4)] For each $i\in I$ there exists $J_i\subset I^c$ such that $\mathfrak{g}_{J_i}$ is semisimple and $\Delta_{\alpha_i,I^c}=\Delta_{\alpha_i,J_i}$. 
	        \end{itemize}
	    	When the Dynkin subdiagram on $I^c$ is connected and $\Delta_{I,1}\supsetneq \Pi_I$, $J_i$ in (4) is equal to $I^c$ $\forall$ $i\in I$.
	    	\end{cor}
	    	\begin{proof}
	    	The proof immediately follows from Proposition \ref{P6.3} by noting that $\Delta_{I,1}=\bigsqcup\limits_{i\in I}\Delta_{\alpha_i,I^c}$.
	    \end{proof}  
	    	Let $\emptyset\neq I\subset \mathcal{I}$. Observe that if $\Delta_{I,1}$ is finite, then by the parabolic-PSP $\Delta_{I,n}$ is finite for any $n\in\mathbb{Z}\setminus\{0\}$. It might be interesting to check if the finiteness of $\Delta_{I,m}$ for some $m\in\mathbb{Z}\setminus \{0\}$ implies the finiteness of $\Delta_{I,1}$ and hence the finiteness of $\Delta_{I,n}$ for every $n\in\mathbb{Z}\setminus\{0\}$.
	    	
	    	In the rest of this subsection, we look at the lengths of the roots in $\Delta_{\alpha,J}$ when $\mathfrak{g}$ is symmetrizable.	    \begin{lemma}\label{L6.5}
	    	Let $\mathfrak{g}$ be a symmetrizable Kac--Moody algebra with symmetric invariant form $(.,.)$. Fix a real root $\alpha\in \Delta^+$ and $J\subset \mathcal{I}\setminus \supp(\alpha)$. Suppose $\beta \in \Delta_{\alpha,J}$. Then:
	      		\begin{itemize}
	      		\item[(a)] $(\beta ,\beta)\leq(\alpha ,\alpha)$.
	      		\item[(b)] $(\beta,\beta)=(\alpha,\alpha)\implies\exists$ $\omega \in W_{\supp(\beta -\alpha)}$ such that $\omega \alpha =\beta$.
	      	\end{itemize}
	      	\end{lemma}
	      	\begin{proof}
	      		We prove (a) by induction on $\height(\beta-\alpha)\geq 0$. (b) was proved by Carbone et al in \cite[Proposition 6.4]{svis}. In the base step $\height(\beta-\alpha)=0$, (a) is trivial as $\beta=\alpha$.\\
	      		  Induction step: Assume $\height(\beta-\alpha)\geq 1$, and let $\beta =\alpha+\sum_{j\in J}c_j\alpha_j$ for some $c_j\in\mathbb{Z}_{\geq0}$. If $\beta$ is imaginary, then, as $\alpha$ is real, (a) just follows by $(\beta,\beta)\leq0<(\alpha,\alpha)$. So, we assume throughout the proof that $\beta$ is real. If there exists some $j'\in J$ such that $\langle\beta,\alpha_{j'}^{\vee}\rangle>0$, then $s_{j'}\beta\precneqq\beta\in\Delta_{\alpha,J}$. Now, as $\height(s_{j'}\beta-\alpha)<\height(\beta-\alpha)$ and $(s_{j'}\beta,s_{j'}\beta)=(\beta,\beta)$, we will be done by the induction hypothesis applied to $s_{j'}\beta$. So, we also assume now that $\langle\beta,\alpha_j^{\vee}\rangle\leq0$ $\forall$ $j\in J$. This assumption and $(\beta,\beta)>0$ (as $\beta$ is real) together force $\langle\beta ,\alpha^{\vee}\rangle>0$. Now, $\beta\in\Delta_{\alpha,J}$ and $s_{\alpha}\beta\in\Delta^+$ further force $\langle \beta ,\alpha^{\vee}\rangle =1$. Thus, $\beta-\alpha= s_{\alpha}(\beta)$, and so $(\beta ,\beta)=(\beta -\alpha ,\beta -\alpha)>0$. Observe now
	      		  \[(\beta -\alpha ,\beta -\alpha)+(\beta -\alpha ,\alpha)= \big(\sum_{j\in J}c_j\alpha_j,\beta\big)\leq 0\qquad \text{and}\qquad (\alpha ,\alpha)+(\alpha ,\beta -\alpha)= (\alpha ,\beta)>0\] 
	      		  together imply that $(\alpha ,\alpha)>(\beta -\alpha ,\beta -\alpha)=(\beta,\beta)$, completing the proof of (a).
	      	\end{proof} 
	      	\begin{cor}\label{C6.6}
	      		Let $\mathfrak{g}$, $\alpha$ and $J$ be as in Lemma \ref{L6.5}. Then $\Delta_{\alpha,J}=W_{J}\alpha$ if and only if all the roots in $\Delta_{\alpha,J}$ are real and $(\alpha,\alpha)$ is least among the lengths of all the roots in $\Delta_{\alpha,J}$.
	      	\end{cor}
	 \subsection*{Acknowledgements} My deepest thanks and gratitude to my Ph.D. advisor, Apoorva Khare, for introducing me to the two generalizations of the partial sum property studied in this paper, and for his invaluable encouragement and discussions which helped to refine the results and also improve the exposition. I sincerely thank R. Venkatesh for valuable discussions, including the suggestion to explore the parabolic-PSP at the level of Lie words; as well as Gurbir Dhillon for his invaluable comments on the paper, including suggesting the references to the results in the Appendix of the paper. I also thank Amritanshu Prasad for valuable discussions; as well as Sankaran Viswanath for pointing me to some of the references in this paper which helped me to improve the results of subsection \ref{S6.2} of this paper. This work is supported by a scholarship from the National Board for Higher Mathematics (Ref. No. 2/39(2)/2016/NBHM/R{\&}D-II/11431).        
    \address{(G Krishna Teja) \textsc{R-21, Department of Mathematics, Indian Institute of Science, Bangalore 560012, India}}
    
    \textit{E-mail address}: \email{\texttt{tejag@iisc.ac.in}}   \appendix  \section{A maximal property, and the extremal rays, of $\mathcal{P}(\lambda,J)$}
    
	       In the Appendix, we re-define the partial order on $\mathfrak{h}^*$ as follows:
	       \[\text{for }x,y\in\mathfrak{h}^*,\quad x\prec y\iff y-x\in\mathbb{R}_{\geq 0}\Pi.    \]
	       We now use the above analysis to prove a ``maximal property'' satisfied by $\mathcal{P}(\lambda,J)$ under some finiteness conditions, see Maximal property \ref{MA.2} below. Our proof for this result is different to the proofs in \cite{Dhillon_arXiv, Khare_Ad, Khare_Trans}. Recall the definitions of $J_{\lambda}, J'_{\lambda}, \mathcal{P}(\lambda,J)$ from equations \eqref{E2.4} and \eqref{E2.13}. 
	    \begin{prop}\label{PA.1}
	      	Let $\mathfrak{g}$ be a Kac--Moody algebra, and fix $\lambda \in \mathfrak{h}^*$ and $J\subset J_{\lambda}'$. Suppose $W_J$ is finite. Then $\mu \in \mathcal{P}(\lambda,J)\iff\omega \mu \prec \lambda$ $\forall$ $\omega \in W_{J}$.
	      \end{prop}
	      \begin{proof}
	     The proposition holds true if $J=\emptyset$, as $\mathcal{P}(\lambda,\emptyset)=\lambda-\mathbb{R}_{\geq0}\Delta^+=\lambda-\mathbb{R}_{\geq0}\Pi$. So, we assume that $J\neq \emptyset$. Observe by Corollary \ref{C6.4} that the finiteness of $W_J$ implies that $\Delta_{J^c,1}$ is finite. It can be easily seen that the finiteness of $W_J$ and $\Delta_{J^c,1}$, and the Minkowski difference formula for $\mathcal{P}(\lambda,J)$ together imply the closedness of $\mathcal{P}(\lambda,J)$ in the usual Euclidean topology on $\mathfrak{h}^*$. \\
	      	Fix $\mu\precneqq\lambda \in\mathfrak{h}^*$. As $\mathfrak{g}_J$ is semisimple, assume without loss of generality that $\mu$ is the $J$-dominant element in $W_J\mu$, i.e. $\langle\mu,\alpha^{\vee}_{j}\rangle\geq 0$ $\forall$ $j\in J$. Observe that the forward implication of the proposition is obvious, as $\mathcal{P}(\lambda,J)$ is $W_J$-invariant. To show the reverse implication, consider 
	      	 \[ \epsilon: =\inf\big\{\height(\lambda'- \mu) \text{ }|\text{ }\mu\prec\lambda'\in \mathcal{P}(\lambda,J)\big\}\geq0.\]
	      	 As $\eta\in \mathcal{P}(\lambda,J)\implies \eta-\sum_{i\in J^c}b_i \alpha_i\in \mathcal{P}(\lambda,J)\text{ }\forall\text{ } b_i \in \mathbb{R}_{\geq0}$, we also have \[\epsilon=\inf\big\{\height(\lambda'-\mu)\text{ }|\text{ }\mu\prec\lambda'\in \mathcal{P}(\lambda,J)\text{ and }\lambda'-\mu \in \mathbb{R}_{\geq0}\Pi_{J}\}.\]
	      	 Fix a sequence $\xi_n\succ\mu \in \mathcal{P}(\lambda,J)$ such that $\xi_n -\mu \in \mathbb{R}_{\geq0}\Pi_{J}$ and $\height(\xi_n -\mu)\longrightarrow \epsilon$. As $\xi_n$ lie in the compact set $\big(\lambda-[0,\height(\lambda-\mu)]\Pi\big)\cap \mathcal{P}(\lambda,J)$, we get a subsequence $\xi_{n_k}$ and a point $\xi\in \mathcal{P}(\lambda,J)$ such that $\xi_{n_k} \longrightarrow\xi\in \mathcal{P}(\lambda,J)$. Check that $\xi\succ\mu$ and $\height(\xi-\mu)=\epsilon$. Now, the result follows once we prove $\epsilon=0$. So, suppose to the contrary that $\epsilon>0$.\\
	           Let (.,.) denote the positive definite (Killing) form on $\mathfrak{h}^*_J\times \mathfrak{h}^*_J$. Write $\xi-\mu=\sum\limits_{j\in J}c_j\alpha_j$, and define $(\xi-\mu)^{\vee}:=\sum\limits_{j\in J}c'_j\alpha_j^{\vee}$, where $c_j$ and $c'_j=\frac{(\alpha_j,\alpha_j)}{(\xi-\mu,\xi-\mu)}c_j\in\mathbb{R}_{\geq0}$ $\forall$ $j\in J$. Note that \[ \text{for each }j\in J\text{ } c'_j= 0\iff c_j= 0 \text{ },\qquad 2=\langle\xi-\mu,(\xi-\mu)^{\vee}\rangle=\langle\xi,(\xi-\mu)^{\vee}\rangle-\langle
	      	\mu,(\xi-\mu)^{\vee}\rangle.\]
	      	As $\mu$ is $J$-dominant, we have $\langle\mu,(\xi-\mu)^{\vee}\rangle\geq0$, forcing $\langle\xi,(\xi-\mu)^{\vee}\rangle>0$. This implies that there exists $j_0\in J$ such that $c_{j_0}>0$ (equivalently $c'_{j_0}>0$) and $\langle\xi,\alpha_{j_0}^{\vee}\rangle>0$. Now, define $r:=\min\{ 1,\frac{{c_{j_0}}}{\langle\xi,\alpha_{j_0}^{\vee}\rangle}\}$. Observe that 
	      	\[
	      	(1-r)\xi+rs_{j_0}(\xi)\in\mathcal{P}(\lambda,J)\quad\text{ and }\quad0<\height\big([(1-r)\xi+rs_{j_0}(\xi)]-\mu\big)<\height(\xi-\mu)=\epsilon.
	      	\]
	      	This contradicts the minimality of $\epsilon$. Hence, $\epsilon$ must be 0. 
	      \end{proof}
	      \begin{mxlp}\label{MA.2}
	      	Let $\mathfrak{g},\lambda,J$ and $\mathcal{P}(\lambda,J)$ be as in Proposition \ref{PA.1}. Suppose $X\subset\mathfrak{h}^*$ is such that $ X\subset \lambda-\mathbb{R}_{\geq0}\Pi$ and $W_JX=X$. Then $X\subset \mathcal{P}(\lambda,J)$.
	      \end{mxlp}
	      \begin{note}\label{NA.3}
	       Notice that when either (i) $\mathfrak{g}$ is of affine type and $|J^c|\geq1$ or (ii) $\mathfrak{g}$ is hyperbolic and $|J^c|\geq2$, $\mathcal{P}(\lambda,J)$ defined over $\mathfrak{g}$ has the above maximal property.
	       \end{note}
	      It will be interesting to check the extent to which the above result may be extended. We now proceed to prove Proposition \ref{P2.11}---which says that the extremal rays of the shape $\mathcal{P}(\lambda,J)$ for $J\subset J'_{\lambda}$ are $\bigsqcup_{i\in J^c}W_J(\lambda-\mathbb{R}_{\geq 0}\alpha_i)$---using Theorems \ref{thmA} and \ref{thmD}. Recall that this was shown in \cite{Dhillon_arXiv, Khare_Ad, Khare_Trans}, and our proof is different to the proofs therein. Observe also that our proof works well even when $\mathbb{R}$ everywhere in the proof is replaced by an arbitrary subfield $F$ of $\mathbb{R}$ (with $F_{\geq 0}$ in the place of $\mathbb{R}_{\geq 0}$). Thus, Proposition \ref{P2.11} holds true more generally over $F$ for $\mathcal{P}(\lambda,J)$ defined over $F$ (i.e. with $F$ in the place of $\mathbb{R}$, and $J\subset\{j\in \mathcal{I}$ $|$ $\langle\lambda,\alpha_j^{\vee}\rangle\in F_{\geq 0}\}$ in the definition of $\mathcal{P}(\lambda,J)$).
	   \begin{proof}[\textnormal{\textbf{Proof of Proposition \ref{P2.11}:}}]
	    	We assume that $\emptyset\neq J\subsetneq \mathcal{I}$, as (i) the proposition holds true trivially when $J=\emptyset$, and (ii) there is nothing to prove in view of Lemma \ref{LB.1} part (b) below when $J=\mathcal{I}$. Throughout the proof, we only deal with the extremal rays of $\mathcal{P}(\lambda,J)$ at $\lambda$, as all other extremal rays are $W_J$-conjugates of these.\\
	    	Observe that each ray $\lambda -\mathbb{R}_{\geq 0}\alpha_i$ is clearly an extremal ray $\forall$ $i\in \mathcal{I}\setminus J$. We first prove the following useful lemma about the orbits under the parabolic subgroups of Weyl groups. For $\lambda\in\mathfrak{h}^*$ and $J\subset J'_{\lambda}$, recall the definition of $J_0$ in the statement of Proposition \ref{P2.11}, $J_0:=\{j_0\in J$ $|$ $\langle\lambda,\alpha_{j_0}^{\vee}\rangle=0\}$.	       	\begin{lemma}\label{LB.1}
	       	Let $\mu \in \conv_{\mathbb{R}}W_{J}\lambda$ and $\mu\neq\lambda$. Then:
	       	\begin{itemize} \item[(a)]  $\height_{J\setminus J_0}(\lambda -\mu)>0$.
	       	\item[(b)] The ray originating from $\lambda$ and containing $\mu$, which is precisely $\lambda-\mathbb{R}_{\geq0}(\lambda-\mu)$, cannot be an extremal ray of $\conv_{\mathbb{R}}W_J\lambda$.
	       	\end{itemize}	
	       	\end{lemma}
	    	\begin{proof} For (a), we prove that $\omega\in W_J\setminus W_{J_0}\implies \height_{J\setminus J_0}(\lambda - \omega\lambda)>0$ by induction on $\ell(\omega)$. Observe that this proves (a). Base step: $\ell(\omega)=1$, and we must have $\omega=s_{\hat{j}}$ for some $\hat{j}\in J\setminus J_0$. By the definition of $J_0$, the assertion is immediate.\\
	    	Induction step: Let $k:=\ell(\omega)>1$, and let $s_{j_1}\cdots s_{j_k}=\prod_{t=1}^k s_{j_t}$ be a reduced expression of $\omega$ for some $j_1 ,\ldots ,j_k \in J$ not necessarily distinct. As $\omega\in W_J\setminus W_{J_0}$, there must exist some $t\in[k]$ such that $j_t\in J\setminus J_0$. Assume without loss of generality that $j_k \in J\setminus J_0$ and write $\omega \lambda =(\prod_{t=1}^{k-1} s_{j_t})\lambda - \langle \lambda , \alpha_{j_k}^{\vee}\rangle(\prod_{t=1}^{k-1} s_{j_t})\alpha_{j_k}$. If $j_1 ,\ldots , j_{k-1} \in J_0$, then $\omega\lambda=\lambda- (\prod_{t=1}^{k-1}s_{j_t})\alpha_{j_k} \in\lambda-\Delta_{\{j_k\},1}$, and we are done. Else, by the induction hypothesis applied to $\ell(\prod_{t=1}^{k-1}s_{j_t})=\ell(\omega)-1$, we have $\height_{J\setminus J_0}(\lambda-(\prod_{t=1}^{k-1}s_{j_t})\lambda)>0$. As $\prod_{t=1}^{k}s_{j_t}$ is reduced, by \cite[Lemma 3.11]{Kac}, we also have $(\prod_{t=1}^{k-1}s_{j_t})\alpha_{j_k}\succ0$. Hence, $\height_{J\setminus J_0}(\lambda-w\lambda)>0$, completing the proof of (a). 
	    		
	    	Let $\mu=\lambda-\gamma$ for some $0\neq\gamma\in\mathbb{R}_{\geq0}\Pi_J$. For (b), we assume $\lambda-\mathbb{R}_{\geq0}\gamma$ to be an extremal ray of $\conv_{\mathbb{R}}W_J\lambda$ and exhibit a contradiction. If $\mu=c_1w_1\lambda+\cdots+c_rw_r\lambda\text{ for some }w_i\in W_J \text{ and }c_i\in\mathbb{R}_{>0}\text{ such that }\sum\limits_{i=1}^{r}c_i=1,$	then $w_i\lambda\in\lambda-\mathbb{R}_{\geq0}\gamma$ $\forall$ $1\leq i\leq r$ (by the definition of an extremal ray). By the previous line, we may assume that the element $\mu$ we started with lies in $W_J\lambda$. Let $\mu=w\lambda$ for some $w\in W_J\lambda$ and consider $\lambda-2\gamma\in\lambda-\mathbb{R}_{\geq0}\gamma\subset \conv_{\mathbb{R}}W_J\lambda$. Observe that \begin{equation}\label{E5.2}
	    	\mu=w\lambda=\frac{1}{2}(\lambda)+\frac{1}{2}(\lambda-2\gamma)\implies \lambda=\frac{1}{2}(w^{-1}\lambda)+\frac{1}{2}w^{-1}(\lambda-2\gamma).\end{equation} 
	    	Recall that $x\prec \lambda$ $\forall$ $x\in \conv_{\mathbb{R}}W_J\lambda$, and $\conv_{\mathbb{R}}W_J\lambda$ is $W_J$-invariant. Therefore, equation \eqref{E5.2} forces $w^{-1}\lambda=w^{-1}(\lambda-2\gamma)=\lambda$. This implies $\mu=\lambda$ which is a contradiction.
	    	\end{proof}
	    	We now continue with the proof of Proposition \ref{P2.11} in steps below. Let $I:=\mathcal{I}\setminus J$.\newline
	    	 \textbf{Step 1.} Observe firstly that, by the minimal description $\mathbb{Z}_{\geq 0}(\Delta^+ \setminus \Delta_{J}^+) =\mathbb{Z}_{\geq 0}\Delta_{I ,1}$ (by Theorem \ref{thmA} (\ref{thmA}2)) and by the proof of Lemma \ref{L6.1} (c), the extremal rays of $\lambda - \mathbb{R}_{\geq 0}(\Delta^+ \setminus \Delta_{J}^+)$ are precisely $\lambda -\mathbb{R}_{\geq 0}W_{J}\Pi_{I}$. Let $\mu = \mu_1 -\mu_2\in \mathcal{P}(\lambda,J)$ for some $\mu_1\in \conv_{\mathbb{R}}W_J\lambda$ and $\mu_2 \in \mathbb{R}_{\geq 0}\Delta_{I,1}$. We show that if $\mu_1\neq\lambda$, then $\lambda - \mathbb{R}_{\geq 0}(\lambda -\mu)$ cannot be an extremal ray. Note that if $\mu_1 \neq \lambda$ and $\mu_2=0$, this just  follows by Lemma \ref{LB.1} (b). If $\mu_1\neq\lambda$ and $\mu_2\neq0$, then by the non-trivial relation \[\lambda-\frac{1}{2}(\lambda-\mu_1+\mu_2)=\frac{1}{2}(\mu_1)+\frac{1}{2}(\lambda-\mu_2),\] observe that $\lambda - \mathbb{R}_{\geq 0}(\lambda -\mu)$ cannot be an extremal ray. So, all we are left is to check and cut down the further redundancies in $\lambda -\mathbb{R}_{\geq0}W_{J}\Pi_{I}$ in $\mathcal{P}(\lambda ,J)$.	    	 \newline\newline  
	    	 \textbf{Step 2.} Fix $\beta \in \Delta_{I,1}$ such that $\height_{J\setminus J_0}(\beta)>0$, we will show that $\lambda-\mathbb{R}_{\geq0}\beta$ is not an extremal ray. Define $I_0 := I \sqcup (J\setminus J_0)$, and note that $\height_{I_0}(\beta)\geq 2$. By the parabolic-PSP (with $I_0\subset\mathcal{I}$), we can write 
	    	 \[\beta =\gamma_1 +\cdots +\gamma_m\text{ for some }\gamma_1,\ldots,\gamma_m\in\Delta_{I_0,1}\text{, where }m= \height_{I_0}(\beta).\]
	    	 Here, we do not need the partial sums of $\sum_{t=1}^{m}\gamma_t$ to be roots. So, we may assume that $\gamma_1 \in \Delta_{I,1}$ and $\gamma_2,\ldots,\gamma_m\in\Delta_{J\setminus J_0,1}\cap\Delta_J^+$. Note that $\gamma_2,\ldots,\gamma_m \in \conv_{\mathbb{R}}(W_{J_0}\Pi_{J\setminus J_0})$ by Lemma \ref{L6.1} (c). So, we get a system of equations \begin{equation*}
	    	 \gamma_t = \sum\limits_{r=1}^{r(t)} \epsilon(r,t)\omega(r,t)\alpha(r,t)
	       	\end{equation*}
	       	$\text{where }2\leq t\leq m,\text{ } r(t)\in \mathbb{N},\text{ } r(t) \text{ depends on }t,(r,t)
	    	 \text{ vary in a finite set }\Omega:=\bigsqcup_{t=2}^{m}[r(t)]\times\{t\}\subset \mathbb{N}\times\mathbb{N}$, $\epsilon(r,t)>0\text{ and }\sum_{r=1}^{r(t)}\epsilon(r,t)=1\text{ for each }2\leq t\leq m,  \omega(r,t)\in W_{J_0},\text{ and }\alpha(r,t)\in \Pi_{J\setminus J_0}\text{ } \forall\text{ }(r,t)\in\Omega$. Pick $\delta>0$ such that $ \delta<\langle \lambda $, $ \omega(r,t)\alpha(r,t)^{\vee}\rangle=\langle\lambda,\text{ }\alpha(r,t)^{\vee}\rangle$ $\forall$ $(r,t)\in\Omega$. Such a positive $\delta$ exists as $\alpha(r,t)\in\Pi_{J\setminus J_0}$ and $\Pi_{J\setminus J_0}$ is finite. Now, note that \[\lambda - \delta \omega(r,t)\alpha(r,t)\in \conv_{\mathbb{R}}W_{J}\lambda\text{ }\forall\text{ }(r,t)\in \Omega,\]
	    	\[\text{ as } \lambda-\delta\omega(r,t)\alpha(r,t)=\frac{\langle\lambda,\text{ }\omega(r,t)\alpha(r,t)^{\vee}\rangle-\delta}{\langle\lambda,\text{ }\omega(r,t)\alpha(r,t)^{\vee}\rangle}\text{ }(\lambda)+\frac{\delta}{\langle\lambda,\text{ }\omega(r,t)\alpha(r,t)^{\vee}\rangle}\text{ }(s_{\omega(r,t)\alpha(r,t)}\lambda).\]
	    	Similarly, note that $\lambda-\delta\epsilon(r,t)\omega(r,t)\alpha(r,t)\in \conv_{\mathbb{R}}W_J\lambda\text{ }\forall\text{ }(r,t)\in\Omega\text{ as }0<\epsilon(r,t)<1\text{ }\forall \text{ }(r,t)\in\Omega$. Notice also that $\lambda-\mathbb{R}_{\geq0}\gamma_1\subset\mathcal{P}(\lambda,J)$. Now, by the non-trivial convex combination
	    	\[\lambda - \frac{\delta}{|\Omega |+1}\beta = \sum\limits_{(r,t)\in \Omega}\frac{1}{|\Omega|+1}\big(\lambda - \delta \epsilon(r,t)\omega(r,t)\alpha(r,t)\big) + \frac{1}{|\Omega|+1}(\lambda -\delta \gamma_1),\]
	    	 observe that $\lambda -\mathbb{R}_{\geq 0}\beta$ cannot be an extremal ray.\newline\newline 
	    	 \textbf{Step 3.} We now show that if $\lambda-\mathbb{R}_{\geq0}\gamma'$ is an extremal ray of $\mathcal{P}(\lambda,J)$ for some $\gamma'\in W_J\Pi_I$, then $\gamma'\in W_{J_0}\Pi_I$. Fix an $i\in I$, and let $\gamma'=w\alpha_i$ for some $w\in W_J$. Assume that $\lambda-\mathbb{R}_{\geq0}\gamma'$ is an extremal ray of $\mathcal{P}(\lambda,J)$. By step 2 we must have $\height_{J\setminus J_0}(\gamma')=0$ or equivalently $\gamma'\in \Delta_{\alpha_i,J_0}$. By Lemma \ref{L6.1} (b) applied to $\conv_{\mathbb{R}}\Delta_{\alpha_i,J_0}$, there exist $u_1,\ldots,u_M\in W_{J_0}$ such that  \[\gamma'=w\alpha_i=h_1u_1\alpha_i+\cdots+h_Mu_M\alpha_i\quad\text{ for some }h_1,\ldots,h_M\in\mathbb{R}_{>0}\text{ summing up to }1, \text{ }M\in\mathbb{N}.\]
	    	 This gives $\alpha_i=h_1w^{-1}u_1\alpha_i+\cdots+h_M w^{-1}u_M\alpha_i$. Now, 
	    	 \[ w^{-1}u_x\alpha_i\in \Delta_{\alpha_i,J}\text{ }\forall\text{ }x\in [M]\implies\supp(w^{-1}u_x\alpha_i)\subset\{i\}\text{ }\forall\text{ }x\in [M]\implies w^{-1}u_x\alpha_i=\alpha_i\text{ }\forall\text{ }x\in [M].\]
	    	 This implies $\gamma'\in W_{J_0}\alpha_i$. \newline\newline
	    	 \textbf{Step 4.} Finally, we show that if $\xi'\in W_{J_0}\Pi_I$, then $\lambda-\mathbb{R}_{\geq0}\xi'$ is an extremal ray of $\mathcal{P}(\lambda,J)$. Fix $i'\in I$, $\xi\in W_{J_0}\alpha_{i'}$ and $r\in\mathbb{R}_{>0}$. Suppose
	    	 \begin{equation*}
	    	\begin{aligned} \lambda-r\xi=\sum\limits_{l'=1}^{p}c_{l'}w_{l'}\lambda-\sum\limits_{l=1}^{q}t_l\xi_l
	    	\end{aligned}\qquad
	    	\begin{aligned}
	    	&\text{for some }c_{l'}, t_l\in\mathbb{R}_{>0}\text{ such that }\sum\limits_{l'=1}^{p}c_{l'}=1,\text{ }w_{l'}\in W_{J},\\ &\xi_l\in\Delta_{I,1}\text{ }\forall\text{ }l'\in [p]\text{ and }l\in [q].
	    	\end{aligned}
	    	\end{equation*}
	    	Note that $\supp_{J\setminus J_0}(\xi)=\emptyset$ and $\supp_I(\xi)=\{i'\}$. Note also by the proof of Lemma \ref{LB.1} that if $w_{l'}\in W_J\setminus W_{J_0}$, then $\height_{J\setminus J_0}(\lambda-w_{l'}\lambda)>0$. In view of the previous line, as $ht_{J\setminus J_0}(\xi)=0$, we must have $w_{l'}\in W_{J_0}$ $\forall$ $l'$. This implies $w_{l'}\lambda=\lambda$ and $\xi_l\in \Delta_{\alpha_{i'},J_0}$ $\forall$ $l'$ and $l$. So, we have $\xi=\sum\limits_{l=1}^{q}\frac{t_l}{r}\xi_l$. By Lemma \ref{L6.1} (c), we may assume that $\xi_l\in W_{J_0}\alpha_{i'}$. Observe then by a similar argument as at the end of the proof of Lemma \ref{L6.1} (b) (which proves the minimality of $W_J\alpha$ in Lemma \ref{L6.1} (b)), that we must have $\xi_l=\xi$ $\forall$ $l$. This proves that $\lambda-\mathbb{R}_{\geq0}\xi$ is an extremal ray of $\mathcal{P}(\lambda,J)$ at $\lambda$.\\
	    	 Hence, the proof of Proposition \ref{P2.11} is complete. 
	    	 \end{proof}


\begin{thebibliography}{16}
	\bibitem{Venkatesh} G. Arunkumar, D. Kus and R. Venkatesh, Root multiplicities for Borcherds algebras and graph coloring, J. Algebra, 499 (2018) 538--569.
	\bibitem{Borel} A. Borel and J. Tits, Groupes r{\'e}ductifs. Inst. Hautes {\'E}tudes Sci. Publ. Math., 27 (1) (1965) 55--150.
	\bibitem{svis} L. Carbone, K.N. Raghavan, B. Ransingh, K. Roy and S. Viswanath, $\pi$-systems of symmetrizable Kac--Moody algebras, arXiv:1902.06413v4.
	\bibitem{Casselman} W. A. Casselman, Geometric rationality of Satake compactifications. In Algebraic groups and Lie groups, volume 9 of Austral. Math. Soc. Lect. Ser., pages 81--103. Cambridge University Press, Cambridge, 1997.
	\bibitem{Cellini} P. Cellini and M. Marietti, Root polytopes and Borel subalgebras, Int. Math. Res. Not. IMRN, (12) (2015) 4392--4420.
	\bibitem{Chari_contm} V. Chari, R.J. Dolbin and T. Ridenour, Ideals in parabolic subalgebras of simple Lie algebras, Contemp. Math. 490 (2009) 47--60.
	\bibitem{Chari_Adv} V. Chari and J. Greenstein, A family of Koszul algebras arising from finite-dimensional representations of simple Lie algebras, Adv. Math. 220 (4) (2009) 1193--1221.
	\bibitem{Chari_JGeom} V. Chari and J. Greenstein, Minimal affinizations as projective objects, J. Geom. Phys. 61 (3) (2011) 594--609.  
	\bibitem{Chari_JPAA} V. Chari, A. Khare and T. Ridenour, Faces of polytopes and Koszul algebras, J. Pure Appl. Algebra 216 (7) (2012) 1611--1625.
	\bibitem{VVD} V. V. Deodhar, On the root system of a coxeter group, Communications in Algebra, 10:6 (1982) 611--630.
	\bibitem{Dhillon_arXiv}
	G. Dhillon and A. Khare, The weights of simple modules in Category $\mathcal{O}$ for Kac--Moody algebras, arXiv:1606.09640v4.
	\bibitem{Khare_Ad}
    G. Dhillon and A. Khare, Faces of highest weight modules and the universal Weyl polyhedron, Adv. Math. 319 (2017) 111--152.
	\bibitem{Hump} 
	J.E. Humphreys, Introduction to Lie algebras and representation theory, Graduate Texts in Mathematics,
	no. 9, Springer-Verlag, Berlin-New York, 1972.
	\bibitem{Hump_BGG}
	J.E. Humphreys, Representations of  semisimple Lie Algebras in the BGG Category $\mathcal{O}$, Graduate Studies in Mathematics, vol. 94, American Mathematical Society, Providence, RI, 2008.
	\bibitem{Kac}
	V. G. Kac, Infinite-dimensional Lie algebras, Cambridge University Press, Cambridge, third edition,
	1990.
	\bibitem{Khare_JA}
	A. Khare, Faces and maximizer subsets of highest weight modules, J. Algebra, 455 (2016) 32--76.
	\bibitem{Khare_Trans}
	A. Khare, Standard parabolic subsets of highest weight modules. Trans. Amer. Math. Soc., 369 (4) (2017) 2363--2394.
	\bibitem{Khare_AR}
	A. Khare and T. Ridenour, Faces of weight polytopes and a generalization of a theorem of Vinberg, Algebras
	and Representation Theory 15 no. 3 (2012) 593–611.
	\bibitem{L_JA} J. Lepowsky, A generalization of the Bernstein-Gelfand-Gelfand resolution, J. Algebra 49 (2) (1977) 496--511.
	\bibitem{Satake} I. Satake, On representations and compactifications of symmetric Riemannian spaces, Ann. of Math. 71 (2) (1960) 77--110.
	\bibitem{Vinberg} E. B. Vinberg, Some commutative subalgebras of a universal enveloping algebra, Izv. Akad. Nauk SSSR Ser. Mat., 54 (1) (1990) 3--25, 221.
\end{thebibliography}
\end{document}